\documentclass[11pt,twoside,english,reqno,a4paper]{amsart}
\usepackage{a4wide}
\usepackage[utf8]{inputenc}
\usepackage{amsmath,amsthm} 
\usepackage{amssymb}
\usepackage{mathrsfs}
\usepackage{lmodern}
\usepackage[T1]{fontenc}        
\usepackage[english]{babel}
\usepackage[numbers]{natbib}    
\usepackage{graphicx}
\usepackage{url}
\usepackage{tikz-cd}
\usepackage[arrow, matrix, curve]{xy}
\usepackage{faktor}
\usepackage{multirow}
\usepackage{mathtools,amsmath}
\usepackage{hyperref}
\usepackage{color}
\usepackage{epic}
\usepackage{pict2e}
\usepackage{paralist}
\usepackage{enumerate}
\usepackage{enumitem}
\usepackage{float}
\usepackage{mathtools}
\usepackage{blindtext}
\usepackage{framed}
\usepackage{advdate}
\usepackage{relsize}

\usepackage[hang,flushmargin]{footmisc}

\usepackage{currvita}

\usepackage{chngcntr}
\counterwithin{figure}{section}

\DeclareMathOperator{\SL}{SL}

\DeclareMathOperator{\GL}{GL}
\DeclareMathOperator{\Orm}{O}

\DeclareMathOperator{\tr}{tr}
\DeclareMathOperator{\proj}{proj}
\DeclareMathOperator{\rest}{rest}
\DeclareMathOperator{\Hom}{{Hom}}
\DeclareMathOperator{\Ind}{Ind}
\DeclareMathOperator{\diag}{diag}
\let\Re\relax
\DeclareMathOperator{\Re}{Re}

\newcommand{\Ocal}{\mathcal{O}}
\newcommand{\zfrak}{{\mathfrak{z}}}

\newcommand{\Nbar}{\bar{N}}
\newcommand{\Ical}{\mathcal{I}}
\newcommand{\Jcal}{\mathcal{J}}
\newcommand{\R}{\mathbb{R}}
\newcommand{\C}{\mathbb{C}}
\newcommand{\g}{\mathfrak{g}}

\newcommand{\Z}{\mathbb{Z}}
\newcommand{\N}{\mathbb{N}}
\newcommand{\glf}{\mathfrak{gl}}

\newcommand{\Hcal}{\mathcal{H}}

\newcommand{\Pcal}{\mathcal{P}}

\newcommand{\HC}{{\mathrm{HC}}}
\newcommand{\RP}{\R{\rm P}}
\newcommand{\alg}{{\rm alg}}
\renewcommand{\ss}{{{\rm ss}}}
\newcommand{\subz}{{{\rm z}}}

\DeclarePairedDelimiter\abs{\lvert}{\rvert}

\usepackage{tikz}
\usepackage{array}
\newcommand\diagonal[4]{%
  \multicolumn{1}{p{#2}|}{\hskip-\tabcolsep
  $\vcenter{\begin{tikzpicture}[baseline=0,anchor=south west,inner sep=#1]
  \path[use as bounding box] (0,0) rectangle (#2+2\tabcolsep,\baselineskip);
  \node[minimum width={#2+2\tabcolsep},minimum height=\baselineskip+\extrarowheight] (box) {};
  \draw (box.north west) -- (box.south east);
  \node[anchor=south west] at (box.south west) {#3};
  \node[anchor=north east] at (box.north east) {#4};
 \end{tikzpicture}}$\hskip-\tabcolsep}}
\makeatletter

\theoremstyle{plain}
\newtheorem{theorem}{Theorem}[section]
\newtheorem{lemma}[theorem]{Lemma}
\newtheorem{corollary}[theorem]{Corollary}
\newtheorem{prop}[theorem]{Proposition}
\newtheorem{theoremalph}{Theorem}

\theoremstyle{definition}

\newtheorem{remark}[theorem]{Remark}

\numberwithin{equation}{section}

\title{Symmetry breaking operators for line bundles over real projective spaces}

\author{Jan Frahm}
\author{Clemens Weiske}
\thanks{The second author was supported by the DFG project 325558309}

\address{Department Mathematik, FAU Erlangen--N\"{u}rnberg, Cauerstr. 11, 91058 Erlangen, Germany}
\email{frahm@math.fau.de}
\address{Department Mathematik, FAU Erlangen--N\"{u}rnberg, Cauerstr. 11, 91058 Erlangen, Germany}
\email{weiske@math.fau.de}
\begin{document}

\begin{abstract}
The space of smooth sections of an equivariant line bundle over the real projective space $\RP^n$ forms a natural representation of the group $\GL(n+1,\R)$. We explicitly construct and classify all intertwining operators between such representations of $\GL(n+1,\R)$ and its subgroup $\GL(n,\R)$, intertwining for the subgroup. Intertwining operators of this form are called symmetry breaking operators, and they describe the occurrence of a representation of $\GL(n,\R)$ inside the restriction of a representation of $\GL(n+1,\R)$. In this way, our results contribute to the study of branching problems for the real reductive pair $(\GL(n+1,\R),\GL(n,\R))$.\\
The analogous classification is carried out for intertwining operators between algebraic sections of line bundles, where the Lie group action of $\GL(n,\R)$ is replaced by the action of its Lie algebra $\mathfrak{gl}(n,\R)$, and it turns out that all intertwining operators arise as restrictions of operators between smooth sections.
\end{abstract}
\maketitle

\section*{Introduction}

The real projective space $\RP^n$ is in a natural way a homogeneous spaces for the general linear group $\GL(n+1,\R)$. This induces natural representations of $\GL(n+1,\R)$ on functions, or more generally sections of equivariant vector bundles over $\RP^n$. The equivariant line bundles $\mathcal{L}$ over $\RP^n$ are essentially parametrized by two complex numbers, and the corresponding linear actions of $\GL(n+1,\R)$ on smooth sections $C^\infty(\RP^n,\mathcal{L})$ form a two-parameter family of generically irreducible smooth representations.

One natural question to ask in this context is how two such representations are related, or more precisely: for which equivariant line bundles $\mathcal{L},\mathcal{L}'$ over $\RP^n$ does there exist a non-trivial (continuous) linear operator
$$ A:C^\infty(\RP^n,\mathcal{L})\to C^\infty(\RP^n,\mathcal{L}') $$
which respects the action of $\GL(n+1,\R)$? For $n\geq2$, it turns out that apart from the trivial case $\mathcal{L}\cong\mathcal{L'}$, this only occurs for $\mathcal{L}=\Lambda^n(\RP^n)$ the top form bundle and $\mathcal{L}'=\RP^n\times\C$ the trivial bundle, in which case the integration of differential forms produces a non-trivial linear map with one-dimensional image, the constant functions on $\RP^n$ (see e.g.~\cite{HL99,MS14}).

A much richer analysis arises if one allows line bundles over projective spaces of different dimensions. For this we consider the subgroup $\GL(n,\R)$ of $\GL(n+1,\R)$ stabilizing the last standard basis vector. Given equivariant line bundles $\mathcal{L}$ over $\RP^n$ and $\mathcal{L}'$ over $\RP^{n-1}$ one may ask for (continuous) linear operators
$$ A:C^\infty(\RP^n,\mathcal{L})\to C^\infty(\RP^{n-1},\mathcal{L}') $$
which respect the action of the smaller group $\GL(n,\R)$. In this paper we classify all such operators and give explicit formulas for them in terms of their distribution kernels. We further study the corresponding algebraic question, replacing smooth sections by regular sections in the sense of algebraic geometry, and replacing the group action by the Lie algebra action of $\mathfrak{gl}(n,\R)$.

In the language of representation theory of real reductive groups, the action of $\GL(n+1,\R)$ on $C^\infty(\RP^n,\mathcal{L})$ defines a generalized principal series representation, since it can be constructed as a representation induced from a (maximal) parabolic subgroup. Such families of representations play a major role in representation theory, and they are one of the main ingredients for the classification of irreducible smooth representations. Intertwining operators between principal series of a real reductive group $G$ and a subgroup $G'$ have only recently attracted more attention in the context of branching problems as advocated by Kobayashi~\cite{kobayashi_2015} in his $\mathcal{ABC}$-program. Our results can be viewed as a contribution to this program for the real reductive pair $(\GL(n+1,\R),\GL(n,\R))$, similar to the recent contributions \cite{KKP16,kobayashi_speh_2015,KS17,MO17} for the pair $(G,G')=({\rm O}(n+1,1),{\rm O}(n,1))$ and \cite{KL17} for the pair $(G,G')=({\rm O}(p+1,q+1),{\rm O}(p,q+1))$.

Let us now describe our results in more detail.

\subsection*{Symmetry breaking operators}
Let $n\geq2$. The general linear group $\GL(n+1,\R)$ acts transitively on the real projective space $\RP^n$ by $g\cdot[x]=[gx]$, where we write $[x]=\R x\in\RP^n$ for the line through $x\in\R^{n+1}\setminus\{0\}$. This action induces a family of representations $\pi_\lambda$ of $\GL(n+1,\R)$ on the space $C^\infty(\RP^{{n}})$ of smooth functions on $\RP^{{n}}$, parametrized by $\lambda=(\lambda_1,\lambda_2) \in \C^2$:
\begin{equation}
\pi_\lambda(g) \varphi([x]):=\left(  \frac{\abs{g^{-1}x}}{\abs{x}} \right)^{-\lambda_1-\rho_1} \abs{\det g^{-1}}^{-\lambda_2-\rho_2}  \varphi([g^{-1}x])\label{eq:DefRepPilambda}
\end{equation}
for $\varphi \in C^\infty(\RP^{{n}})$, $x\in \R^{n+1}$ and $g \in \GL({n+1},\R)$. Here $(\rho_1,\rho_2)=(\frac{n+1}{2},-\frac{1}{2})$ is chosen such that $\pi_\lambda$ extends to an irreducible unitary representation on $L^2(\RP^n)$ for $\lambda\in(i\R)^2$.

A more natural construction of the representations $\pi_\lambda$ is in terms of line bundles. There exists a family of equivariant line bundles $\mathcal{L}_\lambda\to\RP^n$, and $\GL(n+1,\R)$ acts naturally on the space $C^\infty(\RP^n,\mathcal{L}_\lambda)$ of smooth sections. Trivializing the line bundle identifies $C^\infty(\RP^n,\mathcal{L}_\lambda)\cong C^\infty(\RP^n)$, and under this identification the action $\pi_\lambda$ on $C^\infty(\RP^n)$ corresponds to the natural action of $\GL(n+1,\R)$ on $C^\infty(\RP^n,\mathcal{L}_\lambda)$

In the same way we define a family of representations $\tau_\nu$ of the group $\GL({n},\R)$ on the space $C^\infty(\RP^{n-1})$, parametrized by $\nu=(\nu_1,\nu_2)\in\C^2$:
\begin{equation*}
\tau_\nu (h) \psi([y]):=\left(  \frac{\abs{h^{-1}y}}{\abs{y}} \right)^{-\nu_1-\rho'_1} \abs{\det h^{-1}}^{-\nu_2-\rho'_2}  \psi([h^{-1}y])
\end{equation*}
for $\psi\in C^\infty(\RP^{n-1})$, $y\in \R^n$ and $h \in \GL({n},\R)$, where $(\rho'_1,\rho'_2)=(\frac{n}{2},-\frac{1}{2})\in \C^2$. We view $\GL(n,\R)$ as the subgroup of $\GL(n+1,\R)$ fixing the last standard basis vector $e_{n+1}\in\R^{n+1}$.

Consider the space
\begin{multline*}
H^\infty(\lambda,\nu):=\{ A:C^\infty(\RP^{{n}}) \to C^\infty(\RP^{{n-1}})\mbox{ continuous and linear}:\\
A\circ \pi_\lambda(g) = \tau_\nu(g)\circ A\;  \forall\, g \in \GL({n},\R) \}
\end{multline*}
of operators respecting the actions $\pi_\lambda$ and $\tau_\nu$. Such operators are also called \textit{symmetry breaking operators}, a term coined by Kobayashi~\cite{kobayashi_2015}, and they describe the occurrence of the representation $\tau_\nu$ inside the restriction of $\pi_\lambda$ to the subgroup $\GL(n,\R)\subseteq\GL(n+1,\R)$. It is therefore natural to ask the following questions:
\begin{enumerate}[label=\textbf{Q\arabic{*}}:, ref=\textbf{Q\arabic{*}}]
\item \label{Q1} 
For which parameters $\lambda,\nu\in\C^2$ is $H^\infty(\lambda,\nu)\neq\{0\}$?
\item \label{Q2} 
If $H^\infty(\lambda,\nu)\neq\{0\}$, what is its dimension?
\item \label{Q3}
What are explicit bases of the non-trivial spaces?
\end{enumerate}

Our first main result gives a complete answer to the first two questions \ref{Q1} and \ref{Q2} (see Theorem~\ref{thm:charts} and Proposition~\ref{prop:classification_distr}):
\begin{theoremalph}[Multiplicities]
\label{theorem:A}
If $\dim H^\infty(\lambda,\nu)\neq 0$ then either
$$ \lambda_2+\rho_2=\nu_2+\rho'_2 \qquad \mbox{or} \qquad \lambda_2-\rho_2-\nu_2-\rho'_2=\nu_1+\rho'_1. $$
For $\lambda_2+\rho_2=\nu_2+\rho'_2$ the dimension is given by
\begin{equation*}
\dim H^\infty(\lambda,\nu)=
\begin{cases}
1 & \text{for }  (\lambda_1,\nu_1)\notin L,\\
2 & \text{for }  (\lambda_1,\nu_1)\in L,\\
\end{cases}
\end{equation*}
where $L=\{ (-\rho_1-2i,-\rho'_1-2j) \in \C^2: i,j \in \Z_{\geq 0}, j\leq i \}$, and for $\lambda_2-\rho_2-\nu_2-\rho'_2=\nu_1+\rho'_1$ it is given by
\begin{equation*}
\dim H^\infty(\lambda,\nu)=
\begin{cases}
1 & \text{for } n=2, \\
1 & \text{for } n\geq 3 \text{ and } \nu_1=-\rho'_1,\\
0 & \text{otherwise.}
\end{cases}
\end{equation*}
\end{theoremalph}

To describe an explicit basis of $H^\infty(\lambda,\nu)$ we define families of operators $C^\infty(\RP^{{n}}) \to C^\infty(\RP^{{n-1}})$ by
\begin{align*}
&A_{\lambda,\nu}\varphi([y])=\int_{S^{1}} \frac{\abs{\omega_1}^{\nu_1+\rho'_1-1}\abs{\omega_2}^{\lambda_1+\rho_1-\nu_1-\rho'_1-1}}{2\Gamma(\frac{\nu_1+\rho'_1}{2})\Gamma(\frac{\lambda_1+\rho_1-\nu_1-\rho'_1}{2})}\varphi\left(\left[\omega_1\frac{y}{\abs{y}}:\omega_2\right]\right)d(\omega_1,\omega_2), \\
&B_{\lambda,\nu}\varphi([y])=\frac{1}{2{\Gamma(\frac{\lambda_1-\rho_1+1}{2})}}\int_{ S^{{n}}}{\abs{\omega_{n+1}}^{\lambda_1-\rho_1}}\varphi([\omega])d\omega, \\
\intertext{and for $n=2$ additionally}
&C_{\lambda,\nu}\varphi([y])=\int_{ S^{{2}}}\frac{\left|\frac{\omega_2y_1}{\abs{y}}-\frac{\omega_1 y_2}{\abs{y}}\right|^{-\nu_1-\rho'_1}\abs{\omega_{3}}^{\lambda_1+\rho_1+\nu_1-\rho'_1-1}}{2\Gamma(\frac{-\nu_1-\rho'_1+1}{2}){\Gamma(\frac{\lambda_1+\rho_1+\nu_1-\rho'_1}{2})}}\varphi([\omega])d\omega,\\
\intertext{where $\varphi\in C^\infty(\RP^n)$ and $[y]\in\RP^{n-1}$. Note that $B_{\lambda,\nu}\varphi$ is always a constant function. The normalizing gamma factors are chosen so that $A_{\lambda,\nu}$, $B_{\lambda,\nu}$ and $C_{\lambda,\nu}$ depend holomorphically on $\lambda,\nu\in\C^2$. It turns out that $B_{\lambda,\nu},C_{\lambda,\nu}\neq0$ for all $\lambda,\nu\in\C^2$, but $A_{\lambda,\nu}=0$ if and only if $(\lambda_1,\nu_1)\in L$. Therefore, we define two additional families of operators. First, for $\nu_1=-\rho'_1-2j$, $j \in \Z_{\geq 0}$, we let}
&A^{(1)}_{\lambda,\nu}\varphi([y])=\abs{y}^{\nu_1+\rho_1'} \frac{d^{2j}}{dt^{2j}}\biggr|_{t=0}\abs{(ty,1)}^{-\lambda_1-\rho_1}\varphi([ty:1]), \\
\intertext{and for $\lambda_1+\rho_1-\nu_1-\rho'_1=-2N$, $N\in\Z_{\geq 0}$, we define}
&A^{(2)}_{\lambda,\nu}\varphi([y])=\abs{y}^{\nu_1+\rho_1'}\frac{d^{2N}}{dt^{2N}}\biggr|_{t=0}\abs{(y,t)}^{-\lambda_1-\rho_1}\varphi([y:t]).
\end{align*}
For instance, if $\lambda_1+\rho_1-\nu_1-\rho_1'=0$, the operator $A_{\lambda,\nu}^{(2)}$ is simply restricting a function on $\RP^n$ to $\RP^{n-1}\subseteq\RP^n$.

In Theorem~\ref{theorem:compact_operators} we give a complete answer to \ref{Q3}:
\begin{theoremalph}[Symmetry breaking operators]
\label{theorem:B}
For $\lambda_2+\rho_2=\nu_2+\rho'_2$ we have
\begin{equation*}
H^\infty(\lambda,\nu)=
\begin{cases}
\C A_{\lambda,\nu} & \text{for } (\lambda_1,\nu_1)\notin L,\\
\C A^{(1)}_{\lambda,\nu} \oplus \C A^{(2)}_{\lambda,\nu} & \text{for } (\lambda_1,\nu_1)\in L.
\end{cases}
\end{equation*}
and for $\lambda_2-\rho_2-\nu_2-\rho'_2=\nu_1+\rho'_1$:
\begin{equation*}
H^\infty(\lambda,\nu) 
\begin{cases}
\C C_{\lambda,\nu} & \text{for } n=2, \\
\C B_{\lambda,\nu} & \text{for } n\geq 3 \text{ and } \nu_1=-\rho'_1.
\end{cases}
\end{equation*}
\end{theoremalph}

Although $A_{\lambda,\nu}=0$ for $(\lambda_1,\nu_1)\in L$, the families $A^{(1)}_{\lambda,\nu}$ and $A^{(2)}_{\lambda,\nu}$ can be obtained as residues of $A_{\lambda,\nu}$ along certain one-dimensional affine complex subspaces through $(\lambda_1,\nu_1)\in\C^2$ (see Corollary~\ref{corollary:residues_distr}):
\begin{theoremalph}[Residue Formulas]
The following residue formulas hold:
\label{theorem:C}
\begin{align*}
A_{\lambda,\nu} &=(-1)^j\frac{ j!}{(2j)!\Gamma(\frac{\lambda_1+\rho_1-\nu_1-\rho'_1}{2})}A^{(1)}_{\lambda,\nu} && \text{for }\nu_1=-\rho'_1-2j, \\
A_{\lambda,\nu} &=(-1)^N\frac{N!}{(2N)!\Gamma(\frac{\nu_1+\rho'_1}{2})}A^{(2)}_{\lambda,\nu} && \text{for }\lambda_1+\rho_1-\nu_1-\rho'_1=2N.\\
\intertext{For $n=2$ we further have the following relations:}
A_{\lambda,\nu} &=\frac{1}{\sqrt{\pi}}C_{\lambda,\nu} && \text{for } \nu_1+\rho'_1=-2 \rho_2, \\
B_{\lambda,\nu} &= \sqrt{\pi}C_{\lambda,\nu}  && \text{for } \nu_1+\rho'_1=0.
\end{align*}
\end{theoremalph}

\subsection*{Algebraic symmetry breaking operators}
The problem of constructing and classifying symmetry breaking operators can also be studied in an algebraic framework. The representation $\pi_\lambda$ of $\GL(n+1,\R)$ induces by differentiation an action of the Lie algebra $\glf({n+1},\R)$ on the space of regular functions $\C[\RP^{{n}}]$ on $\RP^n$ which we also denote by $\pi_\lambda$. Similarly $\tau_\nu$ induces an action of $\glf({n},\R)$ on $\C[\RP^{{n-1}}]$ and we consider the space
\begin{equation*}
H^\alg(\lambda,\nu):=\{ T:\C[\RP^{{n}}] \to \C[\RP^{{n-1}}]\mbox{ linear}: T\circ \pi_\lambda(X) = \tau_\nu(X)\circ T\;  \forall\, X \in \glf({n},\R)
\}
\end{equation*}
to which we refer as \textit{algebraic symmetry breaking operators}. Every symmetry breaking operator induces an algebraic symmetry breaking operator by restriction, and since $\C[\RP^n]$ is dense in $C^\infty(\RP^n)$ the restriction map
\begin{equation*}
H^\infty(\lambda,\nu)\hookrightarrow H^\alg(\lambda,\nu), \quad A\mapsto A|_{\C[\RP^n]}
\end{equation*}
is injective. In particular we have the following basic inequality between multiplicities in the smooth and in the algebraic setting:
\begin{equation}
\label{eq:mult_smooth_leq_mult_alg}
\dim H^\infty(\lambda,\nu)\leq \dim H^\alg(\lambda,\nu).
\end{equation}
In Section~\ref{sec:concrete_algebraic} we study algebraic symmetry breaking operators in detail and show in particular that the multiplicities are actually equal, or in other words, algebraic symmetry breaking operators are automatically continuous (see Theorem~\ref{theorem:mult_HC}):
\begin{theoremalph}
\label{theorem:D}
The restriction map $H^\infty(\lambda,\nu)\to H^\alg(\lambda,\nu)$ is a bijection for all $\lambda,\nu\in\C^2$.
\end{theoremalph}

\subsection*{Representation theoretic context}
The representations $\pi_\lambda$ and $\tau_\nu$ belong to the so-called generalized principal series of the reductive Lie groups $G=\GL({n+1},\R)$ and $G'=\GL({n},\R)$. Principal series representations are induced from a parabolic subgroup, which is in this case of $\pi_\lambda$ a maximal parabolic subgroup $P\subseteq G$ such that $G/P\cong\RP^n$. In representation theoretic language, the space of symmetry breaking operators can be written as
\begin{equation*}
H^\infty(\lambda,\nu) = \Hom_{G'}(\pi_\lambda|_{G'},\tau_\nu),
\end{equation*}
the space of intertwining operators between the restriction $\pi_\lambda|_{G'}$ of the representation $\pi_\lambda$ to the subgroup $G'$ and the representation $\tau_\nu$. Its dimension is called multiplicity of $\tau_\nu$ in $\pi_\lambda$, and it describes how often $\tau_\nu$ occurs as a quotient of $\pi_\lambda|_{G'}$. In this way, the study of symmetry breaking operators can also be seen as a contribution to \textit{branching problems}, i.e. the decomposition of restricted representations. In his $\mathcal{ABC}$-program for branching problems, Kobayashi~\cite{kobayashi_2015} initiated the systematic study of symmetry breaking operators, and our natural questions \ref{Q1}, \ref{Q2} and \ref{Q3} can be viewed as a contribution to part $\mathcal{B}$ and $\mathcal{C}$ of his program.

Also algebraic symmetry breaking operators have a representation theoretic interpretation. The orthogonal group $K=\Orm(n+1)\subseteq G$ is a maximal compact subgroup of $G$ and the restriction of $\pi_\lambda$ to $K$ is the natural representation of $\Orm(n+1)$ on $C^\infty(\RP^n)$. In particular, it is independent of $\lambda$ and the regular functions on $\RP^n$ are precisely those which generate a finite-dimensional representation of $K$. Therefore, $\C[\RP^n]$ carries both an action of $K$ and an action of the Lie algebra $\g=\mathfrak{gl}(n+1,\R)$ by differentiation, and these actions are compatible. This defines on $\C[\RP^n]$ the structure of a $(\g,K)$-module, also called the \textit{underlying Harish-Chandra module} of $\pi_\lambda$ and denoted by $(\pi_\lambda)_\HC$.

In the same way, the representation $\tau_\nu$ of $G'=\GL(n,\R)$ on $C^\infty(\RP^{n-1})$ induces the structure of a $(\mathfrak{g}',K')$-module on $\C[\RP^{n-1}]$, where $\g'=\mathfrak{gl}(n,\R)$ and $K'=\Orm(n)\subseteq G'$. Denoting this Harish-Chandra module by $(\tau_\nu)_\HC$, the space of algebraic symmetry breaking operators can be written as
\begin{equation*}
H^\alg(\lambda,\nu)=\Hom_{(\g',K')}((\pi_\lambda)_{\HC}|_{(\g',K')},(\tau_\nu)_{\HC}),
\end{equation*}
the space of intertwining operators for the actions of $\g'$ and $K'$. In this language, Theorem~\ref{theorem:D} confirms for our particular representations a conjecture by Kobayashi, which states that for arbitrary real reductive groups $G$ and $G'$ the restriction
\begin{equation*}
 \Hom_{G'}(\pi|_{G'},\tau) \to \Hom_{(\g',K')}(\pi_{\HC}|_{(\g',K')},\tau_{\HC})
\end{equation*}
is an isomorphism for all admissible smooth representations $\pi$ of $G$ and $\tau$ of $G'$ if the homogeneous space $(G\times G')/\diag(G')$ is real spherical (see \cite[Remark 10.2~(4)]{kobayashi_2014}).

\subsection*{Multiplicities and symmetry breaking operators between irreducible subquotients}
Generically, the representations $\pi_\lambda$ and $\tau_\nu$ are irreducible, but for certain singular parameters $\lambda,\nu\in\C^2$ they have composition series of length $2$ (see Lemma~\ref{Lemma:submodules}). The same questions \ref{Q1}, \ref{Q2} and \ref{Q3} can be asked for the corresponding subrepresentations and quotients of $\pi_\lambda$ and $\tau_\nu$ at reducibility points. In Theorem~\ref{theorem:composition_series_tables} we determine multiplicities and explicit symmetry breaking operators between all subrepresentations and quotients, and show in particular that also in this case every algebraic symmetry breaking operator is automatically continuous.

\subsection*{Structure of this article}
In Section~\ref{sec:general_smooth} we introduce the necessary notation and briefly describe how the main results of this paper are proven. The construction and classification of symmetry breaking operators between the smooth representations $\pi_\lambda$ and $\tau_\nu$ is achieved in Section~\ref{sec:concrete_smooth} where Theorems~\ref{theorem:A}, \ref{theorem:B} and \ref{theorem:C} are shown. Here we also obtain a description of all symmetry breaking operators in the non-compact picture (see Theorem~\ref{thm:SBOsNonCptPicture}) as well as explicit formulas for the action on each $K$-type (see Lemma~\ref{lemma:spectral_functions}). Section~\ref{sec:SBOsHCmodules} is concerned with a general combinatorial characterization of algebraic symmetry breaking operators, generalizing some of the results in~\cite{MolOrs} from semisimple groups to reductive groups, and also treating operators between subquotients of principal series representations. Finally, in Section~\ref{sec:AlgSBOsGL} we apply this combinatorial characterization to the case $(G,G')=(\GL(n+1,\R),\GL(n,\R))$ to prove Theorem~\ref{theorem:D} and compute the multiplicities between all possible subquotients of $\pi_\lambda$ and $\tau_\nu$ (see Theorem~\ref{theorem:composition_series_tables}). In the appendix we collect some elementary facts about homogeneous distributions on $\R^n$, Gegenbauer polynomials and spherical harmonics.

The results in Sections~\ref{sec:SBOsHCmodules} and \ref{sec:AlgSBOsGL} were obtained in the second author's Master thesis~\cite{Weiske}.

\section{Symmetry breaking operators between generalized principal series}
\label{sec:general_smooth}

Let us explain a general setting to study symmetry breaking operators into which the real reductive pair $(G,G')=(\GL(n+1,\R),\GL(n,\R))$ fits.

\subsection{Generalized principal series representations}\label{sec:GenPS}

Let $G$ be a real reductive Lie group and $P\subseteq G$ a parabolic subgroup with Langlands decomposition $P=MAN$. We write $\g$, $\mathfrak{m}$, $\mathfrak{a}$ and $\mathfrak{n}$ for the Lie algebras of $G$, $M$, $A$ and $N$.

For any finite-dimensional representation $(\xi,V_\xi)$ of $M$ and any $\lambda\in\mathfrak{a}_\C^*$, the tensor product $\xi\otimes e^\lambda\otimes1$ is a representation of $P=MAN$ on $V_\xi$ and we form the induced representation (smooth normalized parabolic induction)
$$ \pi_{\xi,\lambda} = \Ind_P^G(\xi\otimes e^\lambda\otimes1) $$
which is called a \textit{generalized principal series representation}. It is given by the left-regular action on the space
\begin{equation*}
I_{\xi,\lambda}(G)= \{ F \in C^\infty (G,V_\xi): F(gman)=a^{-\rho-\lambda}\xi(m)^{-1}F(g)\,\forall\,g \in G, man \in P \},
\end{equation*}
where $\rho=\frac{1}{2}\tr\mathrm{ad}|_{\mathfrak{n}}\in\mathfrak{a}^*$. A more geometric realization of $\pi_{\xi,\lambda}$ is as sections of a certain vector bundle. For this let $V_{\xi,\lambda}$ denote the representation $\xi\otimes e^{\lambda+\rho}\otimes1$ on $V_\xi$ and form the equivariant vector bundle
$$ \mathcal{V}_{\xi,\lambda} = G\times_PV_{\xi,\lambda} \to G/P, $$
then $I_{\xi,\lambda}(G)\cong C^\infty(G/P,\mathcal{V}_{\xi,\lambda})$ as $G$-representations.

\subsubsection{The compact picture}

We choose a Cartan involution $\theta$ that leaves $MA$ invariant, then the fixed point group $K=G^\theta$ of $\theta$ is a maximal compact subgroup of $G$ and $K\cap P=K\cap M$ is a maximal compact subgroup of $M$. On the Lie algebra level, the involution $\theta$ induces a the Cartan decomposition $\g=\mathfrak{k}\oplus\mathfrak{s}$ with $\mathfrak{k}$ the $(+1)$-eigenspace and $\mathfrak{s}$ the $(-1)$-eigenspace of $\theta$. The subspace $\mathfrak{k}$ is precisely the Lie algebra of $K$.

By the  Iwasawa decomposition we have $G=KP$, so that restricting to $K$ defines an isomorphism from $I_{\xi,\lambda}(G)$ to
\begin{equation*}
I_\xi(K) = \{F \in C^\infty (K,V_\xi): F(km)=\xi(m)^{-1}F(k)\,\forall\,k \in K, m \in K\cap M \}.
\end{equation*}
This space is independent of $\lambda$ and by abuse of notation we also denote by $\pi_{\xi,\lambda}$ the representation of $G$ on $I_\xi(K)$ which makes the restriction map $I_{\xi,\lambda}(G)\to I_\xi(K)$ equivariant. The realization $(\pi_{\xi,\lambda},I_\xi(K))$ is called the \textit{compact picture}.

\subsubsection{The non-compact picture}

Let $\Nbar:=\theta N$ be the nilradical of the parabolic subgroup opposite to $P$. Yet another realization of $\pi_{\xi,\lambda}$ is obtained by restricting functions in $I_{\xi,\lambda}(G)$ to $\bar{N}$. Since $\Nbar P$ is an open dense subspace of $G$, restriction to $\bar{N}$ defines an isomorphism $I_{\xi,\lambda}(G)\to I_{\xi,\lambda}(\Nbar)$ onto a subspace $I_{\xi,\lambda}(\Nbar)\subseteq C^\infty(\Nbar)$ which contains the Schwartz space $\mathcal{S}(\Nbar)$. This realization is called the \textit{non-compact picture} of $\pi_{\xi,\lambda}$.

\subsection{Symmetry breaking operators}

Let $G'\subseteq G$ be a reductive subgroup, $P'=M'A'N'\subseteq G'$ a parabolic subgroup and $\g'$, $\mathfrak{m}'$, $\mathfrak{a}'$ and $\mathfrak{n}'$ the corresponding Lie algebras. As above we form generalized principal series representations
$$ \tau_{\eta,\nu} = \Ind_{P'}^{G'}(\eta\otimes e^\nu\otimes1) $$
associated to finite-dimensional representations $(\eta,W_\eta)$ of $M'$ and $\nu\in(\mathfrak{a}')_\C^*$, and realize them in three different ways, on $J_{\eta,\nu}(G')$, $J_{\eta}(K')$ and $J_{\eta,\nu}(\Nbar')$. For the latter two realizations we assume that $G'$ and $M'A'$ are $\theta$-stable, then $K'=K\cap G'$ is a maximal compact subgroup of $G'$ and $\Nbar'=\theta N'$ the nilradical opposite to $N'$. For later purpose we write $\g'=\mathfrak{k}'\oplus\mathfrak{s}'$ for the corresponding Cartan decomposition of $\g'$.

Consider the space
$$ \Hom_{G'}(\pi_{\xi,\lambda}|_{G'},\tau_{\eta,\nu}) $$
of continuous intertwining operators between the restricted representation $\pi_{\xi,\lambda}|_{G'}$ and $\tau_{\eta,\nu}$. Such operators are also called \textit{symmetry breaking operators} by Kobayashi~\cite{kobayashi_2015}, since they describe the occurrence of $\tau_{\eta,\nu}$ as a quotient of $\pi_{\xi,\lambda}|_{G'}$. One way to study symmetry breaking operators is in terms of their distribution kernels. For this it is most convenient to realize both representations as section of equivariant vector bundles.

Let $\rho' = \frac{1}{2}\tr\mathrm{ad}|_{\mathfrak{n}'}$ and denote by $W_{\eta,\nu}$ the representation $\eta\otimes e^{\nu+\rho'}\otimes1$ of $P'$ on $W_\eta$. Then $J_{\eta,\nu}(G')\cong C^\infty(G'/P',\mathcal{W}_{\eta,\nu})$, where $\mathcal{W}_{\eta,\nu}=G'\times_{P'}W_{\eta,\nu}\to G'/P'$. In this realization
$$ \Hom_{G'}(\pi_{\xi,\lambda}|_{G'},\tau_{\eta,\nu}) = \Hom_{G'}(C^\infty(G/P,\mathcal{V}_{\xi,\lambda}), C^\infty(G'/P',\mathcal{W}_{\eta,\nu})). $$
By the Schwartz Kernel Theorem, every continuous linear operator between $C^\infty(G/P,\mathcal{V}_{\xi,\lambda})$ and $C^\infty(G'/P',\mathcal{W}_{\eta,\nu})$ is given by a distribution section of the vector bundle $\mathcal{V}_{\xi^*,-\lambda}\otimes\mathcal{W}_{\eta,\nu}$ over $G/P\times G'/P'$, where $\xi^*$ denotes the representation contragredient to $\xi$. The $G'$-equivariance of the operator then translates to an invariance property for the distribution. Since $G'$ acts transitively on $G'/P'$, the distribution section can be viewed as a section on $G/P$ with certain invariance properties. The precise statement is:

\begin{theorem}[{\cite[Proposition 3.2]{kobayashi_speh_2015}}]
\label{theorem:SBO-distr}
There is a natural bijection
\begin{equation*}
\Hom_{G'}(\pi_{\xi,\lambda}|_{G'},\tau_{\eta,\nu}) \stackrel{\sim}{\longrightarrow} (\mathcal{D}'(G/P,\mathcal{V}_{\xi^*,-\lambda})\otimes W_{\eta,\nu})^{P'}, \quad A\mapsto u_A.
\end{equation*}
\end{theorem}
As in \cite{kobayashi_speh_2015} we use generalized functions instead of distributions, so that $\mathcal{D}'(G/P,\mathcal{V}_{\xi^*,-\lambda})$ is the continuous dual space of $C^\infty(G/P,\mathcal{V}_{\xi,\lambda})$. Then a distribution $u_A \in (\mathcal{D}'(G/P,\mathcal{V}_{\xi^*,-\lambda})\otimes W_{\eta,\nu})^{P'}$ defines a symmetry breaking operator $A:C^\infty(G/P,\mathcal{V}_{\xi,\lambda})\to C^\infty(G'/P',\mathcal{W}_{\eta,\nu})$ by
\begin{equation}
\label{eq:operator-kernel}
A\varphi(h) = \langle u_A,\varphi(h\,\cdot\,)\rangle, \qquad \varphi \in C^\infty(G/P,\mathcal{V}_{\xi,\lambda}).
\end{equation}

In Section~\ref{sec:concrete_smooth} we classify all symmetry breaking operators for the pair $(G,G')=(\GL(n+1,\R),\GL(n,\R))$ and certain generalized principal series in terms of their distribution kernels.

\subsection{Algebraic symmetry breaking operators}

The algebraic counterpart of the smooth representation $\pi_{\xi,\lambda}$ is its underlying $(\g,K)$-module $(\pi_{\xi,\lambda})_\HC$. This is a representation of the Lie algebra $\g$ of $G$ and the maximal compact subgroup $K$ such that both actions are compatible. The representation space of the underlying $(\g,K)$-module is the subspace consisting of \textit{$K$-finite vectors}, i.e. vectors which generate a finite-dimensional representation of $K$.

The structure of the underlying Harish-Chandra module is most easily seen in the compact picture, and we let
$$ \Ical := I_{\xi}(K)_{\text{$K$-finite}} $$
be the subspace of $K$-finite vectors in $I_\xi(K)$. Abusing notation, we denote the $\g$-action on $\Ical$ also by $\pi_{\xi,\lambda}$. The restriction $\pi_{\xi,\lambda}|_K$ is independent of $\lambda$, and the $K$-module $\Ical$ is completely reducible. It decomposes as
\begin{equation*}
\Ical=\bigoplus_{\alpha \in\hat{K}}\Ical(\alpha),
\end{equation*}
where $\Ical(\alpha)$ is the $\alpha$-isotypic component of $\Ical$.

Similarly the underlying $(\g',K')$-module $(\tau_{\eta,\nu})_{\HC}$ of $\tau_{\eta,\nu}$ is realized on the space $\Jcal:=J_{\eta}(K')_{\text{$K'$-finite}}$ of $K'$-finite vectors in $J_\eta(K')$, which decomposes into $K'$-isotypic components as
\begin{equation*}
\Jcal = \bigoplus_{\alpha'\in \hat{K}'}\Jcal(\alpha').
\end{equation*}
The algebraic version of the space of symmetry breaking operators is
$$ \Hom_{(\g',K')}((\pi_{\xi,\lambda})_\HC|_{(\g',K')},(\tau_{\eta,\nu})_\HC), $$
the space of all linear maps $\mathcal{I}\to\mathcal{J}$ which intertwine the Lie algebra actions $\pi_{\xi,\lambda}|_{\g'}$ and $\tau_{\eta,\nu}$ of $\g'$ as well as the group actions $\pi_{\xi,\lambda}|_{K'}$ and $\tau_{\eta,\nu}|_{K'}$ of $K'$.

Let us assume the following three multiplicity-freeness properties:
\begin{equation}
\dim \Hom_K(\alpha,\Ical) \leq 1, \quad \dim \Hom_{K'}(\alpha',\Jcal)\leq 1, \quad \dim \Hom_{K'}(\Ical(\alpha),\Jcal(\alpha'))\leq 1,\label{eq:MultFreeProperties}
\end{equation}
for all $\alpha \in \hat{K}$, $\alpha' \in \hat{K'}$. Then all isotypic components $\Ical(\alpha)$ and $\Jcal(\alpha')$ are either trivial or irreducible, and further every $K$-type $\Ical(\alpha)$ decomposes under the action of $K'$ into a multiplicity-free direct sum
\begin{equation*}
\Ical(\alpha)=\bigoplus_{\alpha' \in \hat{K'}}\Ical(\alpha,\alpha')
\end{equation*}
where $\Ical(\alpha,\alpha')$ is the $\alpha'$-isotypic component of $\Ical(\alpha)$.

Now let $T:\mathcal{I}\to\mathcal{J}$ be a symmetry breaking operator, then $T$ is in particular $K'$-intertwining. By Schur's Lemma, each summand $\Ical(\alpha,\alpha')$ is mapped into $\Jcal(\alpha')$. Since $\Jcal(\alpha')$ is either trivial or irreducible, the restriction
$$ T|_{\Ical(\alpha,\alpha')}:\Ical(\alpha,\alpha')\to\Jcal(\alpha') $$
must either be trivial or a $K'$-equivariant isomorphism, and in the latter case it is unique up to scalar multiples. For all pairs $(\alpha,\alpha')\in\hat{K}\times\hat{K}'$ with $\Ical(\alpha,\alpha'),\Jcal(\alpha')\neq\{0\}$ let us fix a $K'$-isomorphism
\begin{equation}
 R_{\alpha,\alpha'}: \Ical(\alpha,\alpha')\rightarrow \Jcal(\alpha'),\label{eq:FixRalphaalpha'}
\end{equation}
then every $K'$-equivariant linear map $T:\Ical\to\Jcal$ is uniquely determined by a sequence of scalars $(t_{\alpha,\alpha})$ such that
$$ T|_{\Ical(\alpha,\alpha')} = t_{\alpha,\alpha'}\cdot R_{\alpha,\alpha'}. $$

In Section~\ref{sec:SBOsHCmodules} we reformulate the additional property that $T$ is intertwining for the Lie algebra representations $\pi_{\xi,\lambda}|_{\g'}$ and $\tau_{\eta,\nu}$ of $\g'$ in terms of a system of linear relations that the sequence $(t_{\alpha,\alpha'})$ has to satisfy. This was previously done in \cite{MolOrs} for the case of semisimple groups, and we generalized these ideas to reductive groups. Further, we systematically extend this characterization to the case of symmetry breaking operators between submodules and quotients of principal series. These general techniques are applied in Section~\ref{sec:AlgSBOsGL} to the pair of groups $(G,G')=(\GL(n+1,\R),\GL(n,\R))$.

\section{Smooth symmetry breaking for of the general linear group}
\label{sec:concrete_smooth}

In this section we classify all symmetry breaking operators between generalized principal series of $G=\GL(n+1,\R)$ and $G'=\GL(n,\R)$ in terms of their distribution kernels. This proves Theorems~\ref{theorem:A}, \ref{theorem:B} and \ref{theorem:C}.

\subsection{Generalized principal series representations}

From now on let $n\geq2$. We fix the necessary notation for the groups we consider. Let $G=\GL({n+1},\R)$ be the group of invertible $(n+1)\times(n+1)$ matrices with real entries. The group $G$ acts transitively on the real projective space
$$ \RP^n = \{[x]=\R x:x\in\R^{n+1}\setminus\{0\}\} $$
by $g\cdot[x]=[gx]$. The stabilizer $P$ of the line $[e_1]$ through the first standard basis vector $e_1=(1,0,\ldots,0)$ is a maximal parabolic subgroup of $G$. Its Langlands decomposition $P=MAN$ is given by $A= \{\text{diag}(x,x',\ldots ,x') : x,x' \in \R_{>0} \}$, 
$N$ the unipotent group of matrices of the form
\[
\left(
\begin{array}{c|c}
  1 & * \cdots * \\ \hline
  0 & \raisebox{-15pt}{{\large\mbox{{$\mathbf{1}_{{n}}$}}}} \\[-3.7ex]
  \vdots & \\[+0.1ex]
  0 &
\end{array}
\right)
\]
and $M$ isomorphic to $ \{ \pm 1 \} \times \SL^\pm({n},\R)$ realized in $\GL({n+1},\R)$ as matrices of the form
\[
\left(
\begin{array}{c|c}
  \pm 1 & 0 \cdots 0 \\ \hline
  0 & \raisebox{-17pt}{{\Large\mbox{{$h$}}}} \\[-3.5ex]
  \vdots & \\[+0.1ex]
  0 &
\end{array}
\right)
\]
with $h \in \SL^\pm({n},\R)$.
In particular, $MA=\GL(1,\R)\times \GL({n},\R)$.

Let $G' \subseteq G$ be the subgroup of matrices whose last row and column is the last standard basis vector $e_{n+1}$ of $\R^{n+1}$, then $G'\cong\GL({n},\R)$. With this identification we consider from now on $\GL({n},\R)$ as a subgroup of $G$. Then $P'=P\cap G'=M'A'N'$ is maximal parabolic in $G'$ with $M'=M\cap G'$ isomorphic to $ \{ \pm 1 \} \times \SL^\pm({n-1},\R)$, $A'=\{ \mathrm{diag}(y,y',\ldots y',1): y,y'\in \R_{>0} \} \subseteq MA$ and $N'=N\cap G'$. The orbit $G'\cdot[e_1]\cong G'/P'$ inside $\RP^n$ is the real projective space $\RP^{n-1}$, viewed as the subspace $\{[x_1:\ldots:x_n:0]:(x_1,\ldots,x_n)\in\R^n\setminus\{0\}\}$.

Let $K \subseteq G$ and $K' \subseteq G$ denote the maximal compact subgroups $\Orm({n+1})$ and $\Orm({n})$ belonging to the Cartan involution $\theta(X)=-X^\top$ on $\g$. Then the $(-1)$-eigenspace $\mathfrak{s}$ of $\theta$ in $\g$ is the space of symmetric matrices. Further, $G=KP$ and $G'=K'P'$ and therefore $G/P\cong K/(K\cap M)$ and $G'/P'\cong K'/(K'\cap M')$.

Let $\lambda \in \mathfrak{a}_{\C}^*$ and $\nu \in (\mathfrak{a}'_{\C})^*$. 
Let $X=\diag(1,-\frac{1}{{n}},\ldots,-\frac{1}{{n}}), Y=\diag(0,\frac{1}{{n}},\ldots,\frac{1}{{n}})\in \mathfrak{a}$ and $X'=\diag(1,-\frac{1}{{n-1}},\ldots,-\frac{1}{{n-1}},0), Y'=\diag(0,\frac{1}{{n-1}},\ldots,\frac{1}{{n-1}},0)\in \mathfrak{a}'$
We identify $\mathfrak{a}_{\C}^*\cong\C^2$ and $(\mathfrak{a}'_{\C})^*\cong\C^2$ via $\lambda\mapsto(\lambda_1,\lambda_2)=(\lambda(X),\lambda(Y))$ and $\nu\mapsto(\nu_1,\nu_2)=(\nu(X'),\nu(Y'))$.
Then for $l=\diag(a,h) \in MA$, $a \in \R^\times$, $h\in \GL({n-1},\R)$ and $l'=\diag(a',h',1)\in A'$, $a'\in \R^\times$, $h'\in \GL({n-1},\R)$ we have
\begin{align}
\label{eq:character_actions}
l^\lambda=\abs{a}^{\lambda_1}\abs{\det l}^{\lambda_2}, \quad
(l')^\nu=\abs{a'}^{\nu_1}\abs{\det l'}^{\nu_2},
\end{align}
For $\rho\in\mathfrak{a}^*$ and $\rho'\in(\mathfrak{a}')^*$ this means $\rho_1=\frac{{n+1}}{2},\rho'_1=\frac{{n}}{2}, \rho_2=\rho'_2=-\frac{1}{2}$.

Let $\mathbf{1}$ be the trivial $M$ resp. $M'$ representation. We consider the representations 
\begin{equation*}
\pi_\lambda=\Ind_P^G(\mathbf{1}\otimes e^\lambda \otimes \mathbf{1}) \qquad \text{ and } \qquad \tau_\nu=\Ind_{P'}^{G'}(\mathbf{1}\otimes e^\nu \otimes \mathbf{1}),
\end{equation*}
and use the notation $I_\lambda:=I_{\mathbf{1},\lambda}$ reps $J_\nu:=J_{\mathbf{1},\nu}$ for the three different pictures as described in Section~\ref{sec:GenPS}.

\subsection{The $P'$-orbits in $G/P$}

Let $b_0=[1:0:\ldots:0] \in \RP^{{n}}\cong G/P$, then $\bar{N}b_0$ is open and dense in $G/P$.
For $x \in \R^{{n}}$ let 
\[
n_x:=\left(
\begin{array}{c|c}
  1 & x^T \\ \hline
  0 & \raisebox{-15pt}{{\large\mbox{{$\mathbf{1}_{{n}}$}}}} \\[-3.7ex]
  \vdots & \\[+0.1ex]
   0 &
\end{array}
\right) \in N,
\]
and let $\bar{n}_x := n_x^T \in \bar{N}$. Then $\bar{n}_xb_0\ =[1:x]$ and $\Nbar b_0=\{[v_1:\ldots :v_{n+1}] \in \RP^{{n}}: v_1 \neq 0  \}$.
Hence we have an embedding of $\R^{{n}}$ into $G/P$:
\begin{equation*}
\begin{tikzcd}
\psi_1:\R^{{n}} \arrow[r,"\sim"] & \bar{N} \arrow[r,hook,"\cdot b_0"]  & \faktor{G}{P}, \qquad x \mapsto \bar{n}_x \cdot b_0.
\end{tikzcd}
\end{equation*}
In \cite[Theorem 3.16]{kobayashi_speh_2015} the study of $P'$-equivariant distributions on $G/P$ was reduced to the open cell $\bar{N}b_0$ under the condition that $P'\bar{N}P=G$, i.e. every $P'$-orbit in $G/P$ intersects $\bar{N}b_0$. This condition is not satisfied in the present situation as the following description of $P'$-orbits in $G/P$ shows.
\begin{corollary}
\label{corollary:orbits}
The orbits of the left action of $P'$ in $G/P\cong \RP^{{n}}$ are given by
\begin{itemize}
\item $\Ocal_1=\bar{\Ocal}_1=P'.[1:0:\ldots:0]=\{ [1:0:\ldots:0] \}$
\item $\Ocal_2=\bar{\Ocal}_2=P'.[0:\ldots:0:1]=\{ [0:\ldots:0:1] \}$
\item $\Ocal_3=P'.[1:0:\ldots:0:1]=\{ [v_1:0:\ldots:0:1]: v_1 \in \R^\times  \}\cong \RP^1 \setminus(\Ocal_1 \cup \Ocal_2)$
\item $\Ocal_4=P'.[0:1:0:\ldots :0]=
\{ [v_1:v:0]: v_1 \in \R, v \in \R^{{n-1}} \setminus \{0 \} 
\cong \RP^{{n-1}} \setminus 
\Ocal_1$
\item $\Ocal_5=P'.[0:\ldots:0:1:1]=\{ [v_1:v:1]: v_1 \in \R, v \in \R^{{n-1}} \setminus \{0 \}\} 
\cong \RP^{{n}}\setminus (\Ocal_1\cup\Ocal_2\cup\Ocal_3 \cup \Ocal_4)$
\end{itemize}
The closure relations of the orbits are given in the following diagram, where {\small $\mathclap{\begin{matrix}
&A \\ &| \\ & B
\end{matrix}}$} 
{\hspace{0.1cm}\tiny $m$}
means $B$ is a subvariety of $\bar{A}$ of co-dimension $m$.
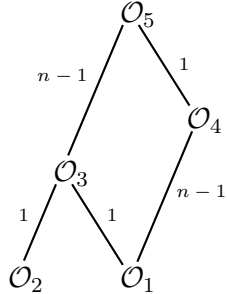
\begin{center} 
\begin{picture}(75,110)
\thicklines
\put(50,100){${\Ocal}_5$}
\put(51,99){\line(-25,-60){20.5}}
\put(57,99){\line(25,-40){19}}
\put(25,40){${\Ocal}_3$}
\put(75,60){${\Ocal}_4$}
\put(77,59){\line(-25,-60){20.5}}
\put(32,39){\line(25,-40){19}}
\put(50,0){${\Ocal}_1$}
\put(26,39){\line(-25,-60){12}}
\put(8,0){${\Ocal}_2$}
\put(12,25){{\tiny $1$}}
\put(45,25){{\tiny $1$}}
\put(71,34){{\tiny ${n-1}$}}
\put(72,82){{\tiny $1$}}
\put(19,78){{\tiny ${n-1}$}}
\end{picture}
\end{center}
\end{corollary}
Since the closed orbit $\Ocal_2$ does not intersect the open cell $U_1=\bar{N}b_0$ we need to consider an additional open subset of $G/P$ which intersects $\Ocal_2$. A convenient choice is obtained by translating $U_1$ with a certain Weyl group element. Let
\begin{equation*}
w_0=
\begin{pmatrix}
0 & 0 & 1 \\
0 & \mathbf{1}_{{n-1}} & 0 \\
1 & 0 & 0
\end{pmatrix}\in K.
\end{equation*}
Then the map
\begin{equation*}
\begin{tikzcd}
\psi_2:\R^{{n}} \arrow[r,"\sim"] & w_0\bar{N} \arrow[r,hook,"\cdot b_0"] & \faktor{G}{P}, \qquad x \mapsto w_0 \bar{n}_x \cdot b_0
\end{tikzcd}
\end{equation*}
is a diffeomorphism onto the open subset $U_2=w_0\bar{N}b_0=\{[v_1:\ldots:v_{n+1}]\in\RP^n:v_{n+1}\neq0\}$. Let $\R^{{n}}_{x_{{n}}}:=\{  (x',x_{{n}}) \in \R^{{n}}: x'\in \R^{{n-1}},x_{{n}} \neq 0  \}$. Then the map $\phi: \R^{{n}}_{x_{{n}}} \to \R^{{n}}_{x_{{n}}},\,(x',x_{{n}})\mapsto (x_{n}^{-1}x',x_{{n}}^{-1})$ is a diffeomorphism which makes the following diagram commute:
\begin{equation}
\label{eq:diagram_phi}
\begin{tikzcd}
\R^{{n}}_{x_{{n}}} \arrow[r,"\psi_1"] \arrow[d,"\phi"] & \faktor{G}{P} \\
\R^{{n}}_{x_{{n}}} \arrow[ur,"\psi_2"]
\end{tikzcd}
\end{equation}
In particular, $\psi_1(\R^{{n}}_{x_{{n}}})=\psi_2(\R^{{n}}_{x_{{n}}})=U_1\cap U_2$.
\begin{lemma}
\label{lemma:covering}
We have $G=P'\Nbar P \cup P' w_0\Nbar P$.
\end{lemma}
\begin{proof}
By Corollary~\ref{corollary:orbits} the set $U_1=\{  [v_1:\ldots:v_{n+1}]\in \RP^{{n}}: v_1 \neq0  \}$ meets all $P'$ orbits in $G/P$ except for $\Ocal_2$, which is contained in $U_2=\{  [v_1:\ldots:v_{n+1}]\in \RP^{{n}}: v_{n+1} \neq0  \}$.
\end{proof}

\subsection{Invariant distribution kernels on open cells}

The open subsets $U_1,U_2\subseteq G/P$ are particularly suitable to study the $P'$-action on $\mathcal{D}'(G/P,\mathcal{V}_{{\bf 1},-\lambda})$ since they are invariant under the action of $M'A'$ as the following lemma shows:
\begin{lemma}
\label{lemma:M'A'_inv}
We have 
\begin{inparaenum}[(i)] \item $M'A'\Nbar P \subseteq \Nbar P$, $M'A'N'w_0\Nbar P \subseteq w_0\Nbar P$ and \item $w_0P'w_0\subseteq P$.
\end{inparaenum}
\end{lemma}
\begin{proof}
For $l=\diag(a,h,1) \in M'A'$, $a\in \R^\times, h \in \GL({n-1},\R)$ and $\bar{n}_x \in \Nbar$ with $x=(x',x_{{n}})\in \R^{{n-1}}\times \R$:
\begin{equation}
\label{eq:M'A'_both}
l\bar{n}_x=\bar{n}_{(a^{-1}hx',a^{-1}x_{{n}})}l \in \bar{N}P, \quad lw_0\bar{n}_x=w_0\bar{n}_{(hx',ax_{{n}})}\diag(1,h,a) \in w_0\Nbar P.
\end{equation}
Let $y=(y',0) \in \R^{{n-1}}\times\{0\}$ such that $n_y\in N' $. Then
\begin{align}
\label{eq:N'_action_on_w_0Nbar}
n_yw_0\bar{n}_x=w_0\bar{n}_{((x',x_{{n}}+(y')^Tx'))}\begin{pmatrix}
1 & 0 & 0 \\
0 & \mathbf{1}_{{n-1}}& 0 \\
0 & y' & 1
\end{pmatrix}\in w_0\Nbar P.
\end{align}
Further we have $w_0lw_0^{-1}=\diag(1,h,a)\in P$ and
\begin{equation*}
w_0n_yw_0^{-1}=
\begin{pmatrix}
1 & 0 & 0 \\
0 & \mathbf{1}_{{n-1}} & 0 \\
0 & y' & 1
\end{pmatrix} \in P. \qedhere
\end{equation*}
\end{proof}

For every open subset $U\subseteq G/P$ the $G$-action on $\mathcal{D}'(G/P,\mathcal{V}_{\mathbf{1},-\lambda})$ induces an infinitesimal $\g$-action on $\mathcal{D}'(U,\mathcal{V}_{\mathbf{1},-\lambda}|_U)$ by vector fields. Let $i=1,2$, then $U=U_i$ is open and $M'A'$-invariant. By Lemma~\ref{lemma:M'A'_inv} we have an $(\g,M'A')$-action on $\mathcal{D}'(U_i,\mathcal{V}_{\mathbf{1},-\lambda}|_{U_i})$, such that the restriction
\begin{equation}
\label{eq:restriction_Nbar}
\mathcal{D}'(G/P,\mathcal{V}_{\mathbf{1},-\lambda}) \to \mathcal{D}'(U_i,\mathcal{V}_{\mathbf{1},-\lambda}|_{U_i}), \quad u\mapsto u|_{U_i}
\end{equation}
is $(\g,M'A')$-equivariant. Using the diffeomorphism $\psi_i:\R^n\to U_i$ we can identify $U_i$ with $\R^n$ and trivialize $\mathcal{V}_{\mathbf{1},-\lambda}|_{U_i}$ to obtain an isomorphism
\begin{equation}
\label{eq:bundle_triv}
 \psi_i^*:\mathcal{D}'(U_i,\mathcal{V}_{\mathbf{1},-\lambda}|_{U_i}) \to \mathcal{D}'(\R^{{n}}).
\end{equation}
This induces a $(\g,M'A')$-action $\tilde{\sigma}^i_\lambda$ on $\mathcal{D}'(\R^{{n}})$ that makes the isomorphism \eqref{eq:bundle_triv} intertwining. Then the action $\sigma^i_{\lambda,\nu}=\tilde{\sigma}^i_\lambda|_{(\mathfrak{n}',M'A')} \otimes W_{\mathbf{1},\nu}$ makes the isomorphism
\begin{equation}
\label{eq:bundle_triv_P'}
\mathcal{D}'(U_i,\mathcal{V}_{\mathbf{1},-\lambda}|_{U_i}) \otimes W_{\mathbf{1},\nu} \to \mathcal{D}'(\R^{{n}}).
\end{equation}
$(\mathfrak{n}',M'A')$-equivariant.

\begin{lemma}
\label{lemma:missing_factor}
Let $u\in \mathcal{D}'(U_1\cap U_2,\mathcal{V}_{\mathbf{1},-\lambda}|_{U_1\cap U_2})$. Then 
\begin{equation*}
u(\bar{n}_x b_0)=\abs{x_{{n}}}^{\lambda_1-\rho_1}u(w_0\bar{n}_{\phi(x)} b_0) \qquad \forall\,x \in \R^{{n}}_{x_{{n}}}.
\end{equation*} 
\end{lemma}
\begin{proof}
First we note that $x_{{n}} \neq 0$ for all $\bar{n}_x b_0 \in U_1\cap U_2=\Nbar b_0 \cap w_0 \Nbar b_0$. We have
\begin{equation*}
\bar{n}_x=w_0\bar{n}_{\phi(x)}\begin{pmatrix}
x_{{n}}& 0 & 1 \\
0 & \mathbf{1}_{{n-1}} & -x_{{n}}^{-1}x' \\
0 & 0 & -x_{{n}}^{-1}
\end{pmatrix} \in w_0\Nbar P.
\end{equation*}
Then \eqref{eq:character_actions} implies the statement.
\end{proof}

Write
$$ \mathcal{D}'(\R^n)^i_{\lambda,\nu} = \left\{u\in\mathcal{D}'(\R^n):\begin{array}{c}\sigma^i_{\lambda,\nu}(g)u=u\,\forall\,g\in M'A'\\\sigma^i_{\lambda,\nu}(X)u=0\,\forall\,X\in\mathfrak{n}'\end{array}\right\} \qquad (i=1,2) $$
for the subspaces of invariant distributions. The following result is a generalization of \cite[Theorem 3.16]{kobayashi_speh_2015} to the case where the open dense Bruhat cell $U_1=\bar{N}b_0\subseteq G/P$ does not meet every $P'$-orbit.

\begin{theorem}
\label{thm:charts}
The linear map 
$$ (\mathcal{D}'(G/P,\mathcal{V}_{\mathbf{1},-\lambda}) \otimes W_{\mathbf{1},\nu})^{P'}\to \mathcal{D}'(\R^n)\times\mathcal{D}'(\R^n), \quad u\mapsto \left(\psi_1^*(u|_{U_1}),\psi_2^*(u|_{U_2})\right) $$
is a linear isomorphism onto 
\begin{equation*}
\mathbb{D}(\lambda,\nu):= \biggr\{  
(u_1,u_2)\in \mathcal{D}'(\R^{{n}})^1_{\lambda,\nu}\times \mathcal{D}'(\R^{{n}})^2_{\lambda,\nu}:  
u_1|_{\R^{{n}}_{x_{{n}}}}=\abs{x_{{n}}}^{\lambda_1-\rho_1}\phi^*\left(u_2|_{\R^{{n}}_{x_{{n}}}}\right)
\biggr\}.
\end{equation*}
\end{theorem}
\begin{proof}
The maps \eqref{eq:restriction_Nbar} and \eqref{eq:bundle_triv} yield a map
\begin{align}
\label{eq:map_1}
\nonumber \mathcal{D}'(G/P,\mathcal{V}_{\mathbf{1},-\lambda})&\to \mathcal{D}'(\R^{{n}}) \times \mathcal{D}'(\R^{{n}})\\
u &\mapsto\left(\psi_1^*(u|_{U_1}),\psi_2^*(u|_{U_2})\right).
\end{align}
Then \eqref{eq:map_1} restricted to $P'$-invariant distributions yields by \eqref{eq:bundle_triv_P'}
\begin{equation}
\label{eq:Nbar_map}
(\mathcal{D}'(G/P,\mathcal{V}_{\mathbf{1},-\lambda})\otimes W_{\mathbf{1},\nu})^{P'} \rightarrow \mathcal{D}'(\R^{{n}})^1_{\lambda,\nu}\times\mathcal{D}'(\R^{{n}})^2_{\lambda,\nu}.
\end{equation}
By Lemma~\ref{lemma:covering} this map is injective and by Lemma~\ref{lemma:missing_factor} the image is contained in $\mathbb{D}(\lambda,\nu)$. Let $(u_1,u_2) \in \mathbb{D}(\lambda,\nu)$. Then by Lemma~\ref{lemma:missing_factor} we can define a distribution $u \in \mathcal{D}'(U_1\cup U_2,\mathcal{V}_{\mathbf{1},-\lambda}|_{U_1\cup U_2})$ by $u|_{U_1}=(\psi_1)_*u_1$ and $u|_{U_2}=(\psi_2)_*u_2$. Translating $u$ by $p'\in P'$ yields a family of distributions
$p'\cdot u \in \mathcal{D}'(p'.(U_1\cup U_2),\mathcal{V}_{\mathbf{1},-\lambda}|_{p'.(U_1\cup U_2)})$ which agree on the intersections since $u_1$ and $u_2$ are $(\mathfrak{n}',M'A')$-invariant. By Lemma~\ref{lemma:covering} and the glueing property of the sheaf of distributions this gives rise to a unique continuation of $u$ in $\mathcal{D}'(G/P,\mathcal{V}_{\mathbf{1},-\lambda})\otimes W_{\mathbf{1},\nu}$ that is $P'$-invariant and is mapped to $(u_1,u_2)$ under \eqref{eq:Nbar_map}.
\end{proof}

\subsection{Classification of invariant distribution kernels}

We introduce the following families of distributions on $\R^{{n}}$. For $(\lambda_1,\nu_1) \in \C^2$:
\begin{align*}
&v^{A}_{\lambda,\nu}(x',x_{{n}})=\frac{\delta(x')\abs{x_{{n}}}^{\lambda_1+\rho_1-\nu_1-\rho'_1-1}}{\Gamma(\frac{\lambda_1+\rho_1-\nu_1-\rho'_1}{2})\Gamma(\frac{\nu_1+\rho'_1}{2})}, && w^{A}_{\lambda,\nu}(x',x_{{n}})=\frac{\delta(x')\abs{x_{{n}}}^{\nu_1+\rho'_1-1}}{\Gamma(\frac{\lambda_1+\rho_1-\nu_1-\rho'_1}{2})\Gamma(\frac{\nu_1+\rho'_1}{2})},
\end{align*}
for $\nu_1 = -\rho'_1-2j$, $j \in \Z_{\geq 0}$:
\begin{align*}
&v^{A,1}_{\lambda,\nu}(x',x_{{n}})=0, && w^{A,1}_{\lambda,\nu}(x',x_{{n}})=\delta^{(2j)}(x_{{n}})\delta(x'),
\end{align*}
for $\lambda_1+\rho_1-\nu_1-\rho'_1=-2N$, $N \in \Z_{\geq 0}$:
\begin{align*}
&v^{A,2}_{\lambda,\nu}(x',x_{{n}})=\delta^{(2N)}(x_{{n}})\delta(x'), && w^{A,2}_{\lambda,\nu}(x',x_{{n}})=0,\\
\end{align*}
for $\nu_1=-\rho'_1$:
\begin{align*}
&v^{B}_{\lambda,\nu}(x',x_{{n}})=\frac{\abs{x_n}^{\lambda_1-\rho_1}}{\Gamma(\frac{\lambda_1-\rho_1+1}{2})}, 
&& w^{B}_{\lambda,\nu}(x',x_{{n}})=\frac{1}{\Gamma(\frac{\lambda_1-\rho_1+1}{2})},
\end{align*}
and for $n=2$ and $(\lambda_1,\nu_1)\in \C^2$ additionally:
\begin{align*}
&v^{C}_{\lambda,\nu}(x_1,x_{{2}})=\frac{\abs{x_1}^{-\nu_1-\rho'_1}\abs{x_2}^{\lambda_1+\rho_1+\nu_1-\rho'_1-1}}{\Gamma(\frac{-\nu_1-\rho'_1+1}{2})\Gamma(\frac{\lambda_1+\rho_1+\nu_1-\rho'_1}{2})}, &&w^{C}_{\lambda,\nu}(x_1,x_{{2}})=\frac{\abs{x_1}^{-\nu_1-\rho'_1}}{\Gamma(\frac{-\nu_1-\rho'_1+1}{2})\Gamma(\frac{\lambda_1+\rho_1+\nu_1-\rho'_1}{2})}.
\end{align*}
We write for $S=A, (A,1), (A,2), B, C$
\begin{equation*}
u^{S}_{\lambda,\nu}:=(v^{S}_{\lambda,\nu},w^{S}_{\lambda,\nu}).
\end{equation*}
By Lemma~\ref{lemma:absolute_value_distr_holomorphic} $u^{S}_{\lambda,\nu}$ are holomorphic families of pairs of distributions and the following holds:
\begin{corollary}
\label{corollary:residues_distr}
The pair $u^{A}_{\lambda,\nu}= 0$ if and only if $(\lambda_1,\nu_1)\in L$. Moreover, the following residue formulas hold:
\begin{align*}
&u^{A}_{\lambda,\nu}=(-1)^j\frac{ j!}{(2j)!\Gamma(\frac{\lambda_1+\rho_1-\nu_1-\rho'_1}{2})}u^{A,1}_{\lambda,\nu} &&\text{for } \nu_1=-\rho'_1-2j, j \in \Z_{\geq 0}, \\ &u^{A}_{\lambda,\nu}=(-1)^N\frac{N!}{(2N)!\Gamma(\frac{\nu_1+\rho'_1}{2})}u^{A,2}_{\lambda,\nu} &&\text{for } \lambda_1+\rho_1-\nu_1-\rho'_1 = -2N, N\in \Z_{\geq 0}.
\intertext{Additionally, for $n=2$ the pair $u^C_{\lambda,\nu}$ never vanishes and}
&u^{A}_{\lambda,\nu}=\frac{1}{\sqrt{\pi}}u^C_{\lambda,\nu} &&\text{for } n=2, \nu_1+\rho'_1=-2 \rho_2, \\
&u^B_{\lambda,\nu}=\sqrt{\pi}u^C_{\lambda,\nu} &&\text{for } n=2, \nu_1+\rho'_1=0.
\end{align*}
\end{corollary}
For the statement describing $\mathbb{D}(\lambda,\nu)$ let
\begin{equation}
 L=\{ (-\rho_1-2i,-\rho'_1-2j) \in \C^2: i,j \in \Z_{\geq 0}, j\leq i \}\label{eq:DefL}
\end{equation}
\begin{prop}
\label{prop:classification_distr}
For $n\geq 3$ we have
\begin{equation*}
\mathbb{D}(\lambda,\nu)=
	\begin{cases}
		\C u^{A}_{\lambda,\nu} & \text{if } \lambda_2+\rho_2=\nu_2+\rho'_2,(\lambda_1,\nu_1) \notin L, \\
		\C u^{A,1}_{\lambda,\nu} \oplus \C u^{A,2}_{\lambda,\nu} & \text{if } \lambda_2+\rho_2=\nu_2+\rho'_2, (\lambda_1,\nu_1)\in L, \\
		\C u^{B}_{\lambda,\nu} &  \text{if } \lambda_2-\rho_2-\nu_2-\rho'_2=\nu_1+\rho'_1=0, \\
		\{0\} & \text{otherwise}.
	\end{cases}
\end{equation*}
For $n=2$ we have
\begin{equation*}
\mathbb{D}(\lambda,\nu)=
	\begin{cases}
		\C u^{A}_{\lambda,\nu} & \text{if } \lambda_2+\rho_2=\nu_2+\rho'_2,(\lambda_1,\nu_1) \notin L, \\
		\C u^{A,1}_{\lambda,\nu} \oplus \C u^{A,2}_{\lambda,\nu} & \text{if } \lambda_2+\rho_2=\nu_2+\rho'_2, (\lambda_1,\nu_1)\in L, \\
		\C u^{C}_{\lambda,\nu} &  \text{if } \lambda_2-\rho_2-\nu_2-\rho'_2=\nu_1+\rho'_1, \\
		\{0\} & \text{otherwise}.
	\end{cases}
\end{equation*}
\end{prop}
Proposition~\ref{prop:classification_distr} together with Theorem~\ref{thm:charts} and Theorem~\ref{theorem:SBO-distr} implies Theorem~\ref{theorem:A}.
We will prove Proposition~\ref{prop:classification_distr} in two steps. First, we separately classify invariant distributions in $\mathcal{D}'(\R^{{n}})^i_{\lambda,\nu}$ ($i=1,2$). Then we compare their pullbacks via $\phi$ resp. $\phi^{-1}$.

For the first step we need to study the actions $\sigma^i_{\lambda,\nu}$ ($i=1,2$) in detail. For $j=1,\ldots,{n}$ let $N_j \in \mathfrak{n}$ be the matrices given by $(N_j)_{k,l}=\delta_{j+1,k}\delta_{1,l}$. Then for $x=(x_1,\ldots,x_{{n}})$ we have $\exp(\sum_{j=1}^{{n}}x_jN_j)=n_x$. Let $E:=\sum_{i=1}^{{n}} x_j\frac{\partial}{\partial x_j}$ denote the Euler-operator on $\R^{{n}}$.
\begin{lemma}
\label{prop:induced_actions}
For $a \in \R^\times, h\in \GL({n-1},\R)$ let $l=\diag(a,h,1) \in M'A'$ and $N_i\in \mathfrak{n}'$, $j\in \{1,\ldots,{n-1}\}$. The $(\mathfrak{n}',M'A')$-actions $\sigma^i_{\lambda,\nu}$ ($i=1,2$) on a generalized function $u \in \mathcal{D}'(\R^{{n}})$ are for $x=(x',x_{{n}}) \in \R^{{n-1}}\times \R$ given by
\begin{enumerate}[label=(\roman{*}), ref=\thetheorem(\roman*)]
\item 
\label{prop:induced_actions:1}
$
\sigma^1_{\lambda,\nu}(l) u(x',x_{{n}})= \abs{a}^{\nu_1+\rho'_1-\lambda_1+\rho_1}\abs{a \det h}^{\nu_2+\rho'_2-\lambda_2+\rho_2}u(ah^{-1}x',ax_{{n}}),
$
\item
\label{prop:induced_actions:2}
$
\sigma^1_{\lambda,\nu}(N_j) u(x) =x_j\left(E-(\lambda_1-\rho_1)\right)u(x),
$
\item
\label{prop:induced_actions:3}
$
\sigma^2_{\lambda,\nu}(l) u(x',x_{{n}}) =
\abs{a}^{\nu_1+\rho'_1}\abs{a\det h}^{\nu_2+\rho'_2-\lambda_2+\rho_2}u(h^{-1}x',a^{-1}x_{{n}}),
$
\item
\label{prop:induced_actions:4}
$
\sigma^2_{\lambda,\nu}(N_j)u(x)=x_j \frac{\partial }{\partial x_{{n}}}u(x).
$
\end{enumerate}
\end{lemma}
\begin{proof}
The $G$-action on $\mathcal{D}'(G/P,\mathcal{V}_{\mathbf{1},-\lambda})$ is given by the left-regular action. 
Now \eqref{eq:M'A'_both} implies that pulling back the left-regular action of $l \in M'A'$ on $U_i$ ($i=1,2$) along $\psi_i$ yields the left-regular action of $\diag(ah^{-1},a)$ for $i=1$ and the left-regular action of $\diag(h^{-1},a^{-1})$ for $i=2$. Then \eqref{eq:character_actions} concludes the proof of \ref{prop:induced_actions:1} and \ref{prop:induced_actions:3}.

Ad(ii):
Fix $x=(x',x_{{n}})\in \R^{{n-1}}\times \R$. Let $y=(y',0)\in \R^{{n-1}}\times \{0\}$. For small $y$ we have
\begin{align*}
n_{y}\bar{n}_x&=\left(
\begin{array}{c|c}
  1+y^Tx & y^T \\ \hline
   & \raisebox{-15pt}{{\large\mbox{{$\mathbf{1}_{{n}}$}}}} \\[-3.7ex]
  x & \\[+0.1ex]
    &
\end{array}
\right)
\\
&=\bar{n}_{\frac{x}{1+y^Tx}}\left(
\begin{array}{c|c}
  1+y^Tx & 0  \\ \hline
   & \raisebox{-15pt}{{\large\mbox{{$\mathbf{1}_{{n}}-\frac{xy^T}{1+y^Tx}$}}}} \\[-3.7ex]
  0 & \\[+0.1ex]
     &
\end{array}
\right) n_{\frac{y}{1+y^Tx}} \in \bar{N}MAN,
\end{align*}
 which implies that the action $\pi_\lambda$ of $N'$ on $v \in \mathcal{D}'(U_1,\mathcal{V}_{\mathbf{1},-\lambda}|_{U_1})$ for fixed $x \in \R^{{n}}$ and small $y$ is given by
\begin{align*}
\pi_\lambda(n_y)v(\bar{n}_x b_0)=(1-y^Tx)^{\lambda_1-\rho_1}v\left(\bar{n}_{\frac{x}{1-y^Tx}}b_0\right).
\end{align*}
Putting $y=te_j$ with $j\in \{1,\ldots ,{n-1} \}$ and differentiating at $t=0$ gives
\begin{equation*}
\sigma^1_{\lambda,\nu}(N_j)u(x)=\frac{d}{dt}\biggr|_{t=0}(1-tx_j)^{\lambda_1-\rho_1}u\left({\frac{x}{1-tx_j}}\right) 
=x_j\left(E-(\lambda_1-\rho_1)\right)u(x).
\end{equation*}

Ad(iv): By \eqref{eq:N'_action_on_w_0Nbar} 
the action $\sigma_{w_0}|_{\mathfrak{n}'}$ is the infinitesimal action of the induced $N'$-action $\pi_{-\lambda}$ on $\mathcal{D}'(\R^{{n}})$ given by
\begin{align*}
\pi_{-\lambda}(n_y)u(w_0\bar{n}_xb_0)=u\left(w_0\bar{n}_{(x',x_{{n}}+(y')^Tx')}b_0\right).
\end{align*}
Hence for $j \in \{ 1,\ldots ,{n-1}  \}$
\begin{equation*}
\sigma^2_{\lambda,\nu}(N_j)u(x)=\frac{d}{dt}\biggr|_{t=0}u(x',x_{{n}}+tx_j)
=x_j \frac{\partial }{\partial x_{{n}}}u(x). \qedhere
\end{equation*}
\end{proof}

\begin{lemma}
\label{lemma:classification_sigma}
For $\lambda_2+\rho_2=\nu_2+\rho'_2$ we have
\begin{equation*}
\mathcal{D}'(\R^{{n}})^1_{\lambda,\nu}= 
\begin{cases}
\C \cdot v^{A}_{\lambda,\nu} & \text{if } (\lambda_1,\nu_1)\notin L \\
\C \cdot v^{A,2}_{\lambda,\nu} & \text{if } (\lambda_1,\nu_1) \in L .
\end{cases}
\end{equation*}
For $n>2$ and $\lambda_2-\rho_2=\nu_2+\rho'_2 \text{ and } \nu_1=-\rho'_1,$ we have
\begin{equation*}
\mathcal{D}'(\R^{{n}})^1_{\lambda,\nu}= \C v^{B}_{\lambda,\nu}.
\end{equation*}
For $n=2$ and $\lambda_2-\rho_2 -\nu_2-\rho'_2=\nu_1+\rho'_1$ we have
\begin{equation*}
\mathcal{D}'(\R^{{2}})^1_{\lambda,\nu}= \C v^C_{\lambda,\nu}
\end{equation*}
In all other cases $\mathcal{D}'(\R^{{n}})^1_{\lambda,\nu}=\{0\}$.
\end{lemma}
\begin{proof}
Let $0\neq u\in\mathcal{D}'(\R^{n})^1_{\lambda,\nu}$. 
To simplify notation we write $\mu_1=\nu_1+\rho'_1-\lambda_1+\rho_1$ and $\mu_2=\nu_2+\rho'_2-\lambda_2+\rho_2$.
Let $l=\diag(a,1,\ldots,1)$, $a \in \R^\times$. Then by  Lemma~\ref{prop:induced_actions:1}
\begin{equation*}
\sigma(l)u(x',x_n)=\abs{a}^{\mu_1+\mu_2}u(ax',ax_n),
\end{equation*}
so that $u$ has to be homogeneous of degree $-\mu_1-\mu_2$. Similarly the action of $l$ for $l=\diag(1,a,\ldots,a,1)$ and $l=\diag(a,a^{-\frac{1}{n-1}},\ldots,a^{-\frac{1}{n-1}},1)$ implies that
$u(x',x_n)=u'(x')u_n(x_n)$ with $u'\in \mathcal{D}'(\R^{n-1})$ homogeneous of degree $(n-1)\mu_2$ and $u_n\in \mathcal{D}'(\R)$ homogeneous of degree $-\mu_1-n\mu_2$. Moreover $u_n$ has to be even and $u'$ has to be invariant under left-regular action of matrices with determinant equal to $\pm 1$. Hence by Lemma~\ref{lemma:homogenous_distributions} we can assume
\begin{align*}
&u_n(x_n)\in \C \frac{\abs{x_n}^{-\mu_1-n\mu_2}}{\Gamma(\frac{-\mu_1-n\mu_2+1}{2})}, &&u'(x')\in \C \frac{\abs{x'}^{(n-1)\mu_2}}{\Gamma(\frac{(n-1)(\mu_2+1)}{2})}.
\end{align*}
Now Lemma~\ref{prop:induced_actions:2} implies that either $x_ju'(x')=0$ for all $j=1,\ldots,n$ or $-\mu_1-\mu_2=\lambda_1-\rho_1$. In the first case it follows that $\mu_2=-1$ such that $u'(x')$ is a multiple of $\delta(x')$. Hence $\lambda_2-\rho_2=\nu_2+\rho'_2$ and $-\mu_1-n\mu_2=-\mu_1+n=\lambda_1+\rho_1-\nu_1-\rho'_1-1$. For the second case let $-\mu_1-\mu_2=\lambda_1-\rho_1$, i.e. $\mu_2=-\nu_1-\rho'_1$. For $n>2$ the $\SL(n,\R)$-invariance of $u'$ implies $\mu_2=0$ i.e. that $u'$ is constant. Then $\lambda_2-\rho_2=\nu_2+\rho'_2$ and $u\in \C v^B_{\lambda,\nu}$.
For $n=2$ we just obtain $u \in \C v^C_{\lambda,\nu}$.\qedhere
\end{proof}

\begin{lemma}
\label{lemma:classification_sigma_w_0}
For $\lambda_2+\rho_2=\nu_2+\rho'_2$ we have
\begin{equation*}
\mathcal{D}'(\R^{{n}})^2_{\lambda,\nu}=
\begin{cases}
\C  w^{A}_{\lambda,\nu} & \text{if } (\lambda_1,\nu_1) \notin L, \\
\C w^{A,1}_{\lambda,\nu} & \text{if } (\lambda_1,\nu_1) \in L, \\
\end{cases}
\end{equation*}
For $n>2$ and $\lambda_2-\rho_2=\nu_2+\rho'_2 \text{ and } \nu_1=-\rho'_1$ we have
\begin{equation*}
\mathcal{D}'(\R^{{n}})^2_{\lambda,\nu}=\C.
\end{equation*}
For $n=2$ and $\lambda_2-\rho_2-\nu_2-\rho'_2=\nu_1+\rho'_1$ we have
\begin{equation*}
\mathcal{D}'(\R^2)^2_{\lambda,\nu}=\C\frac{\abs{x_1}^{-\nu_1-\rho'_1}}{\Gamma(\frac{-\nu_1-\rho'_1+1}{2})}
\end{equation*}
In all other cases $\mathcal{D}'(\R^{{n}})^2_{\lambda,\nu}=\{0\}$.
\end{lemma}
\begin{proof}
Let $0\neq u \in \mathcal{D}'(\R^{n})^2_{\lambda,\nu}$. As before we write $\mu_2=\nu_2+\rho'_2-\lambda_2+\rho_2$.
The proof works in the same way as the proof of Lemma~\ref{lemma:classification_sigma}: Acting with $l=\diag(1,a,\ldots,a,1)$, $l=\diag(a,1,\ldots,1)$ and $\diag(a,\ldots,a,1)$ for $a\in \R^\times$ implies by Lemma~\ref{prop:induced_actions:3} that $u(x',x_n)=u'(x')u_n(x_n)$ is homogeneous of degree $\nu_1+\rho'_1+n\mu_2$ with $u'\in \mathcal{D}'(\R^{n-1})$ homogeneous of degree $(n-1)\mu_2$ and $u_n \in \mathcal{D}'(\R)$ homogeneous of degree $\nu_1+\rho'_1+\mu_2$. Since $u_n$ has to be even and $u'$ rotation-invariant by Lemma~\ref{lemma:homogenous_distributions} we can assume
\begin{align*}
&u_n(x_n)=\frac{\abs{x_n}^{\nu_1+\rho'_1+\mu_2}}{\Gamma(\frac{\nu_1+\rho'_1+\mu_2+1}{2})}, &&u'(x')=\frac{\abs{x'}^{(n-1)\mu_2}}{\Gamma(\frac{(n-1)(\mu_2+1)}{2})}.
\end{align*}
Then by Lemma~\ref{prop:induced_actions:4} $u$ is $\mathfrak{n}'$-invariant in two cases. Either $x_iu'(x')$ vanishes for all $i=1,\ldots,n-1$ which implies $\mu_2=-1$, so that $u'$ is a multiple of $\delta(x')$. Hence in this case
\begin{equation*}
u \in \C \frac{\abs{x_n}^{\nu_1+\rho'_1-1}}{\Gamma(\frac{\nu_1+\rho'_1}{2})}\delta(x')
\end{equation*}
and $\lambda_2+\rho_2=\nu_2+\rho'_2$. The second case is $u_n$ is constant i.e. $\nu_1+\rho'_1+\mu_2=0$. For $n>2$ the $\SL(n-1,\R)$-invariance of $u'$ implies also that $u'$ has to be constant, i.e. $\mu_2=0$ such that $\nu_2+\rho'_2-\lambda_2+\rho_2=\nu_1+\rho'_1=0$ and $u$ is constant.
If $n=2$ it just follows $\nu_2+\rho'_2-\lambda_2+\rho_2=\nu_1+\rho'_1$ and
\begin{equation*}
u(x_1,x_2) \in \C \frac{\abs{x_1}^{-\nu_1-\rho'_1}}{\Gamma(\frac{-\nu_1-\rho'_1+1}{2})}.\qedhere
\end{equation*}
\end{proof}
\begin{corollary}
\label{corollary:pullback}
\begin{enumerate}[label=(\roman{*}), ref=\thetheorem(\roman*)]
\item For $(\lambda_1,\nu_1)\in\C^2$
\begin{equation*}
v^{A}_{\lambda,\nu}|_{\R^{{n}}_{x_{{n}}}}=\abs{x_{{n}}}^{\lambda_1-\rho_1}\phi^*\left(w^{A}_{\lambda,\nu}|_{\R^{{n}}_{x_{{n}}}}\right).
\end{equation*}
\item For $\nu_1=-\rho'_1-2j$, $j \in \Z_{\geq 0}$
\begin{equation*}
\phi^*\left(w^{A,1}_{\lambda,\nu}|_{\R^{{n}}_{x_{{n}}}}\right)=0.
\end{equation*}
\item For $\lambda_1+\rho_1-\nu_1-\rho'_1=-2N$, $\N \in \Z_{\geq 0}$
\begin{equation*}
(\phi^{-1})^*\left(v^{A,2}_{\lambda,\nu}|_{\R^{{n}}_{x_{{n}}}}\right)=0.
\end{equation*}
\item For $\nu_1=-\rho_1'$
\begin{equation*}
v^{B}_{\lambda,\nu}|_{\R^{{n}}_{x_{{n}}}}=\abs{x_{{n}}}^{\lambda_1-\rho_1}\phi^*\left(w^{B}_{\lambda,\nu}|_{\R^{{n}}_{x_{{n}}}}\right)
\end{equation*}
\item For $n=2$
\begin{equation*}
v^C_{\lambda,\nu}|_{\R^{{2}}_{x_{{2}}}}=\abs{x_{{2}}}^{\lambda_1-\rho_1}\phi^*\left(w^{C}_{\lambda,\nu}|_{\R^{{2}}_{x_{{2}}}}\right)
\end{equation*}
\end{enumerate}
\end{corollary}
\begin{proof}
(ii) and (iii) follow immediately from the fact that the support of the distributions in question is disjoint from $\R^{{n}}_{x_{{n}}}$ and (iv) and (v) are true by definition.

Ad (i): We prove this identity for $\Re\lambda_1+\rho_1>\Re\nu_1+\rho_1'>0$, then the general statement follows by analytic continuation. Abbreviate $c=\Gamma(\frac{\lambda_1+\rho_1-\nu_1-\rho_1'}{2})\Gamma(\frac{\nu_1+\rho_1'}{2})$. Since $\phi^{-1}=\phi$ we have
\begin{align*}
\langle \phi ^* w^{A}_{\lambda,\nu}|_{\R^{{n}}_{x_{{n}}}}, \varphi \rangle = c^{-1}\int_{\R^{{n}}_{x_{{n}}}}
\abs{x_{{n}}}^{\nu_1+\rho'_1-1}\delta(x') \varphi(\phi(x)) \abs{\det D\phi(x)}dx,
\end{align*}
where $D\phi$ is a upper triangular matrix with $x_{{n}}^{-1}$ on the first $({n-1})$-diagonal entries and $-x^{-2}_{{n}}$ on the last. Hence
\begin{align*}
\langle \phi ^* w^{A}_{\lambda,\nu}|_{\R^{{n}}_{x_{{n}}}}, \varphi \rangle = c^{-1}\int_{\R^{{n}}_{x_{{n}}}}
\abs{x_{{n}}}^{\nu_1+\rho'_1-{n}-2}\delta(x') \varphi(\phi(x)) dx.
\end{align*}
We can pull back by the map $\R^\times \to \R^\times$ given by $y\mapsto y^{-1}$ with Jacobian $-y^{-2}$ and get
\begin{multline*}
c^{-1}\int_{\R^\times}
\abs{x_{{n}}}^{\nu_1+\rho'_1-{n}-2} \varphi(0,x_{{n}}^{-1}) dx_{{n}} = 
c^{-1}\int_{\R^\times}
\abs{x_{{n}}}^{-\nu_1-\rho'_1+{n}} \varphi(0,x_{{n}}) dx_{{n}} \\
=c^{-1}\int_{\R^\times}
\abs{x_{{n}}}^{(\lambda_1+\rho_1-\nu_1-\rho'_1)-1 -\lambda_1+\rho_1} \varphi(0,x_{{n}}) dx_{{n}}
=\langle \abs{x_{{n}}}^{-\lambda_1+\rho_1}v_{\lambda,\nu}^A|_{\R^{{n}}_{x_{{n}}}},\varphi \rangle.\qedhere
\end{multline*}
\end{proof}
\begin{proof}[Proof of Proposition~\ref{prop:classification_distr}]
We know that $u^{A}_{\lambda,\nu}=0$ if and only if $(\lambda_1,\nu_1)\in L$, and in this case $u^{A,1}_{\lambda_1,\nu_1}$ and $u^{A,2}_{\lambda_1,\nu_1}$ are linearly independent.
If $\nu_1\in-\rho'_1-2\Z_{\geq 0}$ and $(\lambda_1,\nu_1)\notin L$ the distributions $u^{A}_{\lambda_1,\nu_1}$ and $u^{A,1}_{\lambda_1,\nu_1}$ are scalar multiples of each other by Lemma~\ref{lemma:absolute_value_distr_holomorphic}. If $\lambda_1+\rho_1-\nu_1-\rho'_1 \in -2\Z_{\geq 0}$ and $(\lambda_1,\nu_1)\notin L$  the distributions $u^{A}_{\lambda_1,\nu_1}$ and $u^{A,2}_{\lambda_1,\nu_1}$ are scalar multiples of each other and $u^A_{\lambda,\nu}$ and $u^B_{\lambda,\nu}$ can not exist at the same time. If $n=2$ and $u^A_{\lambda,\nu}$ and $u^C_{\lambda,\nu}$ exist, they are scalar multiples of each other as well as $u^B_{\lambda,\nu}$ and $u^C_{\lambda,\nu}$ if both exist. Moreover $u^C_{\lambda,\nu} \neq 0$ for all $(\lambda_1,\nu_1) \in \C^2$. Then Lemma~\ref{lemma:classification_sigma}, Lemma~\ref{lemma:classification_sigma_w_0} and Corollary~\ref{corollary:pullback} imply the proposition.
\end{proof}

\subsection{Symmetry breaking operators in the non-compact picture}
From our classification of the distribution kernels of symmetry breaking operators, we first derive formulas for all operators in the non-compact picture, i.e. as operators $I_\lambda(\Nbar)\to J_\nu(\Nbar')$ from functions on $\Nbar\cong\R^n$ to functions on $\Nbar'\cong\R^{n-1}$.
\begin{theorem}\label{thm:SBOsNonCptPicture}
The operators $\mathcal{A}_{\lambda_,\nu},\mathcal{A}^{(1)}_{\lambda_,\nu}, \mathcal{A}^{(2)}_{\lambda_,\nu}, \mathcal{B}_{\lambda_,\nu}, \mathcal{C}_{\lambda,\nu} \in \Hom_{G'}(I_{\lambda}(\Nbar),J_{\nu}(\Nbar'))$ corresponding to the distributions $u^{A}_{\lambda,\nu}, u^{A,1}_{\lambda,\nu}, u^{A,2}_{\lambda,\nu}, u^{B}_{\lambda,\nu}, u^C_{\lambda,\nu}\in\mathbb{D}(\lambda,\nu)$ are given as follows:
\begin{enumerate}[label=(\roman{*}), ref=\thetheorem(\roman*)]
\item For $(\lambda_1,\nu_1)\in\C^2$
\begin{equation*}\mathcal{A}_{\lambda,\nu} \varphi(y)=\frac{1}{{\Gamma(\frac{\nu_1+\rho'_1}{2})\Gamma(\frac{\lambda_1+\rho_1-\nu_1-\rho'_1}{2})}}\int_{\R}{\abs{x_{{n}}}^{\lambda_1+\rho_1-\nu_1-\rho'_1-1}}\varphi(y,x_{{n}})dx_{{n}}.\end{equation*}
\item For $\nu_1=-\rho'_1-2j ,j\in \Z_{\geq 0}$ \begin{equation*}\mathcal{A}^{(1)}_{\lambda,\nu}\varphi(y)=\lim\limits_{t\to 0}\frac{d^{2j}}{dt^{2j}}\left( \abs{t}^{-\lambda_1-\rho_1}\varphi(y,t^{-1}) \right). \end{equation*}
\item  For $\lambda_1+\rho_1-\nu_1-\rho'_1=-2N, N\in \Z_{\geq 0}$ \begin{equation*}\mathcal{A}^{(2)}_{\lambda,\nu}\varphi(y)=\frac{\partial^{2N}\varphi}{\partial x_{{n}}^{2N}}(y,0).\end{equation*}
\item For $\nu_1=-\rho'_1$
\begin{equation*}\mathcal{B}_{\lambda,\nu}\varphi(y)= \frac{1}{{\Gamma(\frac{\lambda_1-\rho_1+1}{2})\Gamma(\frac{n-1}{2})}} \int_{\R^{{n}}}{\abs{x_{{n}}}^{\lambda_1-\rho_1}}\varphi(x)dx.
\end{equation*}
\item For $n=2$ and $(\lambda_1,\nu_1)\in\C^2$
\begin{equation*}
\mathcal{C}_{\lambda,\nu}\varphi(y)=\frac{1}{{\Gamma(\frac{\lambda_1+\rho_1+\nu_1-\rho'_1}{2})\Gamma(\frac{-\nu_1-\rho'_1+1}{2})}} \int_{\R^{{2}}}{\abs{x_1-y}^{-\nu_1-\rho'_1}\abs{x_{{2}}}^{\lambda_1+\rho_1+\nu_1-\rho'_1-1}}\varphi(x)dx.
\end{equation*}
\end{enumerate}
\end{theorem}
\begin{proof}
Assume first that a pair of distributions $u=(v,w)\in\mathbb{D}(\lambda,\nu)$ corresponds via Theorem~\ref{thm:charts} to a kernel on $G/P$ which is either supported on a single $P'$-orbit contained in $U_1$, or which is given by a sufficiently regular function times an equivariant measure on a single $P'$-orbit which intersects $U_1$ non-trivially. Then for $\bar{n}_y \in \Nbar'$, $y \in \R^{{n-1}}$ and $\varphi \in I_{\lambda}(\Nbar)$ the operator $\mathcal{A}\in\Hom_{G'}(I_{\lambda}(\Nbar),J_{\nu}(\Nbar'))$ corresponding to $u=(v,w)$ is given by
\begin{equation*}
\mathcal{A}\varphi(\bar{n}_{y})= \int_{\Nbar}v (\bar{n}_{(y,0)}^{-1}\bar{n}_x)\varphi(\bar{n}_x) d\bar{n}_x = \int_{\R^{{n}}} v(x-(y,0))\varphi(x)dx
\end{equation*}
in the distribution sense. This applies to the distributions $u_{\lambda,\nu}^A$, $u_{\lambda,\nu}^{A,2}$, $u_{\lambda,\nu}^B$ and $u_{\lambda,\nu}^C$ if $\Re\lambda_1\gg\Re\nu_1\gg0$. For instance,
\begin{multline*}
\mathcal{A}_{\lambda,\nu}\varphi(y)=\int_{\R^{{n}}} v^{A}_{\lambda,\nu}(x-(y,0))\varphi(x)dx =\int_{\R^{{n}}}\frac{\abs{x_{{n}}}^{\lambda_1+\rho_1-\nu_1-\rho'_1-1}}{\Gamma(\frac{\nu_1+\rho'_1}{2})\Gamma(\frac{\lambda_1+\rho_1-\nu_1-\rho'_1}{2})}\delta_{x'-y}\varphi(x',x_{{n}})dx \\
=\int_{\R}\frac{\abs{x_{{n}}}^{\lambda_1+\rho_1-\nu_1-\rho'_1-1}}{\Gamma(\frac{\nu_1+\rho'_1}{2})\Gamma(\frac{\lambda_1+\rho_1-\nu_1-\rho'_1}{2})}\varphi((y,x_{{n}}))d{x_{{n}}}.
\end{multline*}
The claimed formulas for $\mathcal{A}_{\lambda,\nu}^{(2)}$, $\mathcal{B}_{\lambda,\nu}$ and $\mathcal{C}_{\lambda,\nu}$ follow in the same way. To obtain the above expression for $\mathcal{A}_{\lambda,\nu}^{(1)}$ we use the residue formula of Corollary~\ref{corollary:residues_distr}. For $\nu_1=-\rho'_1-2j, j \in \Z_{\geq 0}$ we find
\begin{align*}
 \mathcal{A}^{(1)}_{\lambda,\nu}\varphi(y) &= (-1)^j\frac{(2j)!}{j!}\Bigg[\int_{\R} \frac{\abs{x_{{n}}^{-1}}^{\nu_1+\rho'_1-1}}{\Gamma(\frac{\nu_1+\rho'_1}{2})}\abs{x_{{n}}}^{\lambda_1+\rho_1-2} \varphi(y,x_{{n}})dx_{{n}}
\Bigg]_{\nu_1+\rho_1'=-2j}\\
 &= (-1)^j\frac{(2j)!}{j!}\Bigg[\int_{\R}\frac{\abs{t}^{\nu_1+\rho'_1-1}}{\Gamma(\frac{\nu_1+\rho'_1}{2})}\abs{t}^{-\lambda_1-\rho_1} \varphi(y,t^{-1})dt\Bigg]_{\nu_1+\rho_1'=-2j}\\
 &= \lim\limits_{t\to 0}\frac{d^{2j}}{dt^{2j}}\left( \abs{t}^{-\lambda_1-\rho_1}\varphi(y,t^{-1}) \right).
\end{align*}
where we have used Lemma~\ref{lemma:absolute_value_distr_holomorphic} in the last step. Note that this limit exists for every $\varphi\in I_\lambda(\Nbar)$. In fact, using the notation of the next section, for $t \neq 0$:
\begin{align}
\label{eq:limit_existence}
\abs{t}^{-\lambda_1-\rho_1} \varphi(y,t^{-1})=\abs{t}^{-\lambda_1-\rho_1} \varphi(\bar{n}_{(y,t^{-1})})&=\left(\abs{t}\sqrt{1+\abs{y}^2+t^{-2}}\right)^{-\lambda_1-\rho_1}\gamma_{\lambda}(\varphi)([1:y:t^{-1}])\notag \\&=(t^2+t^2\abs{y}^2+1)^{-\frac{\lambda_1+\rho_1}{2}}\gamma_{\lambda}(\varphi)([t:ty:1]),
\end{align}
so that the limit is given by
\begin{equation*}
\lim\limits_{t\to 0}\abs{t}^{-\lambda_1-\rho_1} \varphi(y,t^{-1})=\gamma_{\lambda}(\varphi)([0:\ldots:0:1]).
\end{equation*}
By the same argument the limits of derivatives exist.
\end{proof}

\subsection{Symmetry breaking operators in the compact picture}
We now find expressions for all symmetry breaking operators in the compact picture. For this consider the equivariant isomorphism from the non-compact to the compact picture
\begin{equation*}
\begin{tikzcd}
\gamma_{\lambda}: I_{\lambda}(\Nbar) \arrow[r,"\sim"] & I(K),
\end{tikzcd}
\end{equation*}
where we identify $I(K)$ with $C^\infty(K/(K \cap M))=C^\infty(\RP^{{n}})$.
By the Iwasawa decomposition $G=KP$ we write $\bar{n}_x=kp$ with $k\in K$ and $p\in P$. Then
\begin{equation*}
 k=(1+\abs{x}^2)^{-\frac{1}{2}}\left(
\begin{array}{c|c}
  1 & * \\ \hline 
   &  \raisebox{-15pt}{{\Huge\mbox{{$*$}}}} \\[-2.7ex]
  x & \\[+0.1ex]
    &
\end{array}
\right)\in K, \qquad
 p=(1+\abs{x}^2)^{\frac{1}{2}}\left(
\begin{array}{c|c}
  1 & * \\ \hline 
   &  \raisebox{-15pt}{{\Huge\mbox{{$*$}}}} \\[-2.7ex]
  0 & \\[+0.1ex]
    &
\end{array}
\right) \in P.
\end{equation*}
A function $\varphi \in I_{\lambda}(\bar{N})$ can be extended to a smooth function $\widetilde{\varphi}$ on $G$ which is right-equivariant under the action of $P$. As such it can be restricted to $K\subseteq G$ which gives
\begin{equation*}
\widetilde{\varphi}(\bar{n}_x)=(1+\abs{x}^2)^{-\frac{1}{2}(\lambda_1+\rho_1)}\widetilde{\varphi}(k).
\end{equation*}
This implies for all $v_1 \neq 0$
\begin{equation*}
\gamma_{\lambda}(\varphi)([v_1:\ldots :v_{n+1}])=\left(\frac{\abs{v}}{\abs{v_1}}\right)^{\lambda_1+\rho_1}\varphi\left(\bar{n}_{v_1^{-1}(v_2,\ldots,v_{n+1})}\right).
\end{equation*}
We denote the corresponding map $J_\nu(\Nbar')\to J(K')$ between the non-compact and the compact picture of $\tau_{\nu}$ by $\gamma'_{\nu}$.

\begin{theorem}
\label{theorem:compact_operators}
The operators $A_{\lambda,\nu},A^{(1)}_{\lambda,\nu}, A^{(2)}_{\lambda,\nu}, B_{\lambda,\nu}, C_{\lambda,\nu} \in \Hom_{G'}(I(K),J(K'))$ corresponding to the distributions $u^{A}_{\lambda,\nu},u^{A,1}_{\lambda,\nu}, u^{A,2}_{\lambda,\nu}, u^{B}_{\lambda,\nu}, u^C_{\lambda,\nu}$ are given as follows:
\begin{enumerate}[label=(\roman{*}), ref=\thetheorem(\roman*)]
\item For $(\lambda_1,\nu_1)\in\C^2$
\begin{equation*}
A_{\lambda,\nu}\varphi([y])=\int_{S^{1}} \frac{\abs{\omega_1}^{\nu_1+\rho'_1-1}\abs{\omega_2}^{\lambda_1+\rho_1-\nu_1-\rho'_1-1}}{2\Gamma(\frac{\nu_1+\rho'_1}{2})\Gamma(\frac{\lambda_1+\rho_1-\nu_1-\rho'_1}{2})}\varphi\left(\left[\omega_1\frac{y}{\abs{y}}:\omega_2\right]\right)d(\omega_1,\omega_2).
\end{equation*}
\item For $\nu_1=-\rho'_1-2j ,j\in \Z_{\geq 0}$ \begin{equation*}
A^{(1)}_{\lambda,\nu}\varphi[(y])=\abs{y}^{\nu_1+\rho_1'} \frac{d^{2j}}{dt^{2j}}\biggr|_{t=0}\abs{(ty,1)}^{-\lambda_1-\rho_1}\varphi([ty:1]).
\end{equation*}
\item For $\lambda_1+\rho_1-\nu_1-\rho'_1=-2N\in -2 \Z_{\geq 0}$ \begin{equation*}
A^{(2)}_{\lambda,\nu}\varphi([y])=\abs{y}^{\nu_1+\rho_1'}\frac{d^{2N}}{dt^{2N}}\biggr|_{t=0}\abs{(y,t)}^{-\lambda_1-\rho_1}\varphi([y:t]).
\end{equation*}
\item For $\nu_1=-\rho_1'$
\begin{equation*}
B_{\lambda,\nu}\varphi([y])=\frac{1}{2\Gamma(\frac{\lambda_1-\rho_1+1}{2})}\int_{ S^{{n}}}\abs{\omega_{n+1}}^{\lambda_1-\rho_1}\varphi([\omega])d\omega.
\end{equation*}
\item For $n=2$ and $(\lambda_1,\nu_1)\in \C^2$
\begin{equation*}
C_{\lambda,\nu}\varphi([y])=\int_{ S^{{2}}}\frac{\left|\frac{\omega_2y_1}{\abs{y}}-\frac{\omega_1 y_2}{\abs{y}}\right|^{-\nu_1-\rho'_1}\abs{\omega_{3}}^{\lambda_1+\rho_1+\nu_1-\rho'_1-1}}{2\Gamma(\frac{-\nu_1-\rho'_1+1}{2}){\Gamma(\frac{\lambda_1+\rho_1+\nu_1-\rho'_1}{2})}}\varphi([\omega])d\omega.
\end{equation*}
\end{enumerate}
\end{theorem}
\begin{proof}
Let $\mathcal{A} \in \Hom_{G'}(I_{\lambda}(\Nbar),J_\nu(\Nbar'))$. Then the corresponding operator $A$ in the compact picture is given in the following way:
\begin{equation*}
A=\gamma'_{\nu}\circ \mathcal{A} \circ \gamma_{\lambda}^{-1}
\end{equation*}
Hence for $y=[y_1:y']$ with $y_1 \neq 0$
\begin{align}
\label{eq:compact_operator_calculation}
A\varphi([y])=\left(\frac{\abs{y_1}}{\abs{y}}\right)^{-\nu_1-\rho'_1}\mathcal{A}\left(\gamma_{\lambda}^{-1}(\varphi)\right)(y_1^{-1}y').
\end{align}
For the first operator \eqref{eq:compact_operator_calculation} is
\begin{align*}
A_{\lambda,\nu}\varphi([y])&=\abs{y_1}{\abs{y}}^{\nu_1+\rho'_1}\int_{\R}\frac{\abs{y_1x}^{(\lambda_1+\rho_1-\nu_1-\rho'_1)-1}\sqrt{\abs{y}^2+y_1^2x^2}^{-\lambda_1-\rho_1}}{\Gamma(\frac{\nu_1+\rho'_1}{2})\Gamma(\frac{\lambda_1+\rho_1-\nu_1-\rho'_1}{2})}\varphi([y:y_1x])d{x} \\
&={\abs{y}}^{\nu_1+\rho'_1}\int_{\R}\frac{\abs{z}^{(\lambda_1+\rho_1-\nu_1-\rho'_1)-1}\sqrt{\abs{y}^2+z^2}^{-\lambda_1-\rho_1}}{\Gamma(\frac{\nu_1+\rho'_1}{2})\Gamma(\frac{\lambda_1+\rho_1-\nu_1-\rho'_1}{2})}\varphi([y:z])d{z} \\
&=\int_{-1}^{1}\frac{\sqrt{1-r^2}^{\nu_1+\rho'_1-2}\abs{r}^{\lambda_1+\rho_1-\nu_1-\rho'_1-1}}{\Gamma(\frac{\nu_1+\rho'_1}{2})\Gamma(\frac{\lambda_1+\rho_1-\nu_1-\rho'_1}{2})}\varphi\left(\left[\frac{\sqrt{1-r^2}}{\abs{y}}y:r\right]\right)d{r} \\
&=\frac{1}{2}\int_{S^1}\frac{\abs{\omega_1}^{\nu_1+\rho'_1-1}\abs{\omega_2}^{\lambda_1+\rho_1-\nu_1-\rho'_1-1}}{\Gamma(\frac{\nu_1+\rho'_1}{2})\Gamma(\frac{\lambda_1+\rho_1-\nu_1-\rho'_1}{2})}\varphi\left(\left[\frac{\omega_1}{\abs{y}}y:\omega_2\right]\right)d(\omega_1,\omega_2).
\end{align*}
(ii) follows from \eqref{eq:limit_existence} and \eqref{eq:compact_operator_calculation} and (iii) follows from \eqref{eq:compact_operator_calculation}. (iv) and (v) follows  from \eqref{eq:compact_operator_calculation} by pulling back the integral with the map $S^{{n}}_{\omega_1 \neq 0} \to \R^{{n}}$, $(\omega_1,\ldots,\omega_{n+1})\mapsto \omega_1^{-1}(\omega_2,\ldots,\omega_{n+1})$.
\end{proof}
Theorem~\ref{theorem:compact_operators} concludes the proof of Theorem~\ref{theorem:B} and together with Corollary~\ref{corollary:residues_distr} it implies Theorem~\ref{theorem:C}. 
\begin{remark}
For $z \in \C, m\in \Z_{\geq 0}$ let $(z)_m$ denote the Pochhammer symbol given by $(z)_m=\frac{\Gamma(z+m)}{\Gamma(z)}$.
Let $\lambda_1+\rho_1-\nu_1-\rho'_1=-2N$, $N\in \Z_{\geq 0}$. By the product rule the operator $A^{(2)}_{\lambda,\nu}$ can be written as
\begin{equation*}
A^{(2)}_{\lambda,\nu}=\rest \circ \sum_{k=0}^{N} (-1)^{k}\frac{(2N)!}{k!(2N-2k)!}\left(\frac{\lambda_1+\rho_1}{2}\right)_k 
\partial_{\rm n}^{2N-2k},
\end{equation*}
where $\partial_{\rm n}$ denotes the normal vector field with respect to the submanifold $\RP^{n-1}\subseteq\RP^n$ and $\rest:C^\infty(\RP^n)\to C^\infty(\RP^{n-1})$ the restriction map. This expression should be compared to Juhl's family of conformally covariant differential operators (see \cite[Chapter 5]{juhl_2009} and also \cite[10.2.]{kobayashi_speh_2015}). In contrast to Juhl's operators, the family $A^{(2)}_{\lambda,\nu}$ does not involve derivatives tangential to the submanifold $\RP^{n-1}\subseteq\RP^n$, but only normal derivatives.
\end{remark}

\begin{remark}
The operator $B_{\lambda,\nu}$ for $\lambda_1\neq \rho_1-2j-1$, $j \in  \Z_{\geq 0}$ defines the only family of regular symmetry breaking operators, in the sense that the support of its distribution kernel contains an open $P'$-orbit. All other symmetry breaking operators are singular. More precisely, the support of the distribution kernel $K^A\in(\mathcal{D}'(G/P,\mathcal{V}_{{\bf 1},\lambda})\otimes W_{{\bf 1},\nu})^{P'}$ of a symmetry breaking operator $A$ is given by
\begin{align*}
&\operatorname{supp} K^{A_{\lambda,\nu}}=\begin{cases}
\bar{\mathcal{O}}_3 & \text{for } (\lambda_1,\nu_1) \notin L, \nu_1\notin -\rho_1'-2\Z_{\geq 0}, \lambda_1+\rho_1-\nu_1-\rho'_1 \notin -2\Z_{\geq 0},\\
\mathcal{O}_2  & \text{for } (\lambda_1,\nu_1) \notin L, \lambda_1+\rho_1-\nu_1-\rho'_1 \in -2\Z_{\geq 0}, \\
\mathcal{O}_1 & \text{for } (\lambda_1,\nu_1) \notin L, \nu_1\in -\rho_1'-2\Z_{\geq 0}, \\
\{0\} & \text{otherwise,}
\end{cases} \\
&\operatorname{supp} K^{A^{(1)}_{\lambda,\nu}}=\mathcal{O}_1, \\
&\operatorname{supp} K^{A^{(2)}_{\lambda,\nu}}=\mathcal{O}_2, \\
&\operatorname{supp} K^{B_{\lambda,\nu}}=\begin{cases}
G/P & \text{for }  \lambda_1 \neq \rho_1-1-2\Z_{\geq 0}, \nu_1=-\rho'_1, \\
\bar{\mathcal{O}}_4 & \text{for } \lambda_1 = \rho_1-1-2\Z_{\geq 0}, \nu_1=-\rho'_1.
\end{cases} \\
\intertext{For $n=2$}
&\operatorname{supp}K^{C_{\lambda,\nu}}=
	\begin{cases}
		G/P & \text{for } -\nu_1-\rho'_1+1 \notin -2\Z_{\geq 0}, \lambda_1+\rho_1+\nu_1-\rho'_1 \notin -2\Z_{\geq 0}, \\
		\bar{\mathcal{O}_4 }& \text{for } -\nu_1-\rho'_1+1 \notin -2\Z_{\geq 0}, \lambda_1+\rho_1+\nu_1-\rho'_1 \in -2\Z_{\geq 0}, \\		
		\bar{\mathcal{O}_3} & \text{for } -\nu_1-\rho'_1+1 \in -2\Z_{\geq 0}, \lambda_1+\rho_1+\nu_1-\rho'_1 \notin -2\Z_{\geq 0}, \\
		\mathcal{O}_1 & \text{for } -\nu_1-\rho'_1+1 \in -2\Z_{\geq 0}, \lambda_1+\rho_1+\nu_1-\rho'_1 \in -2\Z_{\geq 0}.
	\end{cases}
\end{align*}
In \cite[Corollary 3.6]{kobayashi_speh_2015} a necessary condition for the occurrence of regular symmetry breaking operators was given. This condition is in our case satisfied for $n\geq 3$ if and only if $\nu_1=-\rho'_1$ and $\lambda_2-\rho_2=\nu_2+\rho'_2$ and for $n=2$ if and only if $\lambda_2-\rho_2-\nu_2-\rho'_2=\nu_1+\rho'_1$.
\end{remark}

\subsection{$K$-type analysis}

The space $I(K)$ can be identified with the space of smooth functions on the real projective space $\RP^{{n}}$. Let $\Ical:=C^\infty(\RP^{{n}})_K$ be the space of $K$-finite vectors, which decomposes following \cite[Chapter IX]{vilenkine_1968} into irreducible $K$-modules
\begin{equation}
\label{EQ:Kdecomp}
\Ical=\bigoplus_{\mathclap{\alpha \in \Z_{\geq 0}}} \;\Ical(\alpha)
\end{equation}
where $\Ical(\alpha)\cong\Hcal^{2\alpha}(\R^{n+1})$, the space of harmonic, homogeneous polynomials on $\R^{{n+1}}$ of degree $2\alpha$. In particular, we have $\mathcal{I}=\C[\RP^n]$, the space of regular functions on the projective variety $\RP^n$ in the sense of algebraic geometry. Similarly we decompose
\begin{equation}
\label{EQ:Kdecomp2}
\Jcal:=C^\infty(\RP^{{n-1}})_{K'}=\C[\RP^{n-1}]=\bigoplus_{\mathclap{\alpha' \in \Z_{\geq 0}}} \;\Jcal(\alpha'), \qquad \Jcal(\alpha')\cong\Hcal^{2\alpha'}(\R^{{n}}).
\end{equation}
\begin{remark}
\label{REM:SR}
Since homogenous functions on $\R^{n+1}$ are uniquely determined by their values on the unit sphere $S^n$, the restriction from $\R^{n+1}$ to $S^n$ defines an isomorphism $\Hcal^{2\alpha}(\R^{n+1})\to C^\infty(S^n)$ onto a subspace $\Hcal^{2\alpha}(S^{{n}})$ of $C^\infty(S^n)$. The inverse is given by writing for $\phi \in \Hcal^{2\alpha}(S^{{n}})$: $\phi(x)=|x|^{2\alpha}\phi(\frac{x}{\abs{x}}) \in \Hcal^{2\alpha}(\R^{n+1})$. Since $\Orm({n+1})$ leaves $|x|$ invariant, these are isomorphisms of $\Orm({n+1})$-representations.
\end{remark}
Following \eqref{EQ:SphHarm1} we have
\begin{equation}
\Hcal^{2\alpha}(\R^{n+1})|_{K'} \cong \; \bigoplus_{\mathclap{0 \leq k \leq 2\alpha }} \; \Hcal^{k}(\R^{{n}})\text{.}
\end{equation}
Therefore every $K$-type $\Ical(\alpha)$ contains all $K'$-types $\Jcal(\alpha')$ with $0\leq\alpha'\leq\alpha$. The explicit embedding $I_{\alpha' \rightarrow \alpha}: \Hcal^{2\alpha'}(\R^{{n}}) \rightarrow \Hcal^{2\alpha}(\R^{{n+1}})$ is given in \eqref{EQ:SphHarm3} and we put 
\begin{equation}
\label{EQ:K,K'rel}
\Ical(\alpha, \alpha')=I_{\alpha'\rightarrow \alpha}(\Hcal^{2\alpha'}(\R^{{n}}))\subseteq\Ical(\alpha).
\end{equation}
Further let $R_{\alpha,\alpha'}: \Ical(\alpha,\alpha') \rightarrow \Jcal(\alpha')$ be the map restricting functions on $K$ to $K'$ which is an isomorphism of $K'$-modules.
\begin{figure}[h]
\centering
\label{fig:K'types}
\setlength{\unitlength}{6pt}
\begin{picture}(31,31)
\thicklines
\put(0,0){\vector(1,0){30}}
\put(0,0){\vector(0,1){30}}

\multiput(0,0)(5,0){6}{\circle*{1}}
\multiput(5,5)(5,0){5}{\circle*{1}}
\multiput(10,10)(5,0){4}{\circle*{1}}
\multiput(15,15)(5,0){3}{\circle*{1}}
\multiput(20,20)(5,0){2}{\circle*{1}}
\multiput(25,25)(5,0){1}{\circle*{1}}
\put(28,1){$\alpha$}
\put(1,28){$\alpha'$}
\put(-1.5,-2){$0$}
\put(4.6,-2){$1$}
\put(9.6,-2){$2$}
\put(14.6,-2){$3$}
\put(19.6,-2){$4$}
\put(24.6,-2){$5$}
\put(-1.5,4.6){$1$}
\put(-1.5,9.6){$2$}
\put(-1.5,14.6){$3$}
\put(-1.5,19.6){$4$}
\put(-1.5,24.6){$5$}
\end{picture}
\vspace{0.3cm}
\caption{$K'$-types $\Ical(\alpha,\alpha')$ that occur in $\Jcal$.}
\end{figure}
\begin{lemma}
\label{lemma:spectral_functions}
Let $\varphi \in \Ical(\alpha,\alpha')$. Then
\begin{multline}
\label{eq:spectral_function_1}
A_{\lambda,\nu}\varphi=
\frac{(\frac{\nu_1+\rho'_1}{2})_{\alpha'}}{\Gamma(\frac{\lambda_1+\rho_1}{2}+\alpha')} \\ \cdot{_3F_2}\left(\alpha'-\alpha,\frac{\lambda_1+\rho_1-\nu_1-\rho'_1}{2},\frac{{n-1}}{2}+\alpha+\alpha';\frac{\lambda_1+\rho_1}{2}+\alpha',\frac{1}{2};1\right) R_{\alpha,\alpha'}(\varphi)
\end{multline}
and
\begin{align}
\label{eq:spectral_function_4}
B_{\lambda,\nu}\varphi=\delta_{0,\alpha'}\frac{\pi^{\frac{{n}}{2}}(\frac{\rho_1-\lambda_1}{2})_{\alpha}(\frac{{n}}{2})_{\alpha}}{(\frac{1}{2})_{\alpha}\Gamma(\frac{\lambda_1+\rho_1}{2}+\alpha)}R_{\alpha,\alpha'}(\varphi).
\end{align}
For $n=2$ we further have
\begin{multline}
\label{eq:spectral_function_C}
C_{\lambda,\nu}\varphi =\frac{\sqrt{\pi}(\frac{\nu_1+\rho'_1}{2})_{\alpha'}}{\Gamma(\frac{\lambda_1+\rho_1}{2}+\alpha')}
\\ \cdot{_3F_2}\left(\alpha'-\alpha,\frac{\lambda_1+\rho_1+\nu_1-\rho'_1}{2},\frac{{1}}{2}+\alpha+\alpha';\frac{\lambda_1+\rho_1}{2}+\alpha',\frac{1}{2};1\right) R_{\alpha,\alpha'}(\varphi).
\end{multline}
\end{lemma}
Here ${_3F_2}(a_1,a_2,a_3;b_1,b_2;z)$ denotes the generalized hypergeometric function. 
\begin{proof}
\eqref{eq:spectral_function_1} follows directly from \cite[7.319]{Tables_Integrals_Prodicts}.
For \eqref{eq:spectral_function_4}  let $ \varphi=I_{\alpha'\to\alpha}(\phi), \phi \in \mathcal{J}(\alpha') $. We calculate using the coordinates $(\sqrt{1-r}\omega,\sqrt{r})$
\begin{equation*}
B_{\lambda,\nu}\varphi=\int_{S^{{n-1}}}\phi(\omega)d\omega \cdot \int_{0}^{1} \frac{r^{\frac{\lambda_1-\rho_1-1}{2}}(1-r)^{\frac{{n-2}}{2}+\alpha'}}{2\Gamma(\frac{\lambda_1-\rho_1+1}{2})}C^{\frac{{n-1}}{2}+2\alpha'}_{2(\alpha-\alpha')}(\sqrt{r}) dr,
\end{equation*}
which vanishes if $\alpha'\neq 0$. For $\alpha'=0$, \cite[7.319]{Tables_Integrals_Prodicts} and \cite[Theorem 2.2.6]{vilenkine_1968} imply the formula.
For \eqref{eq:spectral_function_C} we calculate using the same coordinates for $y \in S^1$
\begin{multline*}
C_{\lambda,\nu}\varphi([y])=\int_{0}^{1}\frac{\abs{r}^{\frac{\lambda_1+\rho_1+\nu_1-\rho'_1}{2}-1}(1-r)^{\frac{-\nu_1-\rho'_1}{2}+\alpha'}}{\Gamma(\frac{\lambda_1+\rho_1+\nu_1-\rho'_1}{2})}C^{\frac{1}{2}+2\alpha'}_{2(\alpha-\alpha')}(\sqrt{r})dr \\ \cdot \int_{S^1}\frac{\abs{\omega_2y_1-\omega_1y_2}^{-\nu_1-\rho'_1}}{2\Gamma(\frac{-\nu_1-\rho'_1+1}{2})}\phi(\omega)d\omega.
\end{multline*}
The first integral can again be evaluated with \cite[7.319]{Tables_Integrals_Prodicts} and yields
\begin{equation*}
C^{\frac{1}{2}+2\alpha'}_{2(\alpha-\alpha')}(0)\frac{\Gamma(\frac{-\nu_1+\rho_1}{2}+\alpha')}{\Gamma(\frac{\lambda_1+\rho_1}{2}+\alpha')}
{_3F_2}\left(\alpha'-\alpha,\frac{\lambda_1+\rho_1+\nu_1-\rho'_1}{2},\frac{{1}}{2}+\alpha+\alpha';\frac{\lambda_1+\rho_1}{2}+\alpha',\frac{1}{2};1\right).
\end{equation*}
The right $K'$-invariance of $\phi$ implies that the second integral is equal to
\begin{align*}
\phi(y)\int_{S^1}\frac{\abs{\omega_1}^{-\nu_1-\rho'_1}}{2\Gamma(\frac{-\nu_1-\rho'_1+1}{2})}\phi(\omega)d\omega ,
\end{align*}
which is by \cite[3.631 (8)]{Tables_Integrals_Prodicts} equal to
\begin{equation*}
\frac{(-1)^{\alpha'}2^{\nu_1+\rho'_1}\pi \Gamma(-\nu_1+\rho'_1)}{\Gamma(\frac{-\nu_1+\rho_1}{2}+\alpha')\Gamma(\frac{-\nu_1+\rho'_1}{2}-\alpha')\Gamma(\frac{-\nu_1-\rho'_1+1}{2})(-\nu_1-\rho'_1+1)}=\frac{\sqrt{\pi}(\frac{\nu_1+\rho'_1}{2})_{\alpha'}}{\Gamma(\frac{-\nu_1+\rho_1}{2}+\alpha')}. \qedhere
\end{equation*}
\end{proof}
Together with the residue formulas of Theorem~\ref{theorem:C}, the identity \eqref{eq:spectral_function_1} also yields the explicit action of $A^{(1)}_{\lambda,\nu}$ and $A^{(2)}_{\lambda,\nu}$ on every $K'$-type $\Ical(\alpha,\alpha')$. In particular, the following formulas for the action on the spherical vector follow:
\begin{corollary}
\label{cor:spherical_vectors}
For the spherical vectors $\mathbf{1}_\lambda\in I(K)$ and $\mathbf{1}_{\nu}\in J(K')$ we have
\begin{align*}
&A_{\lambda,\nu}\mathbf{1}_{\lambda} =\frac{1}{\Gamma(\frac{\lambda_1+\rho_1}{2})}\mathbf{1}_{\nu},
&&B_{\lambda,\nu}\mathbf{1}_{\lambda} =\frac{\pi^{\frac{{n}}{2}}}{\Gamma(\frac{\lambda_1+\rho_1}{2})}\mathbf{1}_{\nu},
\end{align*}
and for $(\lambda_1,\nu_1)=(-\rho_1-2i,-\rho'_1-2j) \in L$
\begin{align*}
&A^{(1)}_{\lambda,\nu}\mathbf{1}_{\lambda} =\frac{i!(2j)!}{j!(i-j)!}\mathbf{1}_{\nu}, 
&&A^{(2)}_{\lambda,\nu}\mathbf{1}_{\lambda} 
=\frac{i!(2i-2j)!}{j!(i-j)!}\mathbf{1}_{\nu}.
\end{align*}
For $n=2$ we have
\begin{equation*}
C_{\lambda,\nu}\mathbf{1}_\lambda= \frac{\sqrt{\pi}}{\Gamma(\frac{\lambda_1+\rho_1}{2})}\mathbf{1}_\nu.
\end{equation*}
\end{corollary}

\section{Symmetry breaking operators between Harish-Chandra modules}\label{sec:SBOsHCmodules}

We return to the general setting as described in Section~\ref{sec:general_smooth}. In particular, we assume the multiplicity-freeness properties \eqref{eq:MultFreeProperties}.

Denote the center of $G$ by $Z$, then $A=A_\ss A_\subz$ with $A_\ss=A\cap \exp([\g,\g])$ and $A_{\rm z}=A\cap Z$. We write $\zfrak$, $\mathfrak{a}_\ss$ and $\mathfrak{a}_\subz$ for the Lie algebras of $Z$, $A_\ss$ and $A_\subz$. The analogous notation is used for $G'$ which explains the use of $Z'$, $A'_{\ss}$, $A'_\subz$, $\zfrak'$, $\mathfrak{a}'_\ss$ and $\mathfrak{a}'_\subz$.

Write $\lambda=\lambda_\ss+\lambda_\subz \in \mathfrak{a}_{\C}^*$ such that $\lambda_\ss$ vanishes on $\mathfrak{a}_\subz$ and $\lambda_\subz$ vanishes on $\mathfrak{a}_\ss$. Similarly, we write $\nu=\nu_\ss+\nu_\subz$ such that $\nu_\ss$ vanishes on $\mathfrak{a}'_\subz$ and $\nu_\subz$ vanishes on $\mathfrak{a}'_\ss$. In the following we view $\lambda_\subz$ and $\nu_\subz$ as linear forms on the whole Lie algebras $\g=\zfrak \oplus \g_\ss$ resp. $\g'=\zfrak' \oplus \g'_\ss$ vanishing on $\g_\ss$ and $\mathfrak{k}$ resp. $\g'_\ss$ and $\mathfrak{k}'$, where $\g_\ss=[\g,\g]$, $\g'_\ss=[\g',\g']$.

\subsection{The spectrum generating operator}
Assume that the parabolic subgroup $P\subseteq G$ is maximal, then $\mathfrak{a}_\ss$ is one-dimensional. Following \cite[9.4]{Humph} we can choose a non-trivial element $H \in \mathfrak{a}_\ss$ such that the positive $(\g,\mathfrak{a})$-roots form an unbroken string $\varepsilon, 2 \varepsilon, \ldots, q\varepsilon$ with $q \in \N$ and $\varepsilon(H)=1$. 
We have
\begin{equation*}
\rho=\frac{1}{2}\left(\sum_{j=1}^{q}j\dim \g_{j\nu}\right)\varepsilon,
\end{equation*}
where $\g_{j\varepsilon}$ is the $j\varepsilon$-eigenspace of $\operatorname{ad}(H)$. We can choose an invariant non-degenerate symmetric bilinear form $B$ on $\g$ with $B(H,H)=1$ and 
consider Casimir elements relative to $B$:
Let $\{X_i\}$ be a basis of a reductive subalgebra $\mathfrak{h} \subseteq \g$ with the property $|B(X_i,X_k)|=\delta_{i,k}$. Then
\begin{equation*}
\mathrm{Cas}_\mathfrak{h}=\sum_{i} B(X_i,X_i)X_i^2
\end{equation*}
is a central second order element of the universal enveloping algebra of $\mathfrak{h}$.
In the same way we can consider Casimir elements corresponding to subspaces $\mathfrak{k}_j=\mathfrak{k}\cap (\g_{j\varepsilon}\oplus \g_{-j\varepsilon})$: Let $\{X_{i}^j\}$ be a basis of $\mathfrak{k}_j$ with $B(X_i^j,X_k^j)=-\delta_{i,k}$. Then we define
\begin{equation*}
\mathrm{Cas}_j=-\sum_{i}(X_i^j)^2,
\end{equation*}
which is a second order element of the universal enveloping algebra of $\mathfrak{k}$. Note that by \cite[Remark 2.4]{BOO_1996} the elements $\mathrm{Cas}_j$ can be written as rational linear combinations of Casimir elements of subalgebras of $\mathfrak{k}$, even though not all $\mathfrak{k}_j$ might be subalgebras. The spectrum generating operator $\Pcal$ is
\begin{equation*}
\mathcal{P}=\sum_{j=1}^{q}j^{-1}\mathrm{Cas}_j.
\end{equation*}
Since the right-regular action of $\Pcal$ commutes with the left-regular action of $K$ we have that $\mathcal{R}_{\Pcal}$ is a linear transformation $\sigma_{\alpha}$ on the isotypic component $\Ical(\alpha)$. Since $\Ical(\alpha)$ is irreducible, this transformation is scalar.

In the same way we can choose $H' \in \mathfrak{a}'_\ss$ and $\varepsilon' \in \mathfrak{a}'_{\C}$ such that the $(\g',\mathfrak{a}')$-roots form an unbroken string and an invariant non-degenerate bilinear form $B'$ on $\g'$ with $\varepsilon'(H')=B'(H',H')=1$. We denote by $\Pcal'$ the spectrum generating operator for $G'$ whose right action on the isotypic components $\Jcal(\alpha')$ is given by scalars $\sigma_{\alpha'}'$.

\subsection{Reduction to the cocycle}
For $\lambda\in\mathfrak{a}_\C^*$ we write $\lambda_\ss=\lambda(H)\in\C$, then $\lambda=\lambda_\ss\varepsilon+\lambda_\subz$ with $\lambda_\subz\in\mathfrak{z}_\C^*$. The analogous notation is used for $\nu=\nu_\ss\varepsilon'+\nu_\subz\in\mathfrak{a}_\C^*$.
We can extend $B$ to a non-degenerate symmetric $\C$-bilinear form on the complexification $\g_{\C}$. We define for $X \in \g_{\C}$
\begin{equation*}
\omega(X):K\to\C, \quad \omega(X)(k)=B(\mathrm{Ad}(k^{-1})X,H) \qquad (k\in K).
\end{equation*}
Since $\omega(\mathrm{Ad}(h)X)(k)=\omega(X)(h^{-1}k)$, this defines a $K$-equivariant map $\omega$ from $\g_{\C}$ to $I_\mathbf{1}(K)=C^\infty(K/(K\cap M))$, where $\mathbf{1}$ is the trivial $(K \cap M)$-representation. The map $\omega$ is called a \textit{cocycle}. Since $\mathfrak{s}_{\C}$ and $\mathfrak{k}_{\C}$ are $B$-orthogonal and $H \in \mathfrak{s}_{\C}$, the map $\omega$ vanishes on $\mathfrak{k}_{\C}$ and since roots vanish on the center, the map $\omega$ vanishes on $\zfrak_{\C}$. By $m(\omega(X))$ we denote the multiplication operator $\Ical \rightarrow \Ical$ with $\omega(X)$, and we write $\proj_{\Ical(\alpha)}:\Ical\to\Ical(\alpha)$ for the projection onto the direct summand $\Ical(\alpha)$. For $K$-types $\alpha, \beta$ with $\Ical(\alpha),\Ical(\beta)\neq 0$ we define for all $X \in \g_{\C}$
\begin{equation*}
\omega_{\alpha}^\beta(X)=\proj_{\Ical(\beta)}\circ m(\omega(X))|_{\Ical(\alpha)}.
\end{equation*}
The reduction to the cocycle is stated in the following theorem (see \cite[Corollary 2.6]{BOO_1996}). It lets us express the derived representation $\pi_{\xi,\lambda}$ in terms of the cocycle, the linear transformations $\sigma_{\alpha}$ and $\lambda_\subz$.
\begin{theorem}
\label{Thm:cocycle}
For $X \in \mathfrak{s}_{\C}$ and $K$-types $\alpha,\beta$ with $\Ical(\alpha),\Ical(\beta)\neq0$ we have
\begin{equation}
\label{EQ:cocycleG}
\proj_{\Ical(\beta)} \circ \pi_{\xi,\lambda}(X)|_{\Ical(\alpha)}
=\frac{1}{2}(\sigma_{\beta} -
\sigma_{\alpha}+2\lambda_\ss)\omega_{\alpha}^\beta(X)+\delta_{\alpha,\beta}\lambda_\subz(X).
\end{equation}
\end{theorem}
Following \cite[Remark 2.8]{BOO_1996} it is evident that $\omega_{\alpha}^\beta$ is non-trivial if and only if $\omega_{\beta}^\alpha$ is non-trivial. Hence we write $\alpha \leftrightarrow \beta$ for $K$-types $\alpha,\beta$ with $\Ical(\alpha), \Ical(\beta)\neq  0$ and $\omega_{\alpha}^\beta \neq0$ i.e. everytime we can reach the $K$-type $\alpha$ by the multiplication with the cocycle restricted to $\Ical(\beta)$ and the other way around.
Let $\mathcal{U}\subseteq \Ical$ be a $(\g,K)$-submodule of $(\pi_{\xi,\lambda})_{\HC}$. We write $\mathcal{U}(\alpha)$ for the $\alpha$-isotypic component of $\mathcal{U}$. That means $\mathcal{U}(\alpha)=\Ical(\alpha)$ whenever $\mathcal{U}(\alpha)\neq 0$ and $\mathcal{U}(\alpha)=\{0\}$ otherwise.
Then
\begin{equation*}
\proj_{\mathcal{U}(\beta)} \circ \pi_{\xi,\lambda}(X)|_{\mathcal{U}(\alpha)}=\frac{1}{2}(\sigma_{\beta} -
\sigma_{\alpha}+2\lambda_\ss)\omega_{\alpha}^\beta(X)+\delta_{\alpha,\beta}\lambda_\subz(X),
\end{equation*}
for all $X \in \mathfrak{s}_{\C}$ and
whenever $\alpha$ and $\beta$ are $K$-types of $\mathcal{U}$. If $\alpha$ or $\beta$ is not a $K$-type of $\mathcal{U}$ but a $K$-type of $\Ical$ we have $\proj_{\mathcal{U}(\beta)} \circ \pi_{\xi,\lambda}(X)|_{\mathcal{U}(\alpha)}=0$.
Let $\Ical / \mathcal{U}$ be the quotient of $(\pi_{\xi,\lambda})_{\HC}$.
We can identify 
\begin{equation*}
\faktor{\Ical}{\mathcal{U}} \cong_K \bigoplus_{\mathclap{\substack{\alpha \\ \Ical(\alpha) \neq 0,\\ \mathcal{U}(\alpha)=0}}} \; \Ical(\alpha)=\mathcal{U}^\perp.
\end{equation*}
This identification is an isomorphism of $(\g,K)$-modules if $\mathcal{U}^\perp$ is endowed with the $\g$-action given by $\proj_{\mathcal{U}^\perp}\circ\pi_{\xi,\lambda}|_{\mathcal{U}^\perp}$.
Then $\mathcal{U}^\perp$ has isotypic components $\mathcal{U}^\perp(\alpha)=\Ical(\alpha)\neq 0$ whenever $\alpha$ is not a $K$-type of $\mathcal{U}$ and  $\mathcal{U}^\perp(\alpha)=0$ otherwise. Hence for a quotient $\mathcal{U}^\perp$ we have for $X \in \mathfrak{s}_{\C}$
\begin{equation*}
\proj_{\mathcal{U}^\perp(\beta)} \circ \pi_{\xi,\lambda}(X)|_{\mathcal{U}^\perp(\alpha)}=\frac{1}{2}(\sigma_{\beta} -
\sigma_{\alpha}+2\lambda_\ss)\omega_{\alpha}^\beta(X)+\delta_{\alpha,\beta}\lambda_\subz(X),
\end{equation*}
if $\mathcal{U}^\perp(\alpha),\mathcal{U}^\perp(\beta)\neq 0$ and $\proj_{\mathcal{U}^\perp(\beta)} \circ \pi_{\xi,\lambda}(X)|_{\mathcal{U}^\perp(\alpha)}=0$ otherwise. Since this formula is the same for submodules and quotients, we will treat submodules and quotients in a uniform way, identifying quotients with subspaces of $\Ical$. This implies that we can use the cocycle reduction in a modified way when working with a submodule or quotient $\mathcal{U}$:
\begin{lemma}
\label{lemma:cocycle_composition_factors}
For a submodule or quotient $\mathcal{U}$ of $(\pi_{\xi,\lambda})_{\HC}$ we have for $X \in \mathfrak{s}_{\C}$ and $K$-types $\alpha,\beta$ with $\mathcal{U}(\alpha),\mathcal{U}(\beta)\neq 0$
\begin{equation}
\label{EQ:cocycle_sub}
\proj_{\mathcal{U}(\beta)} \circ \pi_{\xi,\lambda}(X)|_{\mathcal{U}(\alpha)}
=
\frac{1}{2}(\sigma_{\beta} -
\sigma_{\alpha}+2\lambda_\ss)\omega_{\alpha}^\beta(X) +\delta_{\alpha,\beta}\lambda_\subz(X) .
\end{equation}
If $\mathcal{U}(\alpha)$ or $\mathcal{U}(\beta)$ is trivial then $\proj_{\mathcal{U}(\beta)} \circ \pi_{\xi,\lambda}(X)|_{\mathcal{U}(\alpha)}=0$.
\end{lemma}
Similarly we consider the cocycle $\omega'$ of $G'$ and define
\begin{equation*}
\omega'^{\beta'}_{\alpha'}(X)=\proj_{\Jcal(\beta')}\circ m(\omega'(X))|_{\Jcal(\alpha')} \hspace{1cm} X\in \mathfrak{g}'_{\C}, \text{ } \Jcal(\alpha'),\Jcal(\beta')\neq0,
\end{equation*}
and write $\alpha' \leftrightarrow \beta'$ whenever $\omega'^{\beta'}_{\alpha'}$ is non-trivial.
Theorem~\ref{Thm:cocycle} yields the identity
\begin{equation}
\label{EQ:cocycleG'}
\proj_{\Jcal(\beta')} \circ \tau_{\eta,\nu}(X)|_{\Jcal(\alpha')}
=\frac{1}{2}(\sigma'_{\beta} -
\sigma'_{\alpha'}+2\nu_\ss)\omega'^{\beta'}_{\alpha'}(X)+\delta_{\alpha',\beta'}\nu_\subz(X).
\end{equation}
Then Lemma~\ref{lemma:cocycle_composition_factors} also applies to composition factors of $\tau_{\eta,\nu}$.

\subsection{Symmetry breaking operators for Harish-Chandra modules}
\label{subsec:intertw}
Let $\mathcal{U}$ be a submodule or a quotient of $(\pi_{\xi,\lambda})_{\HC}$ and $\mathcal{U}'$ be a submodule or a quotient of $(\tau_{\eta,\nu})_{\HC}$. Then we have $\mathcal{U} \subseteq \Ical$ and $\mathcal{U}' \subseteq \Jcal$ under the identification above. An intertwining operator for $(\pi_{\xi,\lambda})_{\HC}$ and $(\tau_{\eta,\nu})_{\HC}$ is a linear map $T: \mathcal{U} \rightarrow \mathcal{U}'$ that is intertwining both as $\g'$ and as $K'$-module:
\begin{align}
\label{EQ:g_intertwining}
T\circ\proj_{\mathcal{U}}\circ \pi_{\xi,\lambda}(X)|_\mathcal{U}&=\proj_{\mathcal{U}'}\circ\tau_{\eta,\nu}(X)|_{\mathcal{U}'} \circ T  &&\forall\, X \in \g', \\ \label{EQ:K_intertwining}
T\circ \pi_{\xi,\lambda}(k)|_\mathcal{U}&=\tau_{\eta,\nu}(k)|_{\mathcal{U}'} \circ T  &&\forall\, k \in K'.
\end{align}
The space of these operators is denoted by $\Hom_{(\g',K')}(\mathcal{U}|_{(\g',K')},\mathcal{U}')$. The identity \eqref{EQ:K_intertwining} implies that for all $\alpha,\alpha'$ the map $T|_{\Ical(\alpha,\alpha')}$ is $K'$-equivariant from $\Ical(\alpha,\alpha')$ to $\Jcal(\alpha')$. Recall that both $\Ical(\alpha,\alpha')$ and $\Jcal(\alpha')$ are either trivial or irreducible, and we fixed a $K'$-equivariant isomorphism $R_{\alpha,\alpha'}:\Ical(\alpha,\alpha')\to\Jcal(\alpha')$ whenever $\Ical(\alpha,\alpha'),\Jcal(\alpha')\neq0$ (see \eqref{eq:FixRalphaalpha'}). Then Schur's Lemma implies that $T|_{\Ical(\alpha,\alpha')}=t_{\alpha,\alpha'} \cdot R_{\alpha,\alpha'}$ for numbers $t_{\alpha,\alpha'} \in \C$ whenever $0 \neq \Ical(\alpha,\alpha')\subseteq\mathcal{U}$ and $0 \neq \Jcal(\alpha')\subseteq \mathcal{U} $. If $\mathcal{U}(\alpha,\alpha')=0$ or $\mathcal{U}'(\alpha')=0$ we use the convention $t_{\alpha,\alpha'}=0$.

Restricting the multiplication with the cocycle $\omega$ to $\Ical(\alpha,\alpha')$ and projecting to $\Ical(\beta,\beta')$ yields maps
\begin{align*}
\omega_{\alpha,\alpha'}^{\beta,\beta'}:\mathfrak{s}'_{\C} \otimes \Ical(\alpha,\alpha') \rightarrow \Ical(\beta,\beta') , \hspace{0.2cm}
\omega_{\alpha,\alpha'}^{\beta,\beta'}(X)=\proj_{\Ical(\beta,\beta')} \circ m(\omega(X))|_{\Ical(\alpha,\alpha')},
\end{align*}
where $\omega_{\alpha,\alpha'}^{\beta,\beta'}(X) :\Ical(\alpha,\alpha')\rightarrow \Ical(\beta,\beta')$ is a linear map for all $X \in \mathfrak{s}'_{\C}$.
By \cite[Lemma 3.3]{MolOrs} it is evident that $\omega_{\alpha,\alpha'}^{\beta,\beta'} \neq 0$ if and only if $\omega^{\alpha,\alpha'}_{\beta,\beta'} \neq 0$. Hence we write $(\alpha,\alpha') \leftrightarrow (\beta,\beta')$ whenever $\omega_{\alpha,\alpha'}^{\beta,\beta'}\neq 0$.

We can use the cocycle reduction to prove the following theorem which is a generalization of \cite[Theorem 3.4]{MolOrs}.
\begin{theorem}
\label{Thm:formula}
Let $\mathcal{U}\subseteq \Ical$ be a submodule or quotient of $(\pi_{\xi,\lambda})_{\HC}$ and $\mathcal{U}' \subseteq \Jcal$ be a submodule or a quotient of $(\tau_{\eta,\nu})_{\HC}$.
Then a linear map $T:\mathcal{U} \rightarrow \mathcal{U}'$ is intertwining for $(\pi_{\xi,\lambda})_{\HC}$ and $(\tau_{\eta,\nu})_{\HC}$ if and only if
\begin{align*}
T|_{\mathcal{U}(\alpha,\alpha')}= \begin{cases}
t_{\alpha,\alpha'} \cdot R_{\alpha,\alpha'} & \text{if } \mathcal{U}(\alpha,\alpha') \neq 0 \text{ and } \mathcal{U}'(\alpha')\neq 0, \\
0 & \text{otherwise.}
\end{cases}
\end{align*}
with numbers $t_{\alpha,\alpha'} \in \C$ satisfying
\begin{multline}
\label{EQ:omega_formula}
\; \; \sum_{\mathclap{\substack{\beta\\(\alpha,\alpha')\leftrightarrow(\beta,\beta')\\ \mathcal{U}(\beta,\beta')\neq 0}}}\;(\sigma_\beta-\sigma_\alpha+2\lambda_\ss)t_{\beta,\beta'}\cdot \left( R_{\beta,\beta'}\circ \omega_{\alpha,\alpha'}^{\beta,\beta'}(X)\right) + 2\delta_{\alpha',\beta'}\lambda_\subz(X)t_{\alpha,\alpha'} \cdot R_{\alpha,\alpha'} \\
=\begin{cases}
( \sigma'_{\beta'}-\sigma'_{\alpha'}+2\nu_\ss)t_{\alpha,\alpha'}\left(
\omega'^{\beta'}_{\alpha'}(X)\circ R_{\alpha,\alpha'}\right)+2\delta_{\alpha',\beta'}\nu_\subz(X)t_{\alpha,\alpha'}\cdot R_{\alpha,\alpha'}
 & \text{if } \mathcal{U}(\alpha')\neq 0 ,\\
0 & \text{otherwise,}
\end{cases}
\end{multline}
for all $X\in\mathfrak{s}'_\C$ and $\mathcal{U}(\alpha,\alpha'),\mathcal{U}'(\beta') \neq 0$.
\end{theorem}
\begin{proof}
We have that \eqref{EQ:g_intertwining} holds if and only if
\begin{equation}
\label{EQ:Prelim_formula1}
\proj_{\mathcal{U}'(\beta')}\circ T\circ\proj_{\mathcal{U}}\circ \pi_{\xi,\lambda}(X)|_{\mathcal{U}(\alpha,\alpha')}=\proj_{\mathcal{U}'(\beta')}\circ\tau_{\eta,\nu}(X)\circ T|_{\mathcal{U}({\alpha,\alpha'})},
\end{equation}
for all $\mathcal{U}(\alpha,\alpha')\neq 0$ and $\mathcal{U}'(\beta') \neq 0$. Since $\lambda_\subz$ acts by scalars it leaves $K'$-types invariant. Applying \eqref{EQ:cocycle_sub} to the right-hand side of \eqref{EQ:Prelim_formula1} yields
\begin{equation*}
\frac{1}{2}( \sigma'_{\beta'}-\sigma'_{\alpha'}+2\nu_\ss)t_{\alpha,\alpha'}\left(
\omega'^{\beta'}_{\alpha'}(X)\circ R_{\alpha,\alpha'}
\right)+ \delta_{\alpha',\beta'}\nu_\subz(X)t_{\alpha,\alpha'} \cdot R_{\alpha,\alpha'} ,
\end{equation*}
where $t_{\alpha,\alpha'}=0$ if $\mathcal{U}(\alpha)=0$.
For the left-hand side we have by \eqref{EQ:cocycle_sub}
\begin{align*}
& \proj_{\mathcal{U}'(\beta')}\circ T\circ\proj_{\mathcal{U}}\circ \pi_{\xi,\lambda}(X)|_{\mathcal{U}(\alpha,\alpha')}
=
\sum_{\mathclap{\substack{\beta\\(\alpha,\alpha')\leftrightarrow(\beta,\beta')}}} \;T \circ \proj_{\mathcal{U}(\beta,\beta')} \circ \pi_{\xi,\lambda}(X)|_{\mathcal{U}(\alpha,\alpha')} \\
& \qquad=\;\frac{1}{2}\;\sum_{\mathclap{\substack{\beta\\(\alpha,\alpha')\leftrightarrow(\beta,\beta')\\ \mathcal{U}(\beta,\beta')\neq 0}}}\;(\sigma_{\beta}-\sigma_\alpha+2\lambda_\ss)\cdot \left( T \circ \omega_{\alpha,\alpha'}^{\beta,\beta'}(X) \right) +\delta_{\alpha',\beta'}\lambda_\subz(X)t_{\alpha,\alpha'} \cdot R_{\alpha,\alpha'}\\
&\qquad=\;\frac{1}{2}\sum_{\mathclap{\substack{\beta\\(\alpha,\alpha')\leftrightarrow(\beta,\beta')\\ \mathcal{U}(\beta,\beta')\neq 0}}}\;(\sigma_\beta-\sigma_\alpha+2\lambda_\ss)t_{\beta,\beta'}\cdot \left( R_{\beta,\beta'} \circ \omega_{\alpha,\alpha'}^{\beta,\beta'}(X) \right)+\delta_{\alpha',\beta'}\lambda_\subz(X)t_{\alpha,\alpha'} \cdot R_{\alpha,\alpha'}.\qedhere
\end{align*}
\end{proof}
This theorem gives a description of intertwining operators between submodules and quotients of $(\pi_{\xi,\lambda})_{\HC}$ and $(\tau_{\eta,\nu})_{\HC}$.

\subsection{Scalar identities}
We denote the $(-1)$-eigenspace of $\theta$ on $\g_\ss$ resp. $\g'_\ss$ by $\mathfrak{s}_\ss$ resp. $\mathfrak{s}'_\ss$ such that $\mathfrak{s}=\mathfrak{a}_\subz \oplus\mathfrak{s}_\ss$ resp. $\mathfrak{s}'=\mathfrak{a}'_\subz \oplus\mathfrak{s}'_\ss$ .
We introduce the notation
\begin{equation*}
\eta_{\alpha,\alpha'}^{\beta,\beta'}=R_{\beta,\beta'}\circ \omega^{\beta,\beta'}_{\alpha,\alpha'}|_{(\mathfrak{s}'_\ss)_{\C}}, \qquad \eta_{\alpha,\alpha'}^{\beta'}= \omega'^{\beta'}_{\alpha'}|_{(\mathfrak{s}'_\ss)_{\C}} \circ R_{\alpha,\alpha'}, 
\end{equation*}
\begin{equation*}
\eta_{\alpha,\alpha'}^{\beta,\subz}=R_{\beta,\alpha'}\circ \omega^{\beta,\alpha'}_{\alpha,\alpha'}|_{(\mathfrak{a}'_\subz)_{\C}},  
\end{equation*}
for the maps $(\mathfrak{s}'_\ss)_{\C}\otimes \Ical(\alpha,\alpha') \rightarrow \Jcal(\beta')$ resp. $(\mathfrak{a}'_\subz)_{\C}\otimes \Ical(\alpha,\alpha') \rightarrow \Jcal(\alpha')$.
Further we assume
\begin{align}
\label{EQ:assumpt}
\dim \Hom_{K'}((\mathfrak{s}'_\ss)_{\C}\otimes \alpha',\beta')\leq 1,
\end{align}
for all $\mathcal{U}(\alpha,\alpha'), \mathcal{U}'(\beta')\neq 0$. Then for all $(\beta,\beta')$ for which
$\eta_{\alpha,\alpha'}^{\beta,\beta'}, \eta_{\alpha,\alpha'}^{\beta'}, \neq 0$, they
must be scalar multiples of each other. We define proportionality constants $\lambda_{\alpha,\alpha'}^{\beta,\beta'}$ by
\begin{align*}
\eta_{\alpha,\alpha'}^{\beta,\beta'}=\lambda_{\alpha,\alpha'}^{\beta,\beta'}\eta_{\alpha,\alpha'}^{\beta'}.
\end{align*}
If we further assume
\begin{equation}
\label{eq:assumpt_center}
\dim(\mathfrak{a}'_\subz)\leq 1
\end{equation}
we can fix $0\neq Z\in\mathfrak{a}_\subz'$, then $\mathfrak{a}_\subz'=\R Z$ and there is a unique character $\nu_\subz^\mathbf{1} \in (\mathfrak{a}'_{\C})^*$ with $\nu_\subz^\mathbf{1}(Z)=1$ that vanishes on $\mathfrak{a}'_\ss$.
Again by \eqref{EQ:assumpt}, the maps $\eta_{\alpha,\alpha'}^{\beta,\subz}$ and $\nu_\subz^\mathbf{1} \otimes R_{\alpha,\alpha'}$ are scalar multiples of each other, whenever they are non-zero. This defines proportionality constants $\lambda_{\alpha,\alpha'}^{\beta,\subz}$ by
\begin{equation}
\eta_{\alpha,\alpha'}^{\beta,\subz}=\lambda_{\alpha,\alpha'}^{\beta,\subz}\nu_\subz^\mathbf{1} \otimes R_{\alpha,\alpha'}.
\end{equation}
Since $\mathfrak{s}'_\ss \subseteq \mathfrak{s}_\ss$ the restriction $\lambda_\subz|_{\mathfrak{s}'}$ does in fact only depend on the restriction to $\mathfrak{a}'_\subz\subseteq\mathfrak{s}'$. For simplicity, we write $\lambda_\subz$ resp. $\nu_\subz$ for $\lambda_\subz(Z)$ resp. $\nu_\subz(Z)$.

Now Theorem~\ref{Thm:formula} simplifies as follows:
\begin{corollary}
\label{Cor:scalaridentity}
Let $\mathcal{U} \subseteq \Ical$ and $\mathcal{U}' \subseteq \Jcal$ be quotients or submodules and \eqref{EQ:assumpt} and 
\eqref{eq:assumpt_center} hold. A linear map $T: \mathcal{U} \rightarrow \mathcal{U}'$ is intertwining for $(\pi_{\xi,\lambda})_{\HC}$ and $(\tau_{\eta,\nu})_{\HC}$ if and only if 
\begin{align*}
T|_{\mathcal{U}(\alpha,\alpha')}= \begin{cases}
t_{\alpha,\alpha'} \cdot R_{\alpha,\alpha'} & \text{if } \mathcal{U}(\alpha,\alpha') \neq 0 \text{ and } \mathcal{U}'(\alpha')\neq 0, \\
0 & \text{otherwise,}
\end{cases}
\end{align*}
with numbers $t_{\alpha,\alpha'}$ satisfying
\begin{align}
\label{EQ:omega_formula_sc}
 &\sum_{\mathclap{\substack{\beta\\(\alpha,\alpha')\leftrightarrow(\beta,\beta')\\ \mathcal{U}(\beta,\beta')\neq 0}}}\;(\sigma_\beta-\sigma_\alpha+2\lambda_\ss)\lambda_{\alpha,\alpha'}^{\beta,\beta'}t_{\beta,\beta'}= \begin{cases}
( \sigma'_{\beta'}-\sigma'_{\alpha'}+2\nu_\ss)t_{\alpha,\alpha'} & \text{if } \mathcal{U}'(\alpha')\neq 0, \\
0 & \text{otherwise,}
\end{cases}
\end{align}
for all $\mathcal{U}(\alpha,\alpha')\neq 0$ and $\mathcal{U}'(\beta') \neq 0$ and
\begin{align}
\label{EQ:omega_formula_Z}
 &\sum_{\mathclap{\substack{\beta\\(\alpha,\alpha')\leftrightarrow(\beta,\alpha')\\ \mathcal{U}(\beta,\alpha')\neq 0}}}\;(\sigma_\beta-\sigma_\alpha+2\lambda_\ss)\lambda_{\alpha,\alpha'}^{\beta,\subz}t_{\beta,\alpha'}=
2(\nu_\subz-\lambda_\subz)t_{\alpha,\alpha'} 
\end{align}
for all $\mathcal{U}(\alpha,\alpha')\neq 0$ and $\mathcal{U}(\alpha') \neq 0$. 
\end{corollary}
It is not necessary to compute all proportionality constants by explicit calculations with $K'$-finite vectors since they always fulfil the relations stated in the following lemma,
which is a generalization of \cite[Lemma 3.7]{MolOrs} to the reductive case and to the case where $H$ and $H'$ do not coincide.
\begin{lemma}
\label{Lemma:scalaridentities}
Let $\Ical(\alpha,\alpha')\neq 0$ and $\Jcal(\beta')\neq 0$.
Assume $R_{\alpha,\alpha'}$ and $R_{\beta,\beta'}$ to be the restriction from $K$ to $K'$ for all $\beta$ with $(\alpha,\alpha')\leftrightarrow(\beta,\beta')$. Let $\Omega:=B(H',H)$ and let \eqref{EQ:assumpt} and \eqref{eq:assumpt_center} hold for all $\Ical(\alpha,\alpha'),\Jcal(\beta')\neq 0$.
Then we have

\begin{align*}
 &\sum_{\mathclap{\substack{\beta\\(\alpha,\alpha')\leftrightarrow(\beta,\beta')}}}\; \lambda_{\alpha,\alpha'}^{\beta,\beta'} = \Omega, 
 &&\sum_{\mathclap{\substack{\beta\\(\alpha,\alpha')\leftrightarrow(\beta,\beta')}}}\; (\sigma_\beta-\sigma_\alpha)\lambda_{\alpha,\alpha'}^{\beta,\beta'} = \sigma_{\beta'}'-\sigma_{\alpha'}'+2(\Omega \rho-\rho'), 
\\
 &\sum_{\mathclap{\substack{\beta\\(\alpha,\alpha')\leftrightarrow(\beta,\alpha')}}}\; \lambda_{\alpha,\alpha'}^{\beta,\subz} =0, &&\sum_{\mathclap{\substack{\beta\\(\alpha,\alpha')\leftrightarrow(\beta,\alpha')}}}\;(\sigma_\beta-\sigma_\alpha)\lambda_{\alpha,\alpha'}^{\beta,\subz}= 0,
\end{align*}
where we identify $\rho$ and $\rho'$ with the numbers $\rho(H)$ and $\rho'(H')$.
\end{lemma}
\begin{proof}
We have $\frac{1}{\Omega} B(H',H)=1$. Therefore $B'(\cdot,H')|_{\mathfrak{a}'_\ss}=\varepsilon'|_{\mathfrak{a}'_\ss}=\frac{1}{\Omega} B(\cdot,H)|_{\mathfrak{a}'_\ss}$. 
Let $\mathfrak{m}_{min}\oplus \mathfrak{a}_{\min} \oplus \mathfrak{n}_{min} \subseteq \mathfrak{m} \oplus\mathfrak{a}\oplus\mathfrak{n}$ be a minimal parabolic subalgebra. By the Iwasawa decomposition we have $\g=\mathfrak{k} \oplus  (\mathfrak{a}_{min} \cap \mathfrak{m}) \oplus (\mathfrak{n}_{min} \cap \mathfrak{m}) \oplus \mathfrak{a}_\ss\oplus \mathfrak{a}_\subz \oplus \mathfrak{n}$ where $\mathfrak{a}_\ss$ is $B$-orthogonal to all other summands. Hence for $Y\in\g$, $B(Y,H)$ is uniquely determined by the projection to $\mathfrak{a}_\ss$. By the same argument we have that for $X\in \g'$, $B'(X,H')$ is uniquely determined by the projection to $\mathfrak{a}'_\ss$. We have $\mathfrak{a}'_\ss \subseteq (\mathfrak{a}_{min} \cap \mathfrak{m}) \oplus \mathfrak{a}_\ss$ where $B(\cdot,H)$ vanishes on $(\mathfrak{a}_{min} \cap \mathfrak{m})$. Hence $B(\cdot,H)$ composed with the projection to $\mathfrak{a}_\ss$ coincides on $\g'_\ss$ with $B(\cdot,H)$ composed with the projection to $\mathfrak{a}'_\ss$ and therefore $B(\cdot,H)|_{\g'_\ss}$ is uniquely determined by the projection to $\mathfrak{a}_\ss'$.
 Hence for all $X\in \mathfrak{s}'_\ss$ and $k\in K'$ we have
\begin{align*}
\omega(X)(k)=B(\mathrm{Ad}(k^{-1})X,H)=\Omega B'(\mathrm{Ad}(k^{-1})X,H') = \Omega \omega'(X)(k).
\end{align*}
Hence for all $X\in \mathfrak{s}'_{\C}$
\begin{align*}
R_{\beta,\beta'} \circ \omega(X) =\Omega \omega'(X) \circ R_{\alpha,\alpha'}.
\end{align*}
Therefore 
\begin{align*}
\Omega \eta_{\alpha,\alpha'}^{\beta'}=\; \sum_{\mathclap{\substack{\beta\\(\alpha,\alpha')\leftrightarrow(\beta,\beta')}}}\;\eta_{\alpha,\alpha'}^{\beta,\beta'}=\;\sum_{\mathclap{\substack{\beta\\(\alpha,\alpha')\leftrightarrow(\beta,\beta')}}}\;\lambda_{\alpha,\alpha'}^{\beta,\beta'} \eta_{\alpha,\alpha'}^{\beta'}, \\
\intertext{and}
0=\sum_{\mathclap{\substack{\beta\\(\alpha,\alpha')\leftrightarrow(\beta,\alpha')}}}\;\eta_{\alpha,\alpha'}^{\beta,\subz}=\;\sum_{\mathclap{\substack{\beta\\(\alpha,\alpha')\leftrightarrow(\beta,\alpha')}}}\;\lambda_{\alpha,\alpha'}^{\beta,\subz} \nu_\subz^{\mathbf{1}}\otimes R_{\alpha,\alpha'}.
\end{align*}
For the second identity we look at the restriction operator  $\rest: \Ical \rightarrow \Jcal$ i.e the operator with $t_{\alpha,\alpha'}=1$ for all $(\alpha,\alpha')$. Let $\proj_\mathfrak{a}$ be the orthogonal projection from $\g$ to $\mathfrak{a}$. The restriction operator is $(\pi_{\xi,\lambda},\tau_{\eta,\nu})$-intertwining if
\begin{equation*}
(\lambda_\ss\varepsilon+\rho)\circ\proj_{\mathfrak{a}_\ss}+\lambda_\subz\circ\proj_{\mathfrak{a}_\subz}=\nu_\ss\varepsilon'+\rho'+\nu_\subz,
\end{equation*}
which is the case if and only if $\lambda_\subz|_{\mathfrak{a}'_\subz}=\nu_\subz$ and
\begin{equation*}
(\lambda_\ss+\rho)B(H',H)=(\nu_\ss+\rho') \hfil \iff \hfil \Omega \lambda_\ss= \nu_\ss-(\Omega \rho -\rho').
\end{equation*}
This together with \eqref{EQ:omega_formula_sc} resp. \eqref{EQ:omega_formula_Z} for $\mathcal{U}=\Ical$, $\mathcal{U}'=\Jcal$ and together with identities already proven implies the missing identities.
\end{proof}
\begin{remark}
In the same way as for the restriction above, the knowledge of every additional intertwiner yields a new formula for the proportionality constants.
\end{remark}

\section{Algebraic symmetry breaking for the general linear group}\label{sec:AlgSBOsGL}
\label{sec:concrete_algebraic}
We return to the setting of $(G,G')=(\GL({n+1},\R),\GL({n},\R))$ with the notation as in Section~\ref{sec:concrete_smooth}.

\subsection{The spectrum generating operator}
We first determine the eigenvalues $\sigma_\alpha$ and $\sigma'_{\alpha'}$ of the spectrum generating operators $\Pcal$ and $\Pcal'$ of $G$ and $G'$. Recall the positive $(\g, \mathfrak{a})$-root $\varepsilon \in \mathfrak{a}_{\C}^*$ with root space $\g_\varepsilon=\mathfrak{n}$ and the positive $(\g', \mathfrak{a}')$ root $\varepsilon'\in (\mathfrak{a}'_{\C})^*$ with root space $\g'_{\varepsilon'}=\mathfrak{n}'$. Then we have $\varepsilon(H)=\varepsilon'(H')=1$ for 
\begin{align*}
& H:=\mathrm{diag} \left(\frac{{n}}{{n+1}}, -\frac{1}{{n+1}}, \ldots, -\frac{1}{{n+1}}\right) \in \mathfrak{a}_\ss,  \\ 
& H':= \mathrm{diag}\left(\frac{{n-1}}{{n}}, -\frac{1}{{n}}, \ldots, -\frac{1}{{n}}, 0\right) \in \mathfrak{a}_\ss'.
\end{align*}
The bilinear forms $B(X, Y):= \frac{{n+1}}{{n}} \tr(XY)$ on $\g$ and $B'(X',Y'):= \frac{{n}}{{n-1}} \tr (X',Y')$ on $\g'$ are ${\rm Ad}$-invariant and satisfy $B(H,H)=B'(H',H')=1$. We further fix the central element $Z=\diag(1,\ldots,1,0)\in\mathfrak{a}_\subz'$, then $\nu_\subz^\mathbf{1}=\frac{\tr}{{n}}\in(\mathfrak{a}'_\C)^*$ satisfies $\nu_\subz^\mathbf{1}(Z)=1$.
\begin{remark}
\label{remark:scalars}
In the notation of \eqref{eq:character_actions} we have
\begin{align*}
&\lambda_1=\frac{{n+1}}{{n}}\lambda_\ss,  &&\lambda_2= \lambda_\subz-\frac{\lambda_\ss}{{n},}\\
&\nu_1=\frac{{n}}{{n-1}}\nu_\ss,  &&\nu_2= \nu_\subz-\frac{\lambda_\ss}{{n-1}}.
\end{align*}
\end{remark}
We extend $B$ and $B'$ $\C$-bilinearly to the complexifications $\g_{\C}$ and $\g'_{\C}$.
The set $ \{ X^\mathfrak{k}_{1,j}:=\sqrt{\frac{{n}}{2(n+1)}}( E_{1,j}-E_{j,1}): j=2,\ldots , {n+1}  \}$ is a basis of $(\g_\varepsilon \oplus \g_{-\varepsilon})\cap \mathfrak{k}$ with $B(X^\mathfrak{k}_{1,i},X^\mathfrak{k}_{1,j})=-\delta_{i,j}$. Hence the spectrum generating operator is defined as
\begin{equation}
\Pcal:=-\sum_{j=2}^{{n+1}}(X^\mathfrak{k}_{1,j})^2.
\end{equation}
Note that $\Pcal$ is equal to the difference of the Casimir element for $\mathfrak{k}$ and the Casimir element for $\mathfrak{k}\cap\mathfrak{m}$. Therefore, the eigenvalues $\sigma_\alpha$ of $\Pcal$, and similarly the eigenvalues $\sigma'_{\alpha'}$ of the spectrum generating operator for $\g'$, are given by a renormalization of the eigenvalues of the Laplacian on $\RP^{{n}}$ resp. $\RP^{n-1}$:
\begin{align}
\label{Lemma:eigenvalues}
&\sigma_\alpha= \frac{{n}}{{n+1}}\alpha(2\alpha+{n-1}), &&\sigma'_{\alpha'}= \frac{{n-1}}{{n}}\alpha'(2\alpha'+{n-2}).
\end{align}

\subsection{Reduction to the cocycle}
Let $X^\mathfrak{s}_{i,j}:= \frac{{n}}{2(n+1)}(E_{i,j}+E_{j,i})$ for all $i\neq j$ and $X^\mathfrak{s}_{i,i}=\frac{{n}}{{n+1}}(E_{i,i}- E_{i+1,i+1})$ for $i\leq {n}$. Then $\{ X^\mathfrak{s}_{i,j} \}$ is a basis of $\mathfrak{s}_\ss$. Let $k \in K$ be an orthogonal matrix with first column $(x_1,\ldots ,x_{n+1})^T\in S^{{n}}$. We can explicitly calculate $\omega(X^\mathfrak{s}_{i,j})(k)=B(k^TX^\mathfrak{s}_{i,j}k,H)$: 
\begin{equation}
\label{EQ:cocycle}
\omega(X^\mathfrak{s}_{i,j})(k)= 
\begin{cases}
x_ix_j, & i \neq j,\\
x_i^2-x_{i+1}^2, & i=j.
\end{cases}
\end{equation}
Therefore $\omega(\{ X^\mathfrak{s}_{i,j} \})$ forms a basis of $\Hcal^2(S^{{n}})$. By restricting the decomposition \eqref{EQ:PolDec} to $S^{{n}}$ and combining with the decomposition \eqref{EQ:K,K'rel} we can directly read off the following:
\begin{corollary}
\label{Krelations}
\begin{enumerate}[label=(\roman{*}), ref=\thetheorem(\roman*)]
\item For the $K$-types $\alpha, \beta \in \Z_{> 0}$ of $G$ we have 
\begin{equation*}
\alpha \leftrightarrow \beta \Leftrightarrow \beta \in \{ \alpha \pm 1, \alpha \}\text{.}
\end{equation*}
\item For $\alpha=0$ we have $\omega_0^0=0$ and $\omega_0^1\neq0$ since $\omega(X)\in \Ical(1)$ for all $X \in \mathfrak{s}_{\C}$.
\item For $\phi \in \Hcal^{2\alpha}(S^{{n}})$ and $i\neq j$ we have
\begin{center}
$\omega_\alpha^{\alpha\pm 1}(X_{i,j}^\mathfrak{s})(\phi)=\phi^\pm_{i,j}$ , \hfil
$\omega_\alpha^{\alpha}(X_{i,j}^\mathfrak{s})(\phi)=\phi^0_{i,j}$,
\end{center}
and for $i=j$ we have
\begin{center}
$\omega_\alpha^{\alpha\pm1}(X_{i,j}^\mathfrak{s})(\phi)=\phi^\pm_{i,i}-\phi^\pm_{i+1,i+1}$ , \hfil
$\omega_\alpha^{\alpha}(X_{i,j}^\mathfrak{s})(\phi)=\phi^0_{i,i}-\phi^0_{i+1,i+1}$,
\end{center}
with $\phi_{i,j}^\pm,\phi_{i,j}^0$ defined as in \eqref{eq:PolDec_} with $|x|=1$.
\item For the $K$-types $\alpha,\beta$ of $\Ical$, $\beta \neq 0$, and $\alpha',\beta'$ of $\Jcal$, $\beta'\neq 0$ we have for $n\geq 3$
\begin{equation*}
(\alpha,\alpha') \leftrightarrow (\beta,\beta') \Leftrightarrow \beta \in \{  \alpha \pm 1, \alpha \} \text{ and } \beta' \in \{  \alpha' \pm 1, \alpha' \} \text{ and } \alpha' \leq \alpha, \beta' \leq \beta.
\end{equation*}
and for $n=2$
\begin{equation*}
(\alpha,\alpha') \leftrightarrow (\beta,\beta') \Leftrightarrow \beta \in \{  \alpha \pm 1, \alpha \} \text{ and } \beta' \in \{  \alpha' \pm 1 \} \text{ and } \alpha' \leq \alpha, \beta' \leq \beta.
\end{equation*}
\item We have $(\alpha,0) \not\leftrightarrow (\beta,0)$ for all $n$ and all $\alpha,\beta$.
\end{enumerate}
\end{corollary}
Now that we know the eigenvalues of the spectrum generating operator and how the multiplication with the cocycle acts we can use \eqref{EQ:cocycleG} to decide for which $\lambda \in \mathfrak{a}_\C^*$ the $(\g, K)$-module $(\pi_{\lambda})_{\mathrm{HC}}$ is irreducible.
\begin{lemma}
\label{Lemma:submodules}
$(\pi_{\lambda})_{\mathrm{HC}}$ is irreducible if and only if $\lambda_1 \neq \pm(\rho_1+2i)$,  $i\in \Z_{\geq 0}$. 
If $\lambda_1 = \pm(\rho_1 + 2i)$, $\alpha \in \Z_{\geq 0}$, $(\pi_{\lambda})_{\mathrm{HC}}$ has a unique irreducible submodule and a unique irreducible quotient.
The submodule is finite dimensional if and only if $\lambda_1 = -\rho_1-2i$, $i \in \Z_{\geq 0}$. Its $K$-types are given by
\begin{equation*}
\mathcal{F}_-(i,\lambda_\subz)=\bigoplus_{\alpha \leq i}\Ical(\alpha).
\end{equation*}
For the quotient we have
\begin{equation*}
\mathcal{T}_-(i,\lambda_\subz)=\faktor{\Ical}{\mathcal{F}_-(i,\lambda_\subz)} \cong_K \bigoplus_{\alpha > i}\Ical(\alpha).
\end{equation*}
The quotient is finite dimensional if and only if $\lambda_1= \rho_1+2i$, $i \in \Z_{\geq 0}$. In this case the $K$-types of the unique submodule are given by
\begin{equation*}
\mathcal{T}_+(i,\lambda_\subz)=\bigoplus_{\alpha > i}\Ical(\alpha).
\end{equation*}
For the quotient we have
\begin{equation*}
\mathcal{F}_+(i,\lambda_\subz)=\faktor{\Ical}{\mathcal{T}_+(i,\lambda_\subz)} \cong_K \bigoplus_{\alpha \leq i}\Ical(\alpha).
\end{equation*}
\end{lemma}
\begin{proof}
By \eqref{EQ:cocycleG} we have for $X\in \mathfrak{s}_{\C}$
\begin{equation*}
\proj_{\Ical(\beta)}\circ \pi_{\lambda}(X)|_{\Ical(\alpha)}=\frac{1}{2}
(\sigma_\beta - \sigma_\alpha + 2 \lambda_\ss)\omega_{\alpha}^{\beta}(X),
\end{equation*}
where $\omega_{\alpha}^{\beta}=0$ if and only if $\beta \not\in \{ \alpha \pm 1, \alpha \}$ or if $\alpha=\beta=0$. Since $\omega$ vanishes on $\mathfrak{k}_{\C}$, by the action of $\g$ we can at most step from $\Ical(\alpha)$ to $\Ical(\alpha \pm 1)$.
For $\beta=\alpha+1$ we have that $(\sigma_\beta - \sigma_\alpha + 2 \lambda_\ss)=0 \Leftrightarrow ({n+1}+4\alpha + 2\lambda_1)=0$ and for $\beta=\alpha-1$ we have that $(\sigma_\beta - \sigma_\alpha + 2 \lambda_\ss)=0 \Leftrightarrow ({n+1}+4\alpha -4 - 2\lambda_1)=0$. So we cannot step from $\Ical(\alpha)$ to $\Ical(\alpha+1)$ if and only if $2\lambda_1=-(n+1)-4\alpha$ and we cannot step from $\Ical(\alpha+1)$ to $\Ical(\alpha)$ if and only if $2\lambda_1={n+1}+4\alpha$. In the first case we get a unique submodule of $(\pi_{\lambda})_{\mathrm{HC}}$ given by $\bigoplus_{\beta \leq \alpha}\Ical(\beta)$. In the second case we get a unique submodule $\bigoplus_{\beta > \alpha}\Ical(\beta)$.
\end{proof}
\begin{remark}
\label{remark:composition_series_globalization}
The composition factors of the smooth representations $\pi_{\lambda}$, $\tau_{\nu}$ are given by the Casselmann--Wallach globalizations of the composition factors of the underlying Harish-Chandra modules. We write $\mathcal{F}_\pm(i,\lambda_\subz)^\infty$, $\mathcal{T}_\pm(i,\lambda_\subz)^\infty$, $\mathcal{F}'_\pm(j,\nu_\subz)^\infty$ and $\mathcal{T}'_\pm(j,\nu_\subz)^\infty$ for the globalizations. In particular quotients of the smooth representations can be realized in the compact picture on the orthogonal complements of the submodules in $C^{\infty}(\RP^{{n}})$ with respect to the $L^2$-inner product.
\end{remark}
\begin{remark}
\label{remark:isomoprhies}
If $n=2$ we have the following isomorphisms of composition factors:
\begin{align*}
&\mathcal{F}'(j,\nu_\subz):=\mathcal{F}'_+(j,\nu_\subz)\cong\mathcal{F}'_-(j,\nu_\subz), 
&&\mathcal{T}'(j,\nu_\subz):=\mathcal{T}'_+(j,\nu_\subz)\cong\mathcal{T}'_-(j,\nu_\subz),
\end{align*}
and for all $n\geq 2$
\begin{align*}
&\mathcal{F}'_+(0,\nu_\subz)\cong\mathcal{F}'_-(0,\nu_\subz).
\end{align*}
\end{remark}

\subsection{Proportionality constants}
\label{sec:scalaridenteties}
The identities of Lemma~\ref{Lemma:scalaridentities} do not give enough information to calculate the  proportionality constants $\lambda_{\alpha,\alpha'}^{\beta,\beta'}$. Let $\Ical(\alpha,\alpha')\neq0$. By Corollary~\ref{Krelations} we have for each $ \beta' $ with $\alpha' \leftrightarrow \beta'$ up to three such constants. In the following we calculate one of the $\lambda_{\alpha,\alpha'}^{\beta,\beta'}$ for each $\beta'$ by explicitly decomposing $\proj_{\Ical(\beta)} \circ m(\omega(X))|_{\Ical(\alpha,\alpha')}$. Recall from \eqref{EQ:SphHarm2} and \eqref{EQ:SphHarm3} the $K'$-equivariant embedding $I_{\alpha'\to\alpha}:\Jcal(\alpha')\to\Ical(\alpha,\alpha')$. For an element $\tilde{\phi} \in \Ical(\alpha,\alpha')$ given by 
$\tilde{\phi}=I_{\alpha'\rightarrow \alpha}(\phi)|_{S^{{n}}}$ with $\phi \in \Jcal(\alpha')$ and $X\in (\mathfrak{s}'_\ss)_{\C}$, $\beta \leftrightarrow \alpha$ and $\beta > \alpha'$ we have
\begin{equation}
\label{EQ:GegCalc1}
\proj_{\Ical(\beta)}(\omega(X)\tilde{\phi})=
\omega_{\alpha,\alpha'}^{\beta,\alpha'-1}(X)(\tilde{\phi}) + \omega_{\alpha,\alpha'}^{\beta,\alpha'}(X)(\tilde{\phi}) + 
\omega_{\alpha,\alpha'}^{\beta,\alpha'+1}(X)(\tilde{\phi}).
\end{equation}
If we look at the identity \eqref{EQ:GegCalc1} for an element $X_{i,j}^\mathfrak{s}$ ,$ i,j \neq {n+1}, i\neq j$ we get
\begin{multline}
\label{EQ:GegCalc2}
\proj_{\Ical(\beta)}(x_ix_j\tilde{\phi})=
\Lambda_{\alpha,\alpha'}^{\beta,\alpha'-1}I_{\alpha'-1 \rightarrow \beta}(\phi_{i,j}^-)|_{S^{{n}}} + \Lambda_{\alpha,\alpha'}^{\beta,\alpha'}I_{\alpha' \rightarrow \beta}(\phi_{i,j}^0)|_{S^{{n}}}  \\
+\Lambda_{\alpha,\alpha'}^{\beta,\alpha'+1}I_{\alpha'+1 \rightarrow \beta}(\phi_{i,j}^+)|_{S^{{n}}},
\end{multline}
for some constants $\Lambda_{\alpha,\alpha'}^{\beta,\alpha'-1},\Lambda_{\alpha,\alpha'}^{\beta,\alpha'},\Lambda_{\alpha,\alpha'}^{\beta,\alpha'+1}$.
Note that following the proof of Lemma~\ref{Lemma:scalaridentities} we have $\omega(X)|_{K'}=\Omega \omega'(X)\, \forall\, X \in (\mathfrak{s}'_\ss)_{\C}$, with $\Omega=B(H',H)=\frac{{n}^2-1}{{n}^2}$. This gives
\begin{align*}
\eta_{\alpha,\alpha'}^{\beta'}(X_{i,j}^\mathfrak{s})(\tilde{\phi})&=\Omega^{-1}
C_{2(\alpha-\alpha')}^{\frac{{n-1}}{2}+2\alpha'}(0)\phi^{\beta'}_{i,j}, \\
\eta_{\alpha,\alpha'}^{\beta,\beta'}(X_{i,j}^\mathfrak{s})(\tilde{\phi})
&=\Lambda_{\alpha,\alpha'}^{\beta,\beta'}
C_{2(\beta-\beta')}^{\frac{{n-1}}{2}+2\beta'}(0)\phi^{\beta'}_{i,j},
\end{align*}
where $\phi_{i,j}^{\beta'} \in \{ \phi_{i,j}^\pm, \phi_{i,j}^0\}$ for the corresponding $\beta' \in \{ \alpha' \pm 1, \alpha' \}$. Therefore
\begin{equation}
\label{EQ:Lambdalambda}
\lambda_{\alpha,\alpha'}^{\beta,\beta'}= \frac{C_{2(\beta-\beta')}^{\frac{{n-1}}{2}+2\beta'}(0)}{C_{2(\alpha-\alpha')}^{\frac{{n-1}}{2}+2\alpha'}(0)} \Omega\Lambda_{\alpha,\alpha'}^{\beta,\beta'}.
\end{equation}
\begin{lemma}
\label{Lemma:GegScalars}
For $\alpha >\alpha'+1$ we have
\begin{equation*}
\lambda_{\alpha,\alpha'}^{\alpha-1,\alpha'+1}=\Omega \frac{2(\alpha-\alpha')(2(\alpha-\alpha')-2)}{({n}+4\alpha-1)({n}+4\alpha-3)}, \hspace{0.2cm}
\lambda_{\alpha,\alpha'}^{\alpha-1,\alpha'}=\Omega\frac{({n}+2(\alpha-\alpha')-2)2(\alpha-\alpha')}{({n}+4\alpha-1)({n}+4\alpha-3)},
\end{equation*}
\begin{equation*}
\lambda_{\alpha,\alpha'}^{\alpha-1,\alpha'-1}=\Omega \frac{({n}+2(\alpha+\alpha')-4)({n}+2(\alpha+\alpha')-2)}{({n}+4\alpha-1)({n}+4\alpha-3)}.
\end{equation*}
\end{lemma}
\begin{proof}
Let $\tilde{\phi} \in \Ical(\alpha,\alpha')$,
\begin{equation*}
\tilde{\phi}(x) = I_{\alpha'\to\alpha}(\phi)(x) = |x|^{2(\alpha-\alpha')}\phi(x') C_{2(\alpha-\alpha')}^{\frac{{n-1}}{2}+2\alpha'}\left(\frac{x_{n+1}}{|x|}\right)
\end{equation*}
with $\phi \in \Hcal^{2\alpha'}(\R^{{n}}), x=(x',x_{n+1}) \in \R^{n+1}$. Let further
$\beta=\alpha-1$. In this setting the identity \eqref{EQ:GegCalc2} is 
\begin{align}
\label{EQ:Scalars1}
\tilde{\phi}_{i,j}^-=
\Lambda_{\alpha,\alpha'}^{\alpha-1,\alpha'-1}I_{\alpha'-1 \rightarrow \alpha-1}(\phi_{i,j}^-) + \Lambda_{\alpha,\alpha'}^{\alpha-1,\alpha'}I_{\alpha' \rightarrow \alpha-1}(\phi_{i,j}^0) + 
\Lambda_{\alpha,\alpha'}^{\alpha-1,\alpha'+1}I_{\alpha'+1 \rightarrow \alpha-1}(\phi_{i,j}^+) 
\end{align}
So we need to calculate $\tilde{\phi}_{i,j}^-$ and decompose it into multiples of $\tilde{\phi}_{i,j}^\pm, \tilde{\phi}_{i,j}^0$. To simplify notation we set
\begin{equation}
\label{substitution}
l:=\frac{{n-1}}{2}+2\alpha'\text{, }\quad z:=\frac{x_{n+1}}{|x|} \text{, }\quad N:=2(\alpha-\alpha').
\end{equation}
With this substitution \eqref{EQ:Scalars1} reads as
\begin{multline}
\label{EQ:Scalars2}
\tilde{\phi}_{i,j}^-(x)=
\Lambda_{\alpha,\alpha'}^{\alpha-1,\alpha'-1}\phi_{i,j}^-(x') |x|^{N}C_N^{l-2}(z)+
\Lambda_{\alpha,\alpha'}^{\alpha-1,\alpha'}\phi_{i,j}^0(x')|x|^{N-2} C_{N-2}^{l}(z)  \\
+\Lambda_{\alpha,\alpha'}^{\alpha-1,\alpha'+1}\phi_{i,j}^+(x')|x|^{N-4} C_{N-4}^{l+2}(z).
\end{multline}
Then calculating $\tilde{\phi}^-_{i,j}$ with \eqref{eq:G1} and the product rule immediately yields
\begin{equation*}
\Lambda_{\alpha,\alpha'}^{\alpha-1,\alpha'+1}=\frac{({n}+4\alpha'+1)({n}+4\alpha'-1)}{({n}+4\alpha-3)({n}+4\alpha-1)}.
\end{equation*}
Applying \eqref{eq:G4} to $\tilde{\phi}^-_{i,j}-\Lambda_{\alpha,\alpha'}^{\alpha-1,\alpha'+1}\phi_{i,j}^+(x')|x|^{N-4} C_{N-4}^{l+2}(z)$ yields 
\begin{equation*}
\Lambda_{\alpha,\alpha'}^{\alpha-1,\alpha'}=-\frac{({n}+2(\alpha+\alpha')-2)({n}+2(\alpha+\alpha')-3)}{({n}+4\alpha-3)({n}+4\alpha-1)},
\end{equation*}
and successively applying \eqref{eq:G4}, \eqref{eq:G5}, \eqref{eq:G6} and \eqref{eq:G4} to
\begin{equation*}
\tilde{\phi}^-_{i,j}-\Lambda_{\alpha,\alpha'}^{\alpha-1,\alpha'+1}\phi_{i,j}^+(x')|x|^{N-4} C_{N-4}^{l+2}(z)-\Lambda_{\alpha,\alpha'}^{\alpha-1,\alpha'}\phi_{i,j}^0(x')|x|^{N-2} C_{N-2}^{l}(z)
\end{equation*}
yields
\begin{equation*}
\Lambda_{\alpha,\alpha'}^{\alpha-1,\alpha'}=-\frac{({n}+2(\alpha+\alpha')-2)({n}+2(\alpha+\alpha')-3)}{({n}+4\alpha-3)({n}+4\alpha-1)}.
\end{equation*}
Then \eqref{EQ:Lambdalambda} and \eqref{EQ:G0} yield the proportionality constants.\qedhere
\end{proof}
Now Lemma~\ref{Lemma:scalaridentities} and \eqref{Lemma:eigenvalues} yields all proportionality constanst $\lambda_{\alpha,\alpha'}^{\beta,\beta'}$.

\begin{lemma}
Whenever they are defined, the proportionality constants for $\Ical$ and $\Jcal$ are the following:
\begin{alignat*}{3}
&\boxed{
\beta'= \alpha'-1} \quad
&&\lambda_{\alpha,\alpha'}^{\alpha-1,\alpha'-1} &&=\Omega \frac{({n}+2(\alpha+\alpha')-4)({n}+2(\alpha+\alpha')-2)}{({n}+4\alpha-3)({n}+4\alpha-1)} \\
&&&\lambda_{\alpha,\alpha'}^{\alpha,\alpha'-1} &&=\Omega 
\frac{2(2(\alpha'-\alpha)-1)({n}+2(\alpha+\alpha')-2)}{({n}+4\alpha-3)({n}+4\alpha+1)} \\
&&&\lambda_{\alpha,\alpha'}^{\alpha+1,\alpha'-1} &&=\Omega
\frac{(2(\alpha-\alpha')+3)(2(\alpha-\alpha')+1)}{({n}+4\alpha-1)({n}+4\alpha+1)} \\ 
&\boxed{\beta'= \alpha'} 
&&\lambda_{\alpha,\alpha'}^{\alpha-1,\alpha'} &&=\Omega
\frac{({n}+2(\alpha+\alpha')-2)2(\alpha-\alpha')}{({n}+4\alpha-3)({n}+4\alpha-1)} \\
&&&\lambda_{\alpha,\alpha'}^{\alpha,\alpha'} &&=\Omega
\frac{({n}+4\alpha-3)({n}+4\alpha')+4(\alpha-\alpha')(2(\alpha-\alpha')-1)}{({n}+4\alpha-3)({n}+4\alpha+1)} \\
&&&\lambda_{\alpha,\alpha'}^{\alpha+1,\alpha'} &&=\Omega
\frac{({n}+2(\alpha+\alpha')-1)({n}+2(\alpha-\alpha')+2)}{({n}+4\alpha-1)({n}+4\alpha+1)} \\ 
&\boxed{\beta'= \alpha'+1} \quad
&&\lambda_{\alpha,\alpha'}^{\alpha-1,\alpha'+1} &&=\Omega
\frac{4(\alpha-\alpha')(\alpha-\alpha'-2)}{({n}+4\alpha-3)({n}+4\alpha-1)} \\ 
&&&\lambda_{\alpha,\alpha'}^{\alpha,\alpha'+1} &&=\Omega
\frac{4({n}+2(\alpha-\alpha')-1)(\alpha-\alpha')}{({n}+4\alpha-3)({n}+4\alpha+1)} \\ 
&&&\lambda_{\alpha,\alpha'}^{\alpha+1,\alpha'+1} &&=\Omega
\frac{({n}+2(\alpha+\alpha')-1)({n}+2(\alpha+\alpha')+1)}{({n}+4\alpha-1)({n}+4\alpha+1)}
\end{alignat*}
\end{lemma}

\begin{corollary}
The proportionality constants $\lambda_{\alpha,\alpha'}^{\beta,\subz}$ are given by
\begin{equation}
\lambda_{\alpha,\alpha'}^{\beta,\subz}=\left( \frac{{n}}{{n-1}}\lambda_{\alpha,\alpha'}^{\beta,\alpha'}-\delta_{\alpha,\beta}\right).
\end{equation}
\end{corollary}
\begin{proof}
Let $X=\diag(1,\ldots,1,0) \in \zfrak '$. Then $X=\frac{{n}}{{n+1}}\mathbf{1}_{n+1}+ X_\ss \in \zfrak \oplus \mathfrak{s}_\ss$ with $X_\ss=\diag(\frac{1}{{n+1}},\ldots,\frac{1}{{n+1}},-\frac{{n}}{{n+1}})$. Then $\omega(X)=\omega(X_\ss)$ which is by \eqref{EQ:cocycle} equal to $\frac{\abs{x'}^2}{{n}}-x_{n+1}^2=\frac{{n+1}}{{n}}\abs{x'}^2-1\in \Hcal^2(S^{{n}})$. Let $\tilde{\phi}=I_{\alpha'\to\alpha}(\phi)|_{S^{{n}}} \in \Ical(\alpha,\alpha')$, $\phi \in \Jcal(\alpha')$ as before. Then by \eqref{EQ:GegCalc2} we have
\begin{multline}
\label{EQ:GegCalcZ}
\proj_{\Ical(\beta)}(\abs{x'}^2\tilde{\phi})=
\Lambda_{\alpha\alpha'}^{\beta,\alpha'-1}I_{\alpha'-1, \rightarrow \beta}\left(\sum_{i=1}^{{n}}\phi_{i,i}^-\right) \\+ \Lambda_{\alpha\alpha'}^{\beta,\alpha'}I_{\alpha', \rightarrow \beta}\left(\sum_{i=1}^{{n}}\phi_{i,i}^0\right)  
+\Lambda_{\alpha\alpha'}^{\beta,\alpha'+1}I_{\alpha'+1, \rightarrow \beta}\left(\sum_{i=1}^{{n}}\phi_{i,i}^+\right),
\end{multline}
with the scalars $\Lambda_{\alpha,\alpha'}^{\beta,\beta'}$ as before. But since the first and the last summand of \eqref{EQ:GegCalcZ} vanish we have
\begin{equation*}
\proj_{\Ical(\beta)}(\abs{x'}^2\tilde{\phi})=\Lambda_{\alpha,\alpha'}^{\beta,\alpha'}I_{\alpha' \rightarrow \beta}\left(\sum_{i=1}^{{n}}\phi_{i,i}^0\right) =\Lambda_{\alpha,\alpha'}^{\beta,\alpha'}I_{\alpha' \rightarrow \beta}(\phi).
\end{equation*}
Hence
\begin{align*}
\eta_{\alpha,\alpha'}^{\beta,\subz}(X)(\tilde{\phi})&=\left(\frac
{{n+1}}{{n}}\Lambda_{\alpha,\alpha'}^{\beta,\alpha'}-\delta_{\alpha,\beta}\right)C_{2(\alpha-\alpha')}^{\frac{{n-1}}{2}+2\alpha'}(0)\phi, \\
\frac{\tr(X)}{{n}}\cdot R_{\alpha,\alpha'}(\tilde{\phi})&=C_{2(\alpha-\alpha')}^{\frac{{n-1}}{2}+2\alpha'}(0)\phi,
\end{align*}
which implies the corollary.
\end{proof}

Now Corollary~\ref{Cor:scalaridentity} gives a full characterization of intertwining operators. 
\begin{theorem}
\label{Theorem:fromulas}
An operator $T:\Ical \rightarrow \Jcal$ is intertwining for $(\pi_\lambda)_{\HC}$ and $(\tau_{\nu})_{\HC}$ if and only if
\begin{equation*}
T|_{\Ical(\alpha,\alpha')}=
t_{\alpha,\alpha'}\cdot \rest|_{\Ical(\alpha,\alpha')}
\end{equation*}
for scalars $t_{\alpha,\alpha'}$ satisfying
\begin{multline}
\tag{\textbf{R1}}
\label{EQ:formula1}
\lambda_{\alpha,\alpha'}^{\alpha+1,\alpha'-1}({n}+4\alpha+2\lambda_1+1)t_{\alpha+1,\alpha'-1} +  {\lambda}_{\alpha,\alpha'}^{\alpha,\alpha'-1}2\lambda_1 t_{\alpha,\alpha'-1}+\\
{\lambda}_{\alpha,\alpha'}^{\alpha-1,\alpha'-1}(2\lambda_1-n-4\alpha+3)t_{\alpha-1,\alpha'-1}= 
\Omega(2\nu_1-n-4\alpha'+4)t_{\alpha,\alpha'},
\end{multline}
\begin{multline}
\tag{\textbf{R2}}
\label{EQ:formula2}
{\lambda}_{\alpha,\alpha'}^{\alpha+1,\alpha'}({n}+4\alpha+2\lambda_1+1)t_{\alpha+1,\alpha'}+ {\lambda}_{\alpha,\alpha'}^{\alpha,\alpha'}2\lambda_1t_{\alpha,\alpha'}\\+ {\lambda}_{\alpha,\alpha'}^{\alpha-1,\alpha'}(2\lambda_1-n-4\alpha+3)t_{\alpha-1,\alpha'} 
=2\Omega\nu_1 t_{\alpha,\alpha'},
\end{multline}
\begin{multline}
\tag{\textbf{R3}}
\label{EQ:formula3}
{\lambda}_{\alpha,\alpha'}^{\alpha+1,\alpha'+1}({n}+4\alpha+2\lambda_1+1)t_{\alpha+1,\alpha'+1} +  {\lambda}_{\alpha,\alpha'}^{\alpha,\alpha'+1}2\lambda_1 t_{\alpha,\alpha'+1}+\\
{\lambda}_{\alpha,\alpha'}^{\alpha-1,\alpha'+1}(2\lambda_1-n-4\alpha+3)t_{\alpha-1,\alpha'+1}= \Omega
({n}+4\alpha'+2\nu_1)t_{\alpha,\alpha'},
\end{multline}
\begin{multline}
\tag{\textbf{RZ}}
\label{EQ:formulaZ}
{\lambda}_{\alpha,\alpha'}^{\alpha+1,\alpha'}({n}+4\alpha+2\lambda_1+1)t_{\alpha+1,\alpha'} +{\lambda}_{\alpha,\alpha'}^{\alpha,\alpha'}2\lambda_1t_{\alpha,\alpha'}\\+ {\lambda}_{\alpha,\alpha'}^{\alpha-1,\alpha'}(2\lambda_1-n-4\alpha+3)t_{\alpha-1,\alpha'} 
=2\Omega(\nu_\subz-\lambda_\subz+\frac{{n}}{{n+1}}\lambda_1)t_{\alpha,\alpha'}
\end{multline}
where relation \eqref{EQ:formula2} has to be satisfied only if $n\geq 3$ and $\alpha'>0$.
\end{theorem}
As illustrated in Figure~\ref{fig:ScalarRelations} we can view the relations \eqref{EQ:formula1},\eqref{EQ:formula2},\eqref{EQ:formula2} and \eqref{EQ:formulaZ} as relations between points in the $K'$-type picture which was described in Figure~\ref{fig:K'types}.
\begin{figure}[h]
\vspace{1cm}
\centering
\setlength{\unitlength}{4pt}
\begin{picture}(90,32)
\thicklines
\put(21,32){$(\alpha,\alpha')$}
\put(25,30){\circle*{1}}
\put(0,16){$(\alpha-1,\alpha'-1)$}
\put(10,20){\circle*{1}}
\put(19,16){$(\alpha,\alpha'-1)$}
\put(25,20){\circle*{1}}
\put(33,16){$(\alpha+1,\alpha'-1)$}
\put(40,20){\circle*{1}}
\put(25,29){\line(0,-1){8}}
\put(24,29){\line(-3,-2){13}}
\put(26,29){\line(3,-2){13}}
\put(11,20){\line(1,0){13}}
\put(26,20){\line(1,0){13}}

\put(39,32){$(\alpha-1,\alpha'+1)$}
\put(50,30){\circle*{1}}
\put(58,32){$(\alpha,\alpha'+1)$}
\put(65,30){\circle*{1}}
\put(72,32){$(\alpha+1,\alpha'+1)$}
\put(80,30){\circle*{1}}
\put(65,20){\circle*{1}}
\put(61,16){$(\alpha,\alpha')$}
\put(51,30){\line(1,0){13}}
\put(66,30){\line(1,0){13}}
\put(65,29){\line(0,-1){8}}
\put(51,29){\line(3,-2){13}}
\put(79,29){\line(-3,-2){13}}

\put(30,4){\circle*{1}}
\put(23,0){$(\alpha-1,\alpha')$}
\put(45,4){\circle*{1}}
\put(41,0){$(\alpha,\alpha')$}
\put(60,4){\circle*{1}}
\put(54,0){$(\alpha+1,\alpha')$}
\put(31,4){\line(1,0){13}}
\put(46,4){\line(1,0){13}}

\put(11,26){\eqref{EQ:formula1}}
\put(73,22){\eqref{EQ:formula3}}
\put(23.5,6){\eqref{EQ:formula2},\eqref{EQ:formulaZ}}

\end{picture}
\vspace{0.5cm}
\caption{The relations \eqref{EQ:formula1}, \eqref{EQ:formula2}, \eqref{EQ:formula3} and \eqref{EQ:formulaZ}.}\label{fig:ScalarRelations}
\end{figure}
Our goal will be to define the scalars $t_{\alpha,\alpha'}$ inductively using these four relations.
If $\lambda_1=-\rho_1-2i\in -\rho_1-2\Z_{\geq0}$ we cannot define $t_{i+1,\alpha'}$ in terms of $t_{\alpha,\beta'}$ with $\alpha \leq i$. We can illustrate this fact by a line between the $i$-th and $(i+1)$-th $K$-type column in our $K'$-type picture and in the pictures of our relations, indicating that we cannot step from left to right (see Figure~\ref{fig:Barriers}). Similarly we can draw a horizontal line between the $j$-th and $j+1$-th rows if $\nu_1=-\rho'_1-2j\in-\rho'_1-2\Z_{\geq 0}$ indicating that we cannot step down.
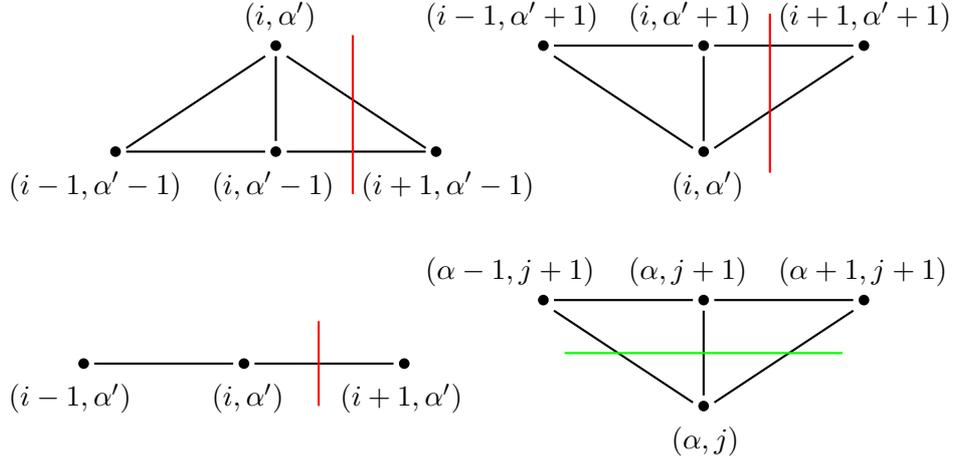
\begin{figure}[h]
\vspace{1cm}
\centering
\setlength{\unitlength}{4pt}
\begin{picture}(90,40)
\thicklines
\put(22,40){$(i,\alpha')$}
\put(25,38){\circle*{1}}
\put(0,24){$(i-1,\alpha'-1)$}
\put(10,28){\circle*{1}}
\put(19,24){$(i,\alpha'-1)$}
\put(25,28){\circle*{1}}
\put(33,24){$(i+1,\alpha'-1)$}
\put(40,28){\circle*{1}}
\put(25,37){\line(0,-1){8}}
\put(24,37){\line(-3,-2){13}}
\put(26,37){\line(3,-2){13}}
\put(11,28){\line(1,0){13}}
\put(26,28){\line(1,0){13}}

\put(39,40){$(i-1,\alpha'+1)$}
\put(50,38){\circle*{1}}
\put(58,40){$(i,\alpha'+1)$}
\put(65,38){\circle*{1}}
\put(72,40){$(i+1,\alpha'+1)$}
\put(80,38){\circle*{1}}
\put(65,28){\circle*{1}}
\put(62,24){$(i,\alpha')$}
\put(51,38){\line(1,0){13}}
\put(66,38){\line(1,0){13}}
\put(65,37){\line(0,-1){8}}
\put(51,37){\line(3,-2){13}}
\put(79,37){\line(-3,-2){13}}

\put(7,8){\circle*{1}}
\put(0,4){$(i-1,\alpha')$}
\put(22,8){\circle*{1}}
\put(19,4){$(i,\alpha')$}
\put(37,8){\circle*{1}}
\put(31,4){$(i+1,\alpha')$}
\put(8,8){\line(1,0){13}}
\put(23,8){\line(1,0){13}}

{\color{red}
\put(32.2,39){\line(0,-1){15}}
\put(71.2,41){\line(0,-1){15}}
\put(29,12){\line(0,-1){8}}}

\put(39,16){$(\alpha-1,j+1)$}
\put(50,14){\circle*{1}}
\put(58,16){$(\alpha,j+1)$}
\put(65,14){\circle*{1}}
\put(72,16){$(\alpha+1,j+1)$}
\put(80,14){\circle*{1}}
\put(65,4){\circle*{1}}
\put(62,00){$(\alpha,j)$}
\put(51,14){\line(1,0){13}}
\put(66,14){\line(1,0){13}}
\put(65,13){\line(0,-1){8}}
\put(51,13){\line(3,-2){13}}
\put(79,13){\line(-3,-2){13}}

\color{green}{
\put(52,9){\line(1,0){26}}
}

\end{picture}
\caption{Barriers for $\lambda_1=-\rho_1-2i$, $i \in \Z_{\geq 0}$ (red) and for $\nu_1=-\rho'_1-2j'$, $j \in  \Z_{\geq 0}$ (green).}\label{fig:Barriers}
\end{figure}
\subsection{Multiplicites between Harish-Chandra modules}
We use Theorem~\ref{Theorem:fromulas} to classify symmetry breaking operators between the Harish-Chandra modules $(\pi_\lambda)_\HC$ and $(\tau_\nu)_\HC$, which concludes the proof of Theorem~\ref{theorem:D}. For the statement recall the definition of the discrete set $L\subseteq\C^2$ from \eqref{eq:DefL}.
\begin{theorem}
\label{theorem:mult_HC}
For the multiplicities between $(\pi_\lambda)_{\HC}$ and $(\tau_\nu)_{\HC}$ we have for $n\geq 3$
\begin{equation*}
m((\pi_\lambda)_{\HC},(\tau_\nu)_{\HC})=
\begin{cases}
1 & \text{if } \lambda_2+\rho_2=\nu_2+\rho'_2, (\lambda_1,\nu_1)\notin L,\\
2 & \text{if } \lambda_2+\rho_2=\nu_2+\rho'_2, (\lambda_1,\nu_1)\in L,\\
1 & \text{if } \lambda_2-\rho_2=\nu_2+\rho'_2, \nu_1=-\rho'_1,\\
0 & \text{otherwise,}
\end{cases}
\end{equation*}
and for $n=2$
\begin{equation*}
m((\pi_\lambda)_{\HC},(\tau_\nu)_{\HC})=
\begin{cases}
1 & \text{if } \lambda_2+\rho_2=\nu_2+\rho'_2, (\lambda_1,\nu_1)\notin L,\\
2 & \text{if } \lambda_2+\rho_2=\nu_2+\rho'_2, (\lambda_1,\nu_1)\in L,\\
1 & \text{if } \lambda_2-\rho_2-\nu_2-\rho'_2=\nu_1+\rho'_1, \\
0 & \text{otherwise.}
\end{cases}
\end{equation*}
\end{theorem}
First we examine the diagonal sequence $(t_{\alpha,\alpha})_\alpha$. If we take \eqref{EQ:formula3} for $\alpha = \alpha'$ we get a relation only involving $t_{\alpha,\alpha}$ and $t_{\alpha+1,\alpha+1}$. Since  $\lambda_{\alpha,\alpha}^{\alpha+1,\alpha+1}=\Omega$, this is
\begin{equation}
\label{EQ:formuladiag}
({n+1}+4\alpha+2\lambda_1)t_{\alpha+1,\alpha+1}=({n}+4\alpha+2\nu_1)t_{\alpha,\alpha}.
\end{equation}
That implies:
\begin{lemma}
\label{Lemma:diagseq}
The space of diagonal sequences $(t_{\alpha,\alpha})_\alpha$ satisfying \eqref{EQ:formuladiag} is one-dimensional if $(\lambda_1,\nu_1) \in \C \setminus L$ . Every non trivial sequence satisfies
\begin{enumerate}[label=(\roman{*}), ref=\thetheorem(\roman*)]
\item 
\label{Lemma:diagseq:i}
for $\lambda_1 \not\in -\rho_1-2\Z_{\geq 0}$ and $\nu_1 \notin -\rho'_1-2\Z_{\geq 0}$:
\begin{equation*}
t_{\alpha,\alpha} \neq 0 \quad \forall\, \alpha \in \Z_{\geq 0},
\end{equation*}
\item
\label{Lemma:diagseq:ii} for $\lambda_1 =-\rho_1-2i, i \in \Z_{\geq 0}$ and $\nu_1 \notin -\rho'_1-2\Z_{\geq 0}$:
\begin{equation*}
t_{\alpha,\alpha} = 0 \quad \forall\, \alpha \leq i \qquad \text{and} \qquad t_{\alpha,\alpha} \neq 0 \quad \forall\,\alpha >i,
\end{equation*}
\item 
\label{Lemma:diagseq:iii} for $\lambda_1 \not\in -\rho_1-2\Z_{\geq 0}$ and $\nu_1=-\rho'_1-2j, j \in\Z_{\geq 0}$:
\begin{equation*}
t_{\alpha,\alpha} \neq 0 \quad \forall\, \alpha \leq j \qquad \text{and} \qquad t_{\alpha,\alpha} = 0 \quad \forall\,\alpha >j,
\end{equation*}
\item 
\label{Lemma:diagseq:iv} for $\lambda_1 =-\rho_1-2i, i \in \Z_{\geq 0}$ and $\nu_1=-\rho'_1-2j, j \in\Z_{\geq 0}$ and $i<j$:
\begin{equation*}
t_{\alpha,\alpha} \neq 0 \quad \forall\, i< \alpha \leq j \qquad \text{and} \qquad t_{\alpha,\alpha} =0\text{ else.}
\end{equation*}
\end{enumerate} 
If $(\lambda_1,\nu_1)=(-\rho_2-2i,-\rho'_1-2j) \in L$ the space of diagonal sequences $(t_{\alpha,\alpha})_\alpha$ is two-dimensional and has a basis $\{ (t'_{\alpha,\alpha})_\alpha , (t''_{\alpha,\alpha})_\alpha\}$ satisfying:
\begin{alignat*}{2}
& t'_{\alpha,\alpha} \neq 0 \quad \forall\, \alpha \leq j, \qquad \qquad && t'_{\alpha,\alpha}=0 \quad \forall\, \alpha >  j, \\
& t''_{\alpha,\alpha} = 0 \quad \forall\, \alpha \leq i, \qquad \qquad && t''_{\alpha,\alpha}\neq 0 \quad \forall\, \alpha >  i.
\end{alignat*}
\end{lemma}
\begin{proof}[Proof of Theorem~\ref{theorem:mult_HC}]
By \eqref{eq:mult_smooth_leq_mult_alg} we just need to prove that the multiplicities between the Harish-Chandra modules are less than or equal to the ones given in Theorem~\ref{theorem:mult_HC}.
This is done the same way as in the proof of \cite[Lemma 4.5]{MolOrs} and will therefore only be sketched. Details can be found in \cite{Weiske}.
We need to distinguish between three cases which split into several subcases each.

\begin{description}
\item[Case 1]
First if we assume 
\begin{equation*}
\nu_1=c'-c+\frac{{n}}{{n+1}}\lambda_1 \qquad \iff \qquad \lambda_2+\rho_2=\nu_2+\rho'_2
\end{equation*}
the formulas \eqref{EQ:formulaZ} and \eqref{EQ:formula2} are equivalent, whence we can handle this case for all $n \geq 2$ simultaneously.
We need to distinguish between three cases to give upper bounds for the multiplicities.
\begin{description}
\item[Case 1.1] Let $\lambda_1 \not\in -\rho_1-2\Z_{\geq 0}$, then $({n+1}+4\alpha+2\lambda_1)$ never vanishes.  If $\alpha=\alpha'$ relation \eqref{EQ:formula2} resp. \eqref{EQ:formulaZ} reduces to a relation between $t_{\alpha,\alpha}$ and $t_{\alpha+1,\alpha}$ for all $\alpha \geq0$. Hence given a diagonal sequence $(t_{\alpha,\alpha})_\alpha$ we can use \eqref{EQ:formula2} and \eqref{EQ:formulaZ} to uniquely define the second diagonal series $(t_{\alpha+1,\alpha})_{\alpha\geq0}$. Now we can define all missing diagonal sequences by moving further in $\alpha$-direction using \eqref{EQ:formula2} or \eqref{EQ:formulaZ}. Hence the dimension of the space satisfying all relations \eqref{EQ:formula1}, \eqref{EQ:formula2}, \eqref{EQ:formula3} and \eqref{EQ:formulaZ} can at most be one-dimensional.

\item[Case 1.2] For all $(\lambda_1,\nu_1)\notin L$ we can expand  a given diagonal sequence to the right using \eqref{EQ:formula2} or \eqref{EQ:formulaZ}. If $\lambda_1=-\rho_1-2i$, $i\in \Z_{\geq 0}$ the scalars $t_{\alpha,\alpha'}$ with $\alpha > i$ and $\alpha' \leq i$ cannot be defined that way since $({n+1}+4i+2\lambda_1)$ vanishes. But we can define these scalars uniquely with \eqref{EQ:formula3} in terms of already defined scalars. This yields multiplicity less than or equal to one for all $(\lambda_1,\nu_1) \notin L$.

\item[Case 1.3]
Let $(\lambda_1,\nu_1)=(-\rho_1-2i,-\rho'_1-2j)\in L$, $i,j \in \Z_{\geq 0}$, $i\geq j$. For a diagonal sequence $(t_{\alpha,\alpha})_\alpha$ satisfying \eqref{EQ:formula3} we have $t_{\alpha,\alpha}=0$ for $j<\alpha\leq i$. In the same way as in the second case we can define all $t_{\alpha,\alpha'}$ with $\alpha >j$ using \eqref{EQ:formulaZ} and \eqref{EQ:formula3} and for $\alpha \leq i$ and $\alpha' \leq j$ we can define all $t_{\alpha,\alpha'}$ as in the first case using \eqref{EQ:formula2} resp. \eqref{EQ:formulaZ}, which gives the situation as in Figure~\ref{FIG:3rdcase3}.

The scalars $t_{\alpha,\alpha'}$ with  $\alpha \leq i$ and $\alpha' \leq j$ are independent of those with $\alpha >i$, $\alpha'>j$. The two blocks do only depend on the diagonal sequence which is taken form a two-dimensional space as shown in Lemma~\ref{Lemma:diagseq}. 
It remains to show that the undefined scalars $t_{\alpha,\alpha'}$ with $\alpha >i$, $\alpha' \leq j$ can be uniquely defined in terms of already defined scalars. 
Since \eqref{EQ:formula2} resp. \eqref{EQ:formulaZ} for $(\alpha,\alpha')=(i+1,j)$ and \eqref{EQ:formula1} for $(\alpha,\alpha')=(i+1,j+1)$ are linearly independent we can uniquely define $t_{i+1,j}$ and $t_{i+2,j}$ using these two relation. Then moving to the right with \eqref{EQ:formulaZ} and down with \eqref{EQ:formula3} implies that the multiplicity of sequences satisfying all four equations is at most two-dimensional.

\begin{figure}[htb]
\centering
\setlength{\unitlength}{2.6pt}
\hspace{0.8cm}
\begin{picture}(80,70)
\thicklines
\put(0,0){\vector(1,0){75}}
\put(0,0){\vector(0,1){65}}

\multiput(0,0)(10,10){7}{\circle*{1.3}}
\multiput(40,30)(10,10){4}{\circle*{1.3}}
\multiput(50,30)(10,10){3}{\circle*{1.3}}
\multiput(60,30)(10,10){2}{\circle*{1.3}}
\multiput(70,30)(10,10){1}{\circle*{1.3}}
\multiput(50,0)(10,0){3}{\circle{1.3}}
\multiput(50,10)(10,0){3}{\circle{1.3}}
\multiput(50,20)(10,0){3}{\circle{1.3}}
\multiput(10,0)(10,0){4}{\circle*{1.3}}
\multiput(20,10)(10,0){3}{\circle*{1.3}}
\multiput(30,20)(10,0){2}{\circle*{1.3}}

\put(40,20){\line(1,1){10}}
\put(50,30){\line(1,-1){9.5}}
\put(40,20.3){\line(1,0){9.5}}
\put(50.5,20.3){\line(1,0){9}}
\put(40,19.7){\line(1,0){9.5}}
\put(50.5,19.7){\line(1,0){9}}
\put(50,30){\line(0,-1){9.5}}

\multiput(31,31)(10,0){2}{$0$}
\multiput(41,41)(10,0){1}{$0$}
\put(72,1){$\alpha$}
\put(1.5,62){$\alpha'$}

\put(39.5,-4){$i$}
\put(48,-4){$i+1$}
\put(-3,19.5){$j$}
\put(-10.5,29.5){$j+1$}

\color{red}{
\put(45,-2){\line(0,1){67}}
}
\color{green}{
\put(-3,25){\line(1,0){78}}
}
\end{picture}
\vspace{0.3cm}
\caption{$K'$-types $\Ical(\alpha,\alpha')$ for $(\lambda_1,\nu_1) \in L$, $ \lambda_2+\rho_2=\nu_2+\rho'_2 $ with scalars $ t_{\alpha,\alpha'}$ already defined $\bullet$, and $ t_{\alpha,\alpha'}$ not yet defined $\circ$ with relations \eqref{EQ:formula1} and \eqref{EQ:formula2}.  }
\label{FIG:3rdcase3}
\end{figure}
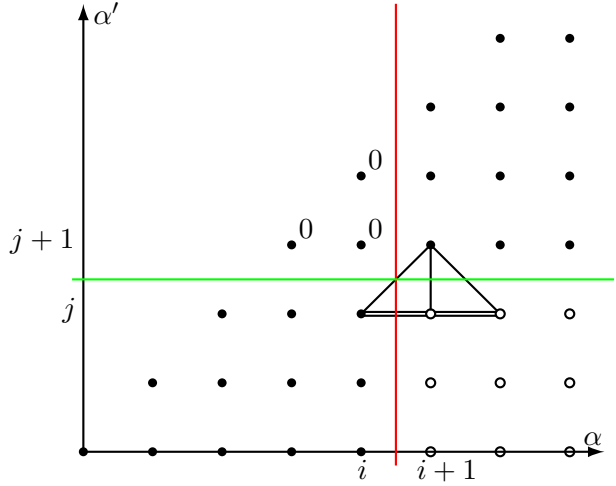
\end{description}
\item[Case 2]
Now we assume 
\begin{align*}
&\lambda_2+\rho_2\neq\nu_2+\rho'_2, && \nu_1=-\rho'_1.
\end{align*}
For $n\geq 3$, subtracting \eqref{EQ:formula2} from \eqref{EQ:formulaZ} implies that a non-trivial operator can only exist in this case if $t_{\alpha,\alpha'}=0$ for all $\alpha'\neq 0$.
By \eqref{EQ:formula3} this can only hold if $\nu_1=-\rho'_1$. Under this assumption we can again handle the cases for $n=2$ and $n \geq 3$ simultaneously since the relations we will use in the following are all defined for $n\geq 2$ and do not depend on $t_{\alpha,\alpha'}$ with $\alpha'>0$. 
\begin{description}
\item[Case 2.1]
For $\lambda_1 \neq -\rho_1-2i$, $i \in \Z_{\geq 0}$ the three relations
\eqref{EQ:formula1} for $(\alpha,\alpha')=(1,1)$ and \eqref{EQ:formulaZ} for $(\alpha,\alpha')=(0,0)$ and $(\alpha,\alpha')=(1,0)$ are a $3 \times 3$ system of linear equations in $t_{0,0}, t_{1,0}$ and $t_{2,0}$. It is easily checked that this system has rank two if and only if $\pm\nu_1=c'-c+\frac{{n}}{{n+1}}\lambda_1$ which holds if and only if $\lambda_2+\rho_2=\nu_2+\rho_2$ or $\lambda_2-\rho_2-\nu_2-\rho_2=\nu_1+\rho'_1$, and rank three otherwise.
The case $\lambda_2+\rho_2=\nu_2+\rho_2$ is already covered and if $\lambda_2-\rho_2-\nu_2-\rho_2=\nu_1+\rho'_1$ we can define all $t_{\alpha,0}$ in terms of $t_{0,0}$ using \eqref{EQ:formulaZ} and get multiplicity less than or equal to one.
\item[Case 2.2]
If $\lambda_1=-\rho_1-2i$ we can look at the two equations \eqref{EQ:formula1} for $(\alpha,\alpha')=(i,1)$ and \eqref{EQ:formulaZ} for $(\alpha,\alpha')=(i,0)$ which only involve $t_{i-1,0}$ and $t_{i,0}$. These two equations are equivalent if and only if $\lambda_2+\rho_2=\nu_2+\rho_2$ such that $t_{\alpha,0}=0$ for all $\alpha \leq i$ in all cases not already covered. Then \eqref{EQ:formula1} for $(\alpha,\alpha')=(i+1,1)$ and \eqref{EQ:formulaZ} for $(\alpha,\alpha')=(i+1,0)$ involve only $t_{i+1,0}$ and $t_{i+2,0}$ and the two equations are linearly dependent if and only if $\lambda_2-\rho_2=\nu_2+\rho'_2$. Hence in this case we can define all $t_{\alpha,0}$ with $\alpha>i$ uniquely in terms of $t_{i+1,0}$ using \eqref{EQ:formulaZ}, which yields also multiplicity less than or equal to one in this case.
\end{description}
\item[Case 3] The remaining case is
\begin{align*}
&\lambda_2+\rho_2\neq\nu_2+\rho'_2, && \nu_1\neq-\rho'_1.
\end{align*}
By the considerations of the case before we can assume $n=2$.
\begin{description}
\item[Case 3.1]
Let $(\lambda_1,\nu_1)\notin L$ and assume that $t_{\alpha,\alpha}=0$ for all $\alpha$. Then expanding to the right with \eqref{EQ:formulaZ} and down with \eqref{EQ:formula3} implies that the whole sequence $(t_{\alpha,\alpha'})_{\alpha'\leq \alpha}$ vanishes. Hence we assume that there exists a $K$-type $\alpha$ such that $t_{\alpha,\alpha} \neq 0$.
Further since $(\lambda_1,\nu_1)\notin L$ we have $\lambda_1 \notin \{ -\rho_1-2\alpha,-\rho_1-2(\alpha+1) \}$. Then the formulas 
\eqref{EQ:formula3} for $(\alpha,\alpha)$, \eqref{EQ:formulaZ} for $(\alpha,\alpha)$, \eqref{EQ:formulaZ} for $(\alpha+1,\alpha)$ and \eqref{EQ:formula1} for $(\alpha+1,\alpha+1)$ yield
a $4\times 4 $ system of linear equations in $t_{\alpha,\alpha}, t_{\alpha+1,\alpha+1}, t_{\alpha+1,\alpha}, t_{\alpha+2,\alpha}$. It is easily checked that this system has rank three if and only if 
\begin{equation*}
\pm \nu_1=c'-c+ \frac{2}{3}\lambda_1.
\end{equation*}
Then $-\nu_1=c'-c+ \frac{2}{3}\lambda_1$ holds if and only if
$\lambda_2-\rho_2-\nu_2-\rho'_2=\nu_1+\rho_1$ and in this case
the same argument as for the case $\nu_1=c'-c+ \frac{2}{3}\lambda_1$ yields multiplicity less than or equal to one for $(\lambda_1,\nu_1) \notin L$.
\item[Case 3.2]
Let $(\lambda_1,\nu_1)=(-\rho_1-2i,-\rho'_1-2j)\in L$ and $i \neq j$. As before expanding to the right with \eqref{EQ:formulaZ} implies that $t_{\alpha,\alpha'}=0$ for all $\alpha \leq i$, $\alpha' >j$. For $i \neq j$ and $j \neq 0$ the two formulas \eqref{EQ:formulaZ} for $(\alpha,\alpha')=(i,j)$ and \eqref{EQ:formula1} for $(\alpha,\alpha')=(i,j+1)$ yield a $2\times 2$ system of linear equations in $t_{i,j}$ and $t_{i-1,j}$ that has a 
non-trivial solution if and only if $\nu_1=c'-c+ \frac{2}{3}\lambda_1$. 
Hence in all other cases, including $-\nu_1=c'-c+ \frac{2}{3}\lambda_1$ we have $t_{\alpha,\alpha'}=0$ for all $\alpha \leq i$.
\item[Case 3.3]
Let $(\lambda_1,\nu_1)=(-\rho_1-2i,-\rho'_1-2j)\in L$ and $i=j$.
The relations \eqref{EQ:formula1} for $(\alpha,\alpha')=(i,i)$, \eqref{EQ:formulaZ} for $(\alpha,\alpha')=(i-1,i)$ and \eqref{EQ:formula3} for $(\alpha,\alpha')=(i-1,i-1)$ yield a $3\times 3$ system of linear equations in $t_{i,i}, t_{i-1,i-1}$ and $t_{i,i-1}$ which is solvable if and only if $\lambda_2+\rho_2=\nu_2+\rho'_2$. Hence for all cases $\lambda_2+\rho_2\neq\nu_2+\rho'_2$ and $(\lambda_1,\nu_1)=(-\rho_1-2i,-\rho'_1-2j) \in L$ we have $t_{\alpha,\alpha'}=0$ for all $\alpha \leq i$. Now assume $t_{\alpha,\alpha}=0$ for all $\alpha$. Then expanding to the right we get $t_{\alpha,\alpha'}=0$ for all $\alpha \leq i $ and all $\alpha' >j$. Then for a choice of $t_{i+1,j}$ we can expand to the right both with \eqref{EQ:formula1} and \eqref{EQ:formulaZ}. But it is easily checked that these two expansions do coincide if and only if $-\nu_1=c'-c+\frac{2}{3}\lambda_1$. Otherwise we can assume that there exists an $\alpha > i$ such that $t_{\alpha,\alpha}\neq 0$. Then in particular $t_{\alpha+1,\alpha+1}\neq 0$ and by the same argument as above we obtain that non-trivial intertwiners can only exist if $\lambda_2-\rho_2-\nu_2-\rho'_2=\nu_1+\rho'_1$. Moreover the space of intertwiners is at most one dimensional since we can uniquely define $t_{i+1,j}$ in terms of $t_{i+1,j+1}$ as in Figure~\ref{FIG:3rdcase3}.
\qedhere
\end{description}
\end{description}
\end{proof}

\subsection{Multiplicities between composition factors}
The same procedure as in the proof of Theorem~\ref{theorem:mult_HC} can be used to compute multiplicities and classify symmetry breaking operators between composition factors.
We abuse notation and write $A_{\lambda,\nu}, A^{(1)}_{\lambda,\nu}, A^{(2)}_{\lambda,\nu}, B_{\lambda,\nu}$ and $C_{\lambda,\nu}$ for the restrictions of the smooth symmetry breaking operators to $\mathcal{I}$. Then $A_{\lambda,\nu}|_{\mathcal{I}(\alpha,\alpha')}=t^A_{\alpha,\alpha'}(\lambda,\nu)\cdot R_{\alpha,\alpha'}$ for numbers $t^A_{\alpha,\alpha'}(\lambda,\nu)\in \C$. Similarly we write $t^{A,1}_{\alpha,\alpha'}(\lambda,\nu), t^{A,2}_{\alpha,\alpha'}(\lambda,\nu), t^B_{\alpha,\alpha'}(\lambda,\nu)$ and $t^C_{\alpha,\alpha'}(\lambda,\nu)$ for the numbers defining the other operators. All these sequences are explicitly given by Lemma~\ref{lemma:spectral_functions}
and they are the only solutions of the relations \eqref{EQ:formula1}, \eqref{EQ:formula2}, \eqref{EQ:formula3} and \eqref{EQ:formulaZ}.
\begin{lemma}
\label{lemma:restriction_subquotients}
\begin{enumerate}[label=(\roman{*}), ref=\thetheorem(\roman*)]
\item If $\mathcal{U}\subseteq (\pi_\lambda)_{\HC}$ is a submodule, then the restriction
\begin{align*}
\Hom_{(\g',K')}((\pi_\lambda)_{\HC}|_{(\g',K')},(\tau_\nu)_{\HC}) &\to\Hom_{(\g',K')}(\mathcal{U}|_{(\g',K')},(\tau_\nu)_{\HC}), \\
T &\mapsto T|_{\mathcal{U}}
\end{align*}
is surjective.
\item If $n\geq 3$, $\mathcal{U}' \neq \mathcal{F}'_+(0,\nu_\subz)$ is a quotient of $(\tau_\nu)_{\HC}$ and $q:(\tau_\nu)_{\HC}\twoheadrightarrow \mathcal{U}'$ is the quotient map, then the composition
\begin{align*}
\Hom_{(\g',K')}((\pi_\lambda)_{\HC}|_{(\g',K')},(\tau_\nu)_{\HC}) &\to\Hom_{(\g',K')}((\pi_\lambda)_{\HC}|_{(\g',K')},\mathcal{U}'), \\
T &\mapsto q \circ T
\end{align*}
is surjective.
\end{enumerate}
\end{lemma}
\begin{proof}
Ad (i): 
By Corollary~\ref{Cor:scalaridentity} an operator $T \in \Hom_{(\g',K')}(\mathcal{U}|_{(\g',K')},(\tau_\nu)_{\HC})$ is given by scalars $t_{\alpha,\alpha'}$ satisfying the identities of Theorem~\ref{Theorem:fromulas} for all $K$-types $\alpha$ that occur in $\mathcal{U}$. But every time $\mathcal{U}(\beta)$ is trivial, and $\mathcal{U}(\alpha)$ with $\alpha \leftrightarrow \beta$ is not, the factor $(\sigma_\beta-\sigma_\alpha+2\lambda_\ss)$ vanishes anyway, so that the relations for $\mathcal{U}$ are just a subset of the old relations, including the diagonal relation \eqref{EQ:formuladiag} and the central relation \eqref{EQ:formulaZ}. 
Hence we get the same restrictions on the parameters $\lambda_2$, $\nu_2$
and if $\lambda_1=-\rho_1-2i$ the scalars $t_{\alpha,\alpha}$ have to satisfy \eqref{EQ:formuladiag} for $(\alpha,\alpha')=(i,i)$ such that in every case the diagonal sequence $(t_{\alpha,\alpha})_{\mathcal{U}(\alpha)\neq0}$ can be expanded to a sequence as in Lemma~\ref{Lemma:diagseq}, which has a unique continuation to all $(\alpha,\alpha')$ by Theorem~\ref{theorem:mult_HC}.

Ad (ii): Let $\nu_1 \neq \rho'_1$. If $\mathcal{U}'(\alpha')$ is trivial and $\mathcal{U}'(\beta')$ with $\alpha' \leftrightarrow \beta'$ is not, then the factor $(\sigma'_{\beta'}-\sigma'_{\alpha'}+2\nu_\ss)$ vanishes anyway, so that the set of relations the scalars $t_{\alpha,\alpha'}$ defining an element $T \in \Hom_{(\g',K')}((\pi_\lambda)_{\HC}|_{(\g',K')},\mathcal{U}')$ have to satisfy is a subset of the relations of Theorem~\ref{Theorem:fromulas}.
In particular since $\mathcal{U}' \neq \mathcal{F}'_-(0,\lambda_\subz)$ there exist $K'$-types for which $\eqref{EQ:formula2}$ and $\eqref{EQ:formulaZ}$ hold. Hence we get the same restrictions on $\lambda_2$, $\nu_2$. As before,
given such an operator, we can expand the diagonal sequence to a sequence as in Lemma~\ref{Lemma:diagseq} using \eqref{EQ:formuladiag}.
\end{proof}
If $\mathcal{U}$ is a quotient of $(\pi_\lambda)_{\HC}$ we have that an intertwiner $T:\mathcal{U} \to (\tau_\nu)_{\HC}$ 
together with the quotient map $q:\mathcal{U}\to (\pi_\lambda)_{\HC}$
induces an intertwiner
\begin{equation*}
\begin{tikzcd}
(\pi_\lambda)_{\HC}  \arrow[two heads]{r}{q} &\mathcal{U} \ar{r}{T} &(\tau_\nu)_{\HC}
\end{tikzcd}
\end{equation*}
that vanishes on the complement of $\mathcal{U}$.
On the other hand if $\mathcal{U}'$ is a submodule of $(\tau_\nu)_{\HC}$, $\iota :\mathcal{U}' \hookrightarrow (\tau_\nu)_{\HC}$ the inclusion and $T:(\pi_\lambda)_{\HC} \to \mathcal{U}'$ intertwining, then
\begin{equation*}
\begin{tikzcd}
(\pi_\lambda)_{\HC}  \arrow{r}{T} &\mathcal{U}' \arrow[hook]{r}{\iota} &(\tau_\nu)_{\HC}
\end{tikzcd}
\end{equation*}
is an intertwiner whose image is contained in $\mathcal{U}'$. Recall that for $n=2$ we have $\mathcal{F}'_+(j,\nu_\subz)\cong \mathcal{F}'_-(j,\nu_\subz)$ and $\mathcal{T}'_+(j,\nu_\subz)\cong \mathcal{T}'_-(j,\nu_\subz)$, and for $n\geq 3$ we have $\mathcal{F}'_-(0,\nu_z)\cong \mathcal{F}'_+(0,\nu_z)$. This implies that every algebraic symmetry breaking operator between composition factors is given by the restriction of a smooth symmetry breaking operator. Hence the considerations above and Lemma~\ref{lemma:restriction_subquotients} reduces
the classification of smooth symmetry breaking operators between composition factors to the analysis of zero-sets of the sequences $t^S_{\alpha,\alpha'}(\lambda,\nu)$, $S=A,(A,1),(A,2),B,C$.
\begin{prop}
\label{prop:zeros}
Let $i,j\in \Z_{\geq 0}$.
Then
\begin{align*}
&(i) &&\left(t^A_{\alpha,\alpha'}(\lambda,\nu)\right)_{\alpha'>j}= 0, && \text{for $\nu_1=-\rho'_1-2j$,} \\
&(ii)&&\left(t^A_{\alpha,\alpha'}(\lambda,\nu)\right)_{\alpha\leq i}= 0, && \text{for $\lambda_1=-\rho_1-2i$,} \\
&(iii)&&\left(t^A_{\alpha,\alpha'}(\lambda,\nu)\right)_{\alpha>i,\alpha'\leq j}= 0, && \text{for $\lambda_1=\rho_1+2i$, $\nu_1=\rho'_1+2j$, $i\geq j$,} \\
&(iv)&&\left(t^{A,1}_{\alpha,\alpha'}(\lambda,\nu)\right)_{\alpha'>j}= 0, && \text{for $\nu_1=-\rho'_1-2j$,} \\
&(v)&&\left(t^{A,1}_{\alpha,\alpha'}(\lambda,\nu)-\frac{(2j)!}{(2i-2j)!}t^{A,2}_{\alpha,\alpha'}(\lambda,\nu)\right)_{\alpha\leq i} = 0, &&\text{for $\lambda_1=-\rho_1-2i$, $\nu_1=-\rho'_1-2j,i\geq j$,}\\
&(vi)&&\left(t^B_{\alpha,\alpha'}(\lambda,\nu)\right)_{\alpha\leq i}= 0, && \text{for $\lambda_1=-\rho_1-2i$,} \\
&(vii)&&\left(t^B_{\alpha,\alpha'}(\lambda,\nu)\right)_{\alpha> i}= 0, && \text{for $\lambda_1=\rho_1+2i$,} \\
& (viii) &&\left(t^C_{\alpha,\alpha'}(\lambda,\nu)\right)_{\alpha'>j}= 0, && \text{for $\nu_1=-\rho'_1-2j$,} \\
& (ix) &&\left( t^C_{\alpha,\alpha'}(\lambda,\nu)\right)_{\alpha \leq i} = 0, && \text{for } \lambda_1=-\rho_1-2i, \\
&(x) &&\left(t^C_{\alpha,\alpha'}(\lambda,\nu)\right)_{\alpha>i,\alpha'\leq j}= 0, && \text{for $\lambda_1=\rho_1+2i$, $\nu_1=-\rho'_1-2j$, $i\geq j$,}\\
&(xi) &&\left(t^C_{\alpha,\alpha'}(\lambda,\nu)\right)_{\alpha>i,\alpha'\leq j}= 0, && \text{for $\lambda_1=-\rho_1-2i$, $\nu_1=\rho'_1+2j$, $i\geq j$.}
\end{align*}
\end{prop}
\begin{proof}
Since the numbers $t^A_{\alpha,\alpha'}(\lambda,\nu),t^{A,1}_{\alpha,\alpha'}(\lambda,\nu), t^{A,2}_{\alpha,\alpha'}(\lambda,\nu), t^B_{\alpha,\alpha'}(\lambda,\nu)$ and $t^C_{\alpha,\alpha'}(\lambda,\nu)$ satisfy the relations of Theorem~\ref{Theorem:fromulas}, the diagonal sequences satisfy the conditions given in Lemma~\ref{Lemma:diagseq}. Expanding to the right with \eqref{EQ:formulaZ} implies (i), (ii), (iv) and (vi).

Ad (iii): The relations \eqref{EQ:formula2} for $(\alpha,\alpha')=(i+1,j+1)$ and \eqref{EQ:formula3} for $(\alpha,\alpha')=i+1,j$ imply $t_{i+1,j}(\lambda,\nu)=0$. Then \eqref{EQ:formula1} for $(\alpha,\alpha')=(i+1,j+1)$ implies $t_{i+2,j}(\lambda,\nu)=0$. Then the zeros expand to the whole block $\alpha >i$, $\alpha' \leq j$ by \eqref{EQ:formulaZ} and \eqref{EQ:formula3}. 

Ad (v): By Lemma~\ref{Lemma:diagseq} there exist non-trivial linear combinations of $t^{(1)}_{\alpha,\alpha'}(\lambda,\nu)$ and $t^{(2)}_{\alpha,\alpha'}(\lambda,\nu)$ that vanish on the diagonal for all $\alpha \leq i$. Such sequences vanish for $(\alpha,\alpha')$ with $\alpha \leq i$ by \eqref{EQ:formulaZ}. The right linear combinations are given by Corollary~\ref{cor:spherical_vectors}. 

Ad (vii): In \eqref{eq:spectral_function_4} the factor $(\frac{\rho_1-\lambda_1}{2})_\alpha=(-i)_\alpha$ vanishes for all $\alpha >i$.

Ad (viii) and (ix): We have for the diagonal
\begin{equation*}
t^C_{\alpha,\alpha}(\lambda,\nu)=\frac{(\frac{\nu_1+\rho'_1}{2})_{\alpha}}{\Gamma(\frac{\lambda_1+\rho_1}{2}+\alpha)},
\end{equation*} 
which vanishes for $\nu_1=-\rho'_1-2j$ for all $\alpha>j$ and for $\lambda_1=-\rho_1-2i$ for all $\alpha \leq i$. Expanding to the right with \eqref{EQ:formulaZ} yields the desired result.

Ad (x) and (xi): The relations \eqref{EQ:formulaZ} for $(\alpha,\alpha')=(i+1,j+1)$ and \eqref{EQ:formula3} for $(\alpha,\alpha')=i+1,j$ imply $t_{i+1,j}(\lambda,\nu)=0$. Then \eqref{EQ:formulaZ} implies the statement.
\end{proof}
Finally together with Remark~\ref{remark:composition_series_globalization} we get a classification of symmetry breaking operators between composition factors of $\pi_\lambda$ and $\tau_\nu$:
\begin{theorem}
\label{theorem:composition_series_tables}
Let $\mathcal{U}\in \{\mathcal{F}_\pm(i,\lambda_\subz)^\infty,\mathcal{T}_\pm(i,\lambda_\subz)^\infty \}$ and $\mathcal{U}'\in \{\mathcal{F}'_\pm(j,\nu_\subz)^\infty,\mathcal{T}'_\pm(j,\nu_\subz)^\infty \}$.
The spaces $\Hom_{G'}(\mathcal{U}|_{G'},\mathcal{U}')$ are spanned by operators given in the following tables:

\begin{center}
\renewcommand{\arraystretch}{1.5}
\begin{tabular}{ c | c  c  c  c }		
  \diagonal{0.1em}{2cm}{$\mathcal{U}$}{$\mathcal{U}'$} & $\mathcal{F}'_+(j,\nu_\subz)^\infty$  & $\mathcal{F}'_-(j,\nu_\subz)^\infty$& $\mathcal{T}'_+(j,\nu_\subz)^\infty$ & $\mathcal{T}'_-(j,\nu_\subz)^\infty$  \\ \hline
  $\mathcal{F}_+(i,\lambda_\subz)^\infty$ & $A_{\lambda,\nu}$ & $0$ & $0$ & $0$ \\ 
  $\mathcal{F}_-(i,\lambda_\subz)^\infty$ & $0$ & $A^{(1)}_{\lambda,\nu}$ & $0$ & $0$ \\
  $\mathcal{T}_+(i,\lambda_\subz)^\infty$ & $0$ & $A^{(1)}_{\lambda,\nu}$ & $A_{\lambda,\nu}$ & $0$ \\
  $\mathcal{T}_-(i,\lambda_\subz)^\infty$ & $A_{\lambda,\nu}$ & $0$ & $0$ & $A^{(1)}_{\lambda,\nu}-\frac{(2j)!}{(2i-2j)!}A^{(2)}_{\lambda,\nu}$ \\
  \hline  \multicolumn{5}{c}{for $j\leq i$ and $\lambda_2+\rho_2=\nu_2+\rho'_2$,}
\end{tabular}
\end{center}
\begin{center}
\renewcommand{\arraystretch}{1.5}
\begin{tabular}{ c | c  c     }			
  \diagonal{0.1em}{2cm}{$\mathcal{U}$}{$\mathcal{U}'$} & $\mathcal{F}'_+(j,\nu_\subz)^\infty$  & $\mathcal{F}'_-(j,\nu_\subz)^\infty$ \\ \hline
  $\mathcal{T}_+(i,\lambda_\subz)^\infty$ & $A_{\lambda,\nu}$ & $A^{(1)}_{\lambda,\nu}$  \\
  $\mathcal{T}_-(i,\lambda_\subz)^\infty$ & $A_{\lambda,\nu}$ & $A^{(1)}_{\lambda,\nu}$  \\
  \hline  \multicolumn{3}{c}{for $j>i$ and $\lambda_2+\rho_2=\nu_2+\rho'_2$,}
\end{tabular}
\end{center}
All other spaces are either trivial or given by the cases above composed with the isomorphisms from Remark~\ref{remark:isomoprhies}.
\end{theorem}

For the relation between the numbers $(\lambda_1,\lambda_2),(\nu_1,\nu_2)$ and $(\lambda_\ss,\lambda_\subz),(\nu_\ss,\nu_\subz)$ see Remark~\ref{remark:scalars}.

\appendix

\section{Homogeneous generalized functions}

Let $n\geq1$. For $\lambda\in\C$ with $\Re\lambda>-n$ the function
\begin{equation*}
u_{\lambda}(x)=\frac{\abs{x}^\lambda}{\Gamma(\frac{\lambda+n}{2})}
\end{equation*}
is locally integrable and hence defines a distribution $u_\lambda\in\mathcal{D}'(\R^n)$ which is homogeneous of degree $\lambda$.

\begin{lemma}[\cite{Gelfand_Shilov_1} Chapter I, 3.5 and 3.9]
\label{lemma:absolute_value_distr_holomorphic}
The family of distributions $u_{\lambda}\in \mathcal{D}'(\R^{n})$ extends holomorphically to an entire function in $\lambda \in \C$. For $\lambda\notin-n-2\Z_{\geq0}$ we have $\operatorname{supp}u_\lambda=\R^n$ and for $\lambda=-n-2N\in-n-2\Z_{\geq 0}$ we have
$$ u_{-n-2N}(x) = (-1)^N\frac{\pi^{\frac{n}{2}}}{2^{2N}\Gamma(\frac{n+2N}{2})}(\Delta_{\R^n}^N\delta)(x). $$
\end{lemma}
\begin{lemma}
\label{lemma:homogenous_distributions}
Let $u\in \mathcal{D}'(\R^n)$ be homogeneous of degree $\lambda$ and invariant under the action of $\Orm(n)$, then $u$ is a scalar multiple of $u_\lambda$.
\end{lemma}
Lemma~\ref{lemma:homogenous_distributions} follows directly from the classification of homogeneous distributions on $\R$ in \cite[Chapter I, 3.11]{Gelfand_Shilov_1}.
\section{Gegenbauer polynomials}
\label{Gegeq}
Following \cite[Chapter 10.9, (18)]{HigherTrans} Gegenbauer polynomials can be defined by
\begin{equation*}
C_n^{\lambda}(z)=\sum_{j=0}^{\lfloor \frac{n}{2} \rfloor}\frac{(-1)^j(\lambda)_{n-j}}{j!(n-2j)!}(2z)^{n-2j}.
\end{equation*}
If $n=2k$ is even the special value at $z=0$ is given by (see \cite[Chapter 10.9, (19)]{HigherTrans})
\begin{equation}
\label{EQ:G0}
C^\lambda_{2k}(0)=\frac{(-1)^k \Gamma(\lambda+k)}{k!\Gamma(\lambda)}.
\end{equation}
The following identities of neighboring polynomials with respect to $n$ and $\lambda$ hold:
\begin{align}
 &\frac{d}{dz}C_n^\lambda = 2\lambda C_{n-1}^{\lambda+1}\label{eq:G1}\\
 & 2\lambda C_n^{\lambda+1} - (n+2\lambda)C_n^\lambda = 2\lambda
zC_{n-1}^{\lambda+1}\label{eq:G6}\\
 & 2\lambda(1-z^2)C_{n-1}^{\lambda+1} = (n+2\lambda)zC_n^\lambda -
(n+1)C_{n+1}^\lambda\label{eq:G7}\\
& 4\lambda(\lambda+1)(1-z^2)C_{n-2}^{\lambda+2} =2\lambda(2\lambda+1)C_n^{\lambda+1} -
(2\lambda+n)(2\lambda+n+1)C_n^\lambda
\label{eq:G4} \\
& 4\lambda(\lambda+1)(1-z^2)C_{n-2}^{\lambda+2} =
2\lambda(2\lambda+1)zC_{n-1}^{\lambda+1} - n(2\lambda+n)C_n^\lambda \label{eq:G5}
\end{align}

The relation
\eqref{eq:G1} is \cite[Chapter 10.9, (23)]{HigherTrans}. 
\eqref{eq:G6} is  \cite[Chapter 10.9, (25)]{HigherTrans} and
\eqref{eq:G1}, \eqref{eq:G7} is \cite[Chapter 10.9, (35)]{HigherTrans}.
\eqref{eq:G4} is a combination of \eqref{eq:G6} and \eqref{eq:G7}. \eqref{eq:G5} is \eqref{eq:G6} applied to \eqref{eq:G4}.

\section{Spherical harmonics}
Let $\Hcal^{\alpha}(\R^n)$ be the space of homogenous harmonic polynomials of degree $\alpha$ on $\R^n$, endowed with the natural action of the orthogonal group $\Orm(n)$. In this way $\Hcal^{\alpha}(\R^n)$ is an irreducible $\Orm(n)$-representation (see \cite[Chapter IX]{stein_weiss_1971} and Remark~\ref{REM:SR}). Restricting this representation to the subgroup $\Orm(n-1)$ realized in $\Orm(n)$ in the upper left $(n-1)\times (n-1)$-block, we get the following decomposition into irreducible $\Orm(n-1)$-representations (see \cite[H5]{kobayashi_mano_2011}):
\begin{equation}
\label{EQ:SphHarm1}
\Hcal^{\alpha}(\R^n)|_{\Orm(n-1)} \cong \bigoplus_{0\leq \alpha' \leq \alpha}\Hcal^{\alpha'}(\R^{n-1})
\end{equation}
Let $x=(x',x_n) \in \R^n$ with $x'\in \R^{n-1}$ and $\phi \in \Hcal^{\alpha'}(\R^{n-1})$. For $0 \leq \alpha' \leq \alpha$ the maps
\begin{align}
& \tilde{I}_{\alpha' \rightarrow \alpha}: \Hcal^{\alpha'}(\R^{n-1}) \rightarrow \Hcal^{\alpha}(\R^n), \nonumber \\
& \tilde{I}_{\alpha' \rightarrow \alpha}\phi(x) = |x|^{\alpha-\alpha'}\phi(x')C^{\frac{n-2}{2}+\alpha'}_{\alpha-\alpha'}\left(\frac{x_n}{|x|}\right), \label{EQ:SphHarm2}
\end{align}
are isomorphisms of $\Orm(n-1)$-representations (see \cite[Fact 7.5.1]{kobayashi_mano_2011}).
For the spaces of even polynomials $\Hcal^{2\alpha}(\R^n)$ and $\Hcal^{2\alpha'}(\R^{n-1})$ we use the notation
\begin{equation}
\label{EQ:SphHarm3}
\tilde{I}_{2\alpha' \rightarrow 2\alpha} = I_{\alpha' \rightarrow \alpha}.
\end{equation}
Following \cite[IV.2]{stein_weiss_1971} every homogenous polynomial in $x \in \R^n$ can be uniquely written as the sum of harmonic homogenous polynomials times even powers of $\abs{x}$. 
\begin{lemma}
For $\phi \in \mathcal{H}^{2\alpha} (\R^n)$ we have 
\begin{equation}
\label{EQ:PolDec}
x_ix_j \phi= \phi_{i,j}^+ + |x|^2 \phi_{i,j}^0 + |x|^4 \phi_{i,j}^-
\end{equation}
with $\phi_{i,j}^\pm \in \mathcal{H}^{2(\alpha \pm 1)}(\R^n)$ and $\phi_{i,j}^0 \in \mathcal{H}^{2\alpha}(\R^n)$ given by ($\alpha \neq 0$)
\begin{align}
	\nonumber&\phi_{i,j}^-= \frac{1}{(n+4\alpha-4)(n+4\alpha-2)}\frac{\partial^2 \phi}{\partial x_i \partial x_j},\\
	\label{eq:PolDec_}&\phi_{i,j}^0= \frac{1}{n+4 \alpha}\left(x_j \frac{\partial 		\phi}{\partial x_i} + x_i \frac{\partial \phi}{\partial x_j}- \frac{2|x|^2}{n+4\alpha-4} \frac{\partial^2 \phi}{\partial x_i \partial x_j} + \delta_{i,j} \phi\right),\\
	\nonumber&\phi_{i,j}^+=x_ix_j \phi- |x|^2\phi_{i,j}^0 - |x|^4\phi_{i,j}^-,
\end{align}
\end{lemma}
It is easily checked that all three are indeed harmonic.
\bibliographystyle{amsplain}
\bibliography{bibliography}
\end{document}